\documentclass[10pt]{amsart}

\usepackage[active]{srcltx}

\usepackage{amsfonts}
\usepackage{amsmath}
\usepackage{amssymb}
\usepackage{,hyperref}
\usepackage{mathabx}
\usepackage{enumerate}
\usepackage[foot]{amsaddr}

\newtheorem{theorem}{Theorem}

\newtheorem{assumption}[theorem]{Assumption}

\newtheorem{corollary}[theorem]{Corollary}

\newtheorem{definition}[theorem]{Definition}

\newtheorem{lemma}[theorem]{Lemma}

\newtheorem{proposition}[theorem]{Proposition}
\newtheorem{remark}[theorem]{Remark}

\numberwithin{theorem}{section}
\numberwithin{equation}{section}

\usepackage{soul,xcolor}

\newcommand{\mr}{\mathbb{R}}

\newcommand{\lan}{\langle}
\newcommand{\ran}{\rangle}
\newcommand{\ranL}{\ran_{L^2}}
\newcommand{\diverg}{\mathrm{div}}
\newcommand{\tr}{\mathrm{tr}}

\newcommand{\eps}{\epsilon}

\newcommand{\omd}{\omega^{\delta}}
\newcommand{\ud}{u^{\delta}}
\newcommand{\R}{\mathbb{R}}
\newcommand{\E}{\mathbb{E}}
\newcommand{\rd}{\mathrm{d}}
\newcommand{\omdR}{\omega^{\delta,R}}

\title[Existence and uniqueness by noise for 2D Euler equations]{Existence and uniqueness by Kraichnan noise for 2D Euler equations with unbounded vorticity}

\author{Michele Coghi}
\address[Michele Coghi]{Dipartimento di Matematica, Universit\`a degli Studi di Trento, via Sommarive 14, 38123 Povo (Trento), Italy; ORCID ID: 0000-0002-4198-0856}
\email{michele.coghi@unitn.it}
\author{Mario Maurelli}
\address[Mario Maurelli]{Dipartimento di Matematica, Universit\`a di Pisa, Largo Bruno Pontecorvo 5, 56127 Pisa, Italy; ORCID ID: 0000-0002-3028-1742}
\email{mario.maurelli@unipi.it}

\begin{document}

\maketitle

\begin{abstract}
    We consider the 2D Euler equations on $\R^2$ in vorticity form, with unbounded initial vorticity, perturbed by a suitable non-smooth Kraichnan transport noise, with regularity index $\alpha\in (0,1)$.
    We show weak existence for every $\dot{H}^{-1}$ initial vorticity. Thanks to the noise, the solutions that we construct are limits in law of a regularized stochastic Euler equation and enjoy an additional $L^2([0,T];H^{-\alpha})$ regularity.
    
    For every $p>3/2$ and for certain regularity indices $\alpha \in (0,1/2)$ of the Kraichnan noise, we show also pathwise uniqueness for every $L^p$ initial vorticity. This result is not known without noise.

\end{abstract}



\setcounter{tocdepth}{1}
\tableofcontents

\section{Introduction}

In this work, we consider the stochastic 2D Euler equations in vorticity form, on $[0,T]\times\R^2$, with Kraichnan noise, namely
\begin{align}
\begin{aligned}\label{eq:main}
&\rd \omega +u\cdot\nabla\omega \rd t +\sum_k\sigma_k\cdot\nabla\omega\circ \rd W^k =0,\\
&u=K\ast\omega,
\end{aligned}
\end{align}
where $K$ is the Biot-Savart kernel, $(W^k)_{k}$ is a family of independent real Brownian motions (on a filtered probability space $(\Omega,(\mathcal{F}_t)_t,P)$) and $\circ$ denotes formally Stratonovich integration. The spatial covariance $Q(x-y)=\sum_k\sigma_k(x)\sigma_k(y)^T$ is given in Fourier modes by
\begin{align}
    \widehat{Q} (n) = \langle n \rangle^{-(2+2\alpha)} \left(I - \frac{nn^{\top}}{|n|^2}\right),\quad n\in\R^2,\label{eq:Q intro}
\end{align}
with $0<\alpha<1$ (here $\langle n \rangle= (1+|n|)^2)^{1/2}$). Our main results, roughly speaking, are weak existence for $\dot{H}^{-1}$-valued vorticity and strong uniqueness for $L^p$-valued vorticity, for suitable $p$ and $\alpha$, including $p=2$ and $0<\alpha<1/2$.

\textbf{Background.} The incompressible Euler equations describe the motion of an incompressible non-viscous fluid. In two dimensions, the classical result by Yudovich \cite{yudovich1963non} proves well-posedness among solutions with bounded vorticity; see also Yudovich \cite{yudovich1995uniqueness} for an extension of these results to a slightly larger class. 
There has been a lot of interest in the case of unbounded vorticity, partly motivated by the formation of singular structures such as vortices and vortex sheets in fluids, especially close to boundaries. DiPerna and Majda \cite{diperna1987concentrations} proved existence of $L^p$ vorticity solutions, starting from an $L^p$ vorticity initial condition. 
In the class of distribution-valued vorticity, Delort \cite{delort1991existence} has shown existence of solutions whose vorticity is in $H^{-1}$ and is a positive measure, with initial condition in the same class; see also Schochet \cite{schochet1995weak} and Poupaud \cite{poupaud2002diagonal}. 
Vecchi and Wu \cite{vecchi19931} have shown existence of solutions which are sum of an $H^{-1}$, positive measure vorticity and an $L^1$ vorticity. Note that the case of $H^{-1}$ solutions is particularly relevant, because it include vortex sheets, i.e. distributions where vorticity is concentrated on a line, and it corresponds to finite-energy solutions, i.e. when $u$ is $L^2$ in space. Solutions outside the $H^{-1}$ class are also relevant as they include the case of $N$ vortices (i.e. the vorticity is a linear combination of Dirac deltas), but we will not consider this case here.

DiPerna and Majda \cite{diperna1987oscillations} give a concept of measure-valued solutions, where existence of solutions with velocity in $L^2$ hold but with a modified nonlinear term (in the sense of Young measures). More recently and with different techniques, Wiedemann \cite{wiedemann2011existence} has shown a general existence result in $H^{-1}$ for 2D Euler equations on the torus, starting from any $H^{-1}$ vorticity (actually his result is valid for any $L^2$ initial velocity in any dimension). His result relies on convex integration methods, in particular on the work \cite{de2010admissibility} by De Lellis and Sz\'ekelyhidi; the result in \cite{wiedemann2011existence} is striking, but the solutions provided in this way are not known, to our knowledge, to be the vanishing viscosity limits of solutions to Navier-Stokes equations.

Concerning uniqueness, even among $L^p$ solutions, $p<\infty$, uniqueness is in general not known and, in some cases, false. A numerical result indicating non-uniqueness has been performed by Bressan and Shen \cite{bressan2020posteriori}. Vishik \cite{vishik2018instability} (see also \cite{albritton2021instability}) has shown non-uniqueness in the class of $L^p$ vorticity, for every $p>2$, when the Euler equations in vorticity form are perturbed by a force in $L^p$ (based on this argument, non-uniqueness of Leray-Hoft solutions to forced 3D Navier-Stokes equations has been proved by Albritton, Bru\'e and Colombo \cite{albritton2022non}). Using a convex integration scheme, Bru\'e and Colombo \cite{brue2021nonuniqueness} have shown non-uniqueness on the 2D torus in the Lorentz class $L^{1,\infty}$, without external force.

The case of unbounded vorticity solutions has been studied also in relation to energy conservation and enstrophy conservation. Beside the celebrated work \cite{constantin1994onsager} by Constantin, E and Titi, the work \cite{cheskidov2016energy} by Cheskidov, Lopes Filho, Nussenzveig Lopes and Shvydkoy has shown energy conservation for $L^p$ vorticity solutions, with $p>3/2$ and, among vanishing viscosity limits, with $p>1$. Recently De Rosa and Inversi \cite{de2023dissipation} have proved energy conservation for bounded and $BV$ velocity. For more irregular solutions, the convex integration approach has produced solutions which dissipate energy, see e.g. the previously cited work \cite{de2010admissibility}. Crippa and Spirito \cite{crippa2015renormalized} and Crippa, Nobili, Seis and Spirito \cite{crippa2017eulerian} have shown enstrophy conservation for $L^1$ vorticity solutions which are vanishing viscosity limits. 

The study of the stochastic Euler equations goes back at least to the works of Bessaih and Flandoli \cite{bessaih19992}, Mikulevicius and Valiukevicius \cite{mikulevicius2000stochastic}, Brzezniak and Peszat \cite{brzezniak2001stochastic}. One of the most complete result is that by Glatt-Holtz and Vicol \cite{glatt2014local}: they show, among other things, global well-posedness for the $2$D Euler equations with additive and multiplicative noise, see also Kim \cite{kim2009existence}. 
Motivated by the transport structure of the Euler equations in vorticity form, the so-called transport noise has been introduced, see e.g. Brzezniak, Flandoli and Maurelli \cite{Brzezniak_Flandoli_Maurelli}, Crisan and Lang \cite{lang2023well} in the 2D case and Crisan, Flandoli and Holm \cite{CriFlaHol2019} in the 3D case. In these works, the transport noise is white in time and smooth in space. The work \cite{agresti2021stochastic} by Agresti and Veraar considers Navier-Stokes equations with a large class of transport noises (though on the velocity and not on the vorticity), possibly irregular in space, including the Kraichnan noise we take here. There exist works also with noises that are not white in time, e.g. Crisan, Holm, Leahy and Nilssen \cite{crisan2022solution} for Euler equations with rough path transport noise.

One of the motivations to study Euler equations and other PDEs with noise is the emergence of regularization by noise phenomena, namely the noise restores existence or uniqueness or regularity for a possibly ill-posed deterministic PDE. This phenomenon has been proved and widely studied for finite-dimensional ODEs with a non-smooth drift, see e.g. Krylov and R\"ockner \cite{krylov2005strong}, Fedrizzi and Flandoli \cite{fedrizzi2013noise}, Catellier and Gubinelli \cite{catellier2016averaging}, and for the associated linear transport PDEs, see e.g. Flandoli, Gubinelli and Priola \cite{FlaGubPri2010}. In the nonlinear case, there exist several results, but many cases are missing yet and a complete theory is not available (and perhaps not possible). In the framework of Euler equations and other incompressible fluid dynamical models, many results have been obtained with transport noise. Flandoli, Gubinelli and Priola \cite{flandoli2011full} have shown that a transport noise with ``rich'' spatial structure avoids the collapse of $N$ vortices, with $N$ arbitrary but fixed, while collapse may happen without noise from very special initial configurations. Galeati \cite{galeati2020convergence} has considered a family of stochastic transport equations whose space covariance tends to be concentrated on the diagonal. He has shown the convergence, in a negative Sobolev space, of the corresponding solutions to a deterministic parabolic PDE, hence getting an anomalous dissipation in the limit. This phenomenon has been extended to many nonlinear transport-type equations, where the additional viscosity of the limiting parabolic equation has been used to show regularization effects. For example, Flandoli, Galeati and Luo \cite{flandoli2021scaling}, see also \cite{flandoli2021quantitative}, have applied this argument to $2$D Euler equation with $L^2$ vorticity, showing that any solution to the 2D stochastic Euler equation tends to the unique solution to the Navier-Stokes equation in the afore-mentioned parabolic limit. Flandoli, Galeati and Luo \cite{flandoli2021delayed,flandoli2021quantitative} have exploited the parabolic limit in the case of 2D and 3D Keller-Segel equations (among other models) to obtain long-time existence with high probability. Flandoli and Luo \cite{flandoli2021high} have considered the 3D Navier-Stokes, perturbed by a transport noise, though without the stretching term, and have shown no blow-up for smooth solutions with high probability (with a noise intensity depending on the initial condition). Transport noise has also been used to show mixing and enhanced dissipation, see the already mentioned paper \cite{flandoli2021quantitative} by Flandoli, Galeati and Luo and \cite{gess2021stabilization} by Gess and Yaroslavtsev.

Perhaps the closest work related to ours is the very recent paper by Galeati and Luo \cite{galeati2023weak}. They have taken a 2D log-Euler equation, namely where the Biot-Savart kernel is slightly smoothed by a logarithmic regularization, and a transport noise with rough spatial covariance $Q$, more irregular than the one we consider in this work. They have shown uniqueness in law in the case of $L^2$-vorticity (i.e. finite enstrophy) solutions. Exploiting Girsanov theorem, they have also shown uniqueness in law, in the class of $L^2$-vorticities, for the 2D hypodissipative Navier-Stokes equations with transport noise as the one we consider here.

Outside the context of transport noise, Glatt-Holtz and Vicol \cite{glatt2014local} have proved that additive noise avoids blow-up of smooth solutions to 3D Euler equations with high probability. A recent line of research has exploited multiplicative noise with superlinear growth to show no blow-up of smooth solutions to 3D Euler equations and other PDEs with probability $1$, see e.g. the works by Tang and Wang \cite{tang2022general}, Alonso-Oran, Miao and Tang \cite{alonso2022global}, Bagnara, Maurelli and Xu \cite{bagnara2023no}. In these works, blow-up is avoided with probability $1$, at the price (at least for Euler equations) of a unphysical noise depending on a positive Sobolev norm of the velocity. 

There are also cases where the noise does not regularize the equation. For example, in the context of stochastic fluid dynamics, Hofmanova, Zhu and Zhu \cite{hofmanova2023nonuniqueness} and Hofmanova, Lange and Pappalettera \cite{hofmanova2022global} have shown, by convex integration techniques, non-uniqueness in law for the stochastic 3D Navier-Stokes equations with additive or multiplicative noise and the stochastic 3D Euler equations with transport noise (on the velocity), respectively.

It is worth mentioning that, for 2D Navier-Stokes equations with additive or multiplicative noise, there are phenomena that might be viewed as regularization by noise. For example, Da Prato and Debussche \cite{da2002two} have shown pathwise uniqueness for 2D Navier-Stokes equations with white noise in time and space on the velocity, for a.e. initial condition with respect to the invariant measure; in particular, the vorticity is distribution-valued. Recently Cannizzaro and Kiedrowski \cite{cannizzaro2023stationary} have considered 2D Navier-Stokes equations with even more irregular noise, namely the derivative of a white-in-space noise on the velocity, and shown that, by a suitable rescaling, the equation still makes sense and is expected to gain an extra viscosity coefficient. In these PDEs, one may view the noise as an enemy, since it is very irregular, but it actually helps in controlling the SPDE via its invariant measure: we are not aware of analogous results for deterministic 2D Navier-Stokes equations with such distribution-valued vorticity. We also mention the existence and selection criterion of Markov solutions for 3D Navier-Stokes by Flandoli and Romito \cite {flandoli2008markov, romito2008analysis}.

We also recall that stochasticity in the initial condition, rather than in the equation, can help in existence of irregular solutions (for example of class $H^{-1-\epsilon}$ in the vorticity), see Albeverio and Cruzeiro \cite{albeverio1990global} and more recently Flandoli and Luo \cite{flandoli2019rho}, Grotto and Romito \cite{grotto2020central} and Grotto, Luongo and Romito \cite{grotto2023gibbs}.

Another relevant motivation to introduce noise is the rigorous derivation of turbulent phenomena. We do not explore much this line here, but we mention recent important results by Bedrossian, Blumenthal and Punshon-Smith \cite{bedrossian2022batchelor}, \cite{bedrossian2021almost} among many others.

\textbf{The main results.} To the best of our knowledge, existence of $H^{-1}$ solutions and uniqueness in $L^p$ for 2D Euler equations in presence of noise are unknown so far. The present work fills this gap with the so-called non-smooth Kraichnan noise: we show weak existence in $H^{-1}$ and pathwise uniqueness in $L^p$, for $p>3/2$, for 2D Euler equations \eqref{eq:main}, where we use a suitable transport noise which is white in time and coloured but rough in space. The random field
\begin{align*}
\dot{V}(t,x) = \sum_k \sigma_k(x) \dot{W}^k_t
\end{align*}
is Gaussian centered, divergence-free, with covariance matrix
\begin{align*}
E[\dot{V}(t,x)\dot{V}(s,y)] = \delta_{0}(t-s)Q(x-y),
\end{align*}
with $Q$ given by \eqref{eq:Q intro} for some $0<\alpha<1$. Such $Q$ has the form, for $x$ close to $0$,
\begin{align}
    Q(x) \approx  (c_0-\beta_L|x|^{2\alpha})\frac{xx^\top}{|x|^2} +(c_0-\beta_N|x|^{2\alpha})\left(I-\frac{xx^\top}{|x|^2}\right) \label{eq:Q intro 2}
\end{align}
for some positive constants $c_0,\beta_L,\beta_N$ with $\beta_N>\beta_L$. The vector field $V$ has been used by Kraichnan \cite{kraichnan1968small} in a toy model of passive scalar (i.e. transport equation) with turbulent transport. The Stratonovich formulation of \eqref{eq:main} is understood rigorously in the It\^o form (see Definition \ref{def:def_main} and Remark \ref{rem: ito strato}).

Our first main result is the existence of $H^{-1}$-valued solutions, which gain additional $H^{-\alpha}$ regularity:
\begin{theorem}[Informal statement, see Theorem \ref{thm: weak-existence}]
    For every $\omega_0$ in $\dot{H}^{-1}$, there exists a (probabilistically) weak, distributional solution $\omega$ to \eqref{eq:main} in $L^\infty([0,T];\dot{H}^{-1})$ $P$-a.s. and satisfying in addition $\omega\in L^2([0,T]\times\Omega;H^{-\alpha})$.
\end{theorem}
With respect to the deterministic result in \cite{wiedemann2011existence}, our result and its proof has some fundamental differences, due to the noise: our solutions exhibit an $H^{-\alpha}$ regularization; no convex integration scheme is used, but our solutions are constructed as limit in law of (a subsequence of) solutions to a family of regularized stochastic Euler equations, with no external forcing.

Our second main result is the pathwise uniqueness of $L^p$ solutions, for $p>3/2$ and suitable $\alpha$:
\begin{theorem}[Informal statement, see Theorem \ref{thm:Lp_wellposed}]
    Assume that $3/2<p<\infty$ and $\max\{0,2/p-1\}<\alpha<\min\{1-1/p,1/2\}$. For every $\omega_0$ in $L^p\cap L^1$, existence and pathwise uniqueness hold for \eqref{eq:main} in the class \begin{equation}
    \label{eq: uniqueness class}
    L^\infty([0,T]\times\Omega;L^p\cap L^1) \cap L^\infty([0,T];L^2(\Omega;\dot{H}^{-1})) \cap L^2([0,T]\times\Omega;H^{-\alpha}).
    \end{equation}
\end{theorem}
As mentioned above, this result is not known without noise.
As we note in Remark \ref{rmk:non optimality}, the restriction $\alpha<\min\{1-1/p,1/2\}$ is natural, but the restriction $p>3/2$, $\alpha>2/p-1$ is technical and we expect to reach any $p>1$ by a finer technical analysis. 

\begin{remark}
    \label{rem: deterministic forcing}
    Our method should also cover the case when a deterministic force $f\in L^1([0,T];L^p\cap L^1\cap \dot{H}^{-1})$ is added to the right hand side of equation \eqref{eq:main}, for which Vishik \cite{vishik2018instability} has proved that, without noise, multiple solutions exist. Indeed, concerning existence of a solution in the class \eqref{eq: uniqueness class}, one should repeat the proof of existence using the bound
    \begin{align*}
        |\langle f,G\ast \omega\rangle| \le C\|\omega\|_{\dot{H}^{-1}}\|f\|_{\dot{H}^{-1}}.
    \end{align*}
    The proof of uniqueness works exactly in the same way, since the force does not affect the estimates on the difference of two solutions.
\end{remark}

\textbf{The non-smooth Kraichnan model.} Before going into the idea of proof, we explain briefly a few facts on the Kraichnan noise and the associated model. The Kraichnan model was first introduced by Kraichnan himself \cite{kraichnan1968small}, as a simplified model of turbulence. In the Lagrangian form, the model considers the motion of a particle $X$ subject to the random velocity field $V$, that is
\begin{align}
    \dot{X}_t(x) = \circ \dot{V}(t,X_t(x)) = \sum_k\sigma_k(X_t(x)) \circ dW^k_t,\quad X_0(x)=x.\label{eq:SDE_Kraichnan}
\end{align}
In the Eulerian form, the (random) evolution of the initial mass $\mu_0$ follows the continuity equation
\begin{align}
    \rd \mu + \diverg (\mu \circ \rd V) =0\label{eq:CE Kraichnan}
\end{align}
Actually, the original model proposed in \cite{kraichnan1968small} refers to a general class of random vector fields $V$ whose law is Gaussian, isotropic in space (i.e. invariant under translations and rotations) and white in time. An important feature of isotropic Gaussian fields, already pointed out by Kraichnan \cite{kraichnan1968small}, see e.g. Baxendale and Harris \cite{baxendale1986isotropic} in the mathematical literature, is the possibility of explicit computations on the associated two-point motion: at least formally, for every $x,y\in \R^2$, the distance of two particles $|X_t(x)-X_t(y)|$ is a one-dimensional Markov process whose generator $\mathcal{M}_2$ can be explicitly computed in terms of the space covariance matrix $Q$. At the Eulerian level, if the initial condition $\mu_0$ is also isotropic, the correlation function $\E[\mu(t,x)\mu(t,y)]$ satisfies formally a one-dimensional linear PDE with generator $\mathcal{M}_2$ (see again Kraichnan \cite{kraichnan1968small}).

If the space covariance matrix $Q$ of the Kraichnan field $\dot{V}$ is given by \eqref{eq:Q intro 2}, and so it is irregular, a peculiar consequence of this irregularity is the splitting phenomenon, or spontaneous stochasticity, see Bernard, Gawedzki and Kupiainen \cite{bernard1998slow}: at least on a formal level, for $x=y$, $|X_t(x)-X_t(y)|$ becomes non-zero immediately after $t=0$; that is, one particle splits into many. Indeed, discarding low frequency contributions (that is, taking \eqref{eq:Q intro 2} as formally exact equality), the one-dimensional generator $\mathcal{M}_2$ gives formally that the process $|X_t(x)-X_t(y)|^{1-\alpha}$ is a Bessel process of effective dimension $2/(1-\alpha)>2$ (see Gawedzki \cite{gawedzki2008stochastic}) and hence leaves zero immediately. This splitting phenomenon seems to suggest non-uniqueness of solutions. Surprisingly, this is not true: as Le Jan and Raimond \cite{leJean_Raimond} (see also \cite{le2004flows}) have shown, there is a unique $V$-adapted solution $\mu$ (that is, adapted to the Brownian filtration) to the continuity equation \eqref{eq:CE Kraichnan} and such $(\mu_t)_t$ is associated with a flow of random kernels rather than a flow of maps, namely
\begin{align}
\mu_t(\rd y) = \int \eta_t(x,\rd y) \mu_0(\rd x),\label{eq:random_kernel}
\end{align}
for a family of random kernels $\eta:[0,T]\times\Omega\times\R^2\to \mathcal{P}(\R^2)$ (where $\mathcal{P}(\R^2)$ is the space of probability measures on $\R^2$).
Moreover (see again Le Jan and Raimond \cite{leJean_Raimond}), there is a compatible notion of $n$-point motion $(X_\cdot(x_1),\ldots X_\cdot(x_n))$, for which one can show rigorously that $|X_t(x)-X_t(y)|$ leaves $0$ immediately after $t=0$. A similar result for (at least) the two-point motion, by a zero-viscosity limit, has been obtained by E and Vanden-Eijnden \cite{weinan2001turbulent}, who have also shown a formula for the decay of order-two structure function $\E[|\mu(t,x)-\mu(t,y)|^2]$, discarding low frequencies and with isotropic initial conditions. Hence the non-smooth Kraichnan model has the peculiar features both of a splitting phenomenon, with explicit computations on the two-point motion, and a selection principle among solutions.

The spontaneous stochasticity is related to certain regularization properties of the generator $\mathcal{M}_2$ of the two-point motion. Indeed, as noted in Bernard, Gawedzki and Kupiainen \cite{bernard1998slow} and in Gawedzki \cite{gawedzki2008stochastic}, discarding low frequency contributions, for $t>0$ the law of the two-point motion distance $|X_t(x)-X_t(y)|$ has an explicit regular density, even when starting from $0$, something that is not possible in the case of smooth flows. Still concerning the generator, Eyink and Xin \cite{eyink1996existence} have given a rigorous upper bound on the two-point motion diffusion matrix, implying that $Q(0)-Q(x-y)\ge C|x-y|^{2\alpha}$ for a positive constant $C$, and have got certain lower bounds on the spectrum of the two-point motion generator. Discarding low frequencies, Hakulinen \cite{hakulinen2003passive} has rigorously shown regularization properties, as well as lower and upper bounds, on the two-point motion generator and its Green kernel and on the two-point motion semigroup.
As we will see, an anomalous regularity gain takes place also for the passive scalar $\mu$, at least in the $H^{-1}$ norm.

Due to its spontaneous stochasticity and other peculiar features, the Kraichnan model of passive scalars has been extensively studied in the physics literature; here we mention some other results and computations, without claim of completeness. At least for isotropic initial conditions, the energy spectrum satisfies a nonlocal but closed equation, e.g. Kraichnan \cite{kraichnan1968small}, see Eyink and Xin \cite{eyink2000self} and also Galeati and Luo \cite{galeati2023weak} for a rigorous statement. The $2N$-point correlation function satisfies also a closed system of equation (the equation for the $2N$-point correlation function depends only on the $(2N-2)$-point motion), see Kraichnan \cite{kraichnan1994anomalous}, Gawedzki and Kupiainen \cite{gawedzki1995anomalous}, Bernard, Gawedzki and Vergassola \cite{bernard1998slow}; in particular the latter work has provided a spectral decomposition of the $2$-point motion discaring low frequencies. Gawedzki and Kupiainen \cite{gawedzki1995anomalous} have noted at least formally that, in the case $\alpha=0$, the two-point motion generator is a Laplacian, that is two particles move independently as diffusions.
A rigorous analysis of the generator of the $N$-point motion, at least for $N\le 4$, has been performed by Eyink and Xin \cite{eyink1996existence}: as already mentioned for $N=2$, they have shown that the lowest eigenvalue of the $N$-point diffusivity matrix is bounded from below by $R^{2\alpha}$, where $R$ is the minimum of the two-point distance; they have used this fact to derive an upper bound on the $N$-point motion generator and finally they have used this bound to derive information on the (possible) steady state in presence of forcing. Eyink and Xin \cite{eyink2000self} exhibit also correlation functions with self-similar structure. Hakulinen \cite{hakulinen2003passive} has rigorously shown regularization properties for the $N$-point motion, discarding low frequencies, and an upper bound on the $N$-point correlation function in a (possible) steady state. Assuming the existence of a stationary state, discarding low frequencies and in the limit of large scale (white-in-time) forcing, an explicit expression for the $p$-structure function can be obtained, see Gawedzki and Vergassola \cite{gawedzki1995anomalous} (and the already mentioned \cite{eyink1996existence} and \cite{hakulinen2003passive}). This computation is related to one of the most studied properties, namely the anomalous scaling of the structure function, see also Kraichnan \cite{kraichnan1994anomalous}, Chertkov, Falkovich, Kolokolov and Lebedev \cite{chertkov1995normal} (and Celani and Vincenzi \cite{celani2002intermittency}). We refer to the works by Falkovich, Gawedzki and Vergassola \cite{falkovich2001particles} and by Gawedzki \cite{gawedzki2008stochastic} and the recent paper \cite{eyink2022kraichnan} by Eyink and Jafari for reviews and further references. Finally, Drivas and Eyink \cite{drivas2017lagrangian} have showed in general the link between spontaneous stochasticity and anomalous dissipation for the passive scalar field $\mu$, at least qualitatively in the $L^2$ norm of $\mu$; see also Eyink and Drivas \cite{eyink2015spontaneous} for an analogous remark (among other results) for Burgers equation and Drivas \cite{drivas2022self} for the interplay between anomalous dissipation and regularity in Euler equations and K41 theory. As pointed out at least in \cite[Section 4.2.1]{gawedzki2008stochastic}, the spontaneous stochasticity in the Kraichnan model is also related to a direct cascade for the passive scalar $\mu$ transported by $\dot{V}$.

\textbf{Idea and strategy of the proof.} The idea of the present work is strongly based on the link between this splitting phenomenon, anomalous dissipation of energy on average and regularization in negative Sobolev spaces. Exploiting the explicit generator of $|X_t(x)-X_t(y)|$ for the Kraichnan model, one can show an integrability gain in the reciprocal of $|X_t(x)-X_t(y)|$, namely, for a given measure $\mu_0$ on the Lagrangian initial conditions,
\begin{align}
\begin{aligned}\label{eq:Lagr dissipation}
    & \E\left[\iint \log|X_t(x)-X_t(y)|^{-1}\mu_0(\rd x)\mu_0(\rd y)\right]\\
    &\qquad \qquad +c\int_0^t\E\left[\iint|X_r(x)-X_r(y)|^{-2+2\alpha}\mu_0(\rd x)\mu_0(\rd y)\right] \rd r \\
    & \quad \le C\iint \log|x-y|^{-1}\mu_0(\rd x)\mu_0(\rd y).
\end{aligned}
\end{align}

Note that, if $(\mu_t)_t$ is the solution to the continuity equation \eqref{eq:CE Kraichnan}, the latter bound can be expressed in terms of the negative Sobolev norms. Indeed, for any $0<\beta< 1$,
\begin{align}
    \iint|X_t(x)-X_t(y)|^{-2+2\beta}\mu_0(\rd x)\mu_0(\rd y) = \iint |x-y|^{-2+2\beta} \mu_t(\rd x)\mu_t(\rd y) \approx \|\mu_t\|_{\dot{H}^{-\beta}}^2,\label{eq:Riesz intro}
\end{align}
and similarly for $\beta=1$ replacing $|x-y|^{-2+2\beta}$ with $\log|x-y|^{-1}$. Hence the bound \eqref{eq:Lagr dissipation} becomes
\begin{align}
    \E[\|\mu_t\|_{\dot{H}^{-1}}^2] +c\int_0^t \E[\|\mu_r\|_{H^{-\alpha}}^2] \rd r \le C\|\mu_0\|_{\dot{H}^{-1}}^2.\label{eq:energy dissipation}
\end{align}
Now the $\dot{H}^{-1}$ norm of the vorticity $\omega$ is the $L^2$ norm (the energy) of the velocity $u$ with $\operatorname{curl}(u)=\omega$, and the nonlinear term in Euler equations gives no contribution to the energy. Hence we may expect the bound \eqref{eq:energy dissipation} to hold also for $\mu=\omega$ solution to the stochastic Euler equations \eqref{eq:main}, so we should gain a control of the $H^{-\alpha}$ norm of $\omega$. This control is crucial in the compactness argument for existence, because it allows to pass to the limit in the nonlinear term.

The strategy of the proof of existence follows the above idea. We consider a family $\omd$ of solutions to regularized Euler equations. The main a priori estimate is a uniform energy estimate. We derive the equation for the energy $\|\omd\|_{\dot{H}^{-1}}^2 \approx \langle \omd, G\ast \omd \rangle_{L^2}$ of $\omd$, where $G(x)\approx \log|x|^{-1}$ is the Green kernel of the Laplacian, see Lemma \ref{lem:Hm1 norm formula} and Corollary \ref{cor:a_priori_1}:
\begin{align*}
    \rd \|\omd\|_{\dot{H}^{-1}}^2 &\approx 2\langle \nabla G\ast \omd, (K\ast\omd)\omd \rangle \rd t + \rd\,\text{martingale}\\
    &\quad + \iint \tr[(Q(0)-Q(x-y))D^2 G(x-y)] \omd(x)\omd(y) \rd x \rd y \rd t
\end{align*}
As well-known, the contribution due to the nonlinear term $\langle \nabla G\ast \omd, (K\ast\omd)\omd \rangle$ is zero. Concerning the term with $Q$, the key computation is the following one, which follows from \eqref{eq:Q intro 2}, see \eqref{eq:key_computation} and Lemma \ref{lem:main bound A}: for some $c>0$,
\begin{align*}
    \tr[(Q(0)-Q(x))D^2 G(x)] \approx  -c|x|^{-2+2\alpha}
\end{align*}
Hence, by \eqref{eq:Riesz intro}, we get
\begin{align*}
    \rd \E[\|\omd\|_{\dot{H}^{-1}}^2] \approx - c \E[\|\omd\|_{H^{-\alpha}}^2] \rd t,
\end{align*}
that is, we gain the control of the $H^{-\alpha}$ norm of $\omd$, see Proposition \ref{prop:a_priori_bd}, as desired. In view of the stochastic Aubin-Lions lemma, we then focus on a uniform estimate on the H\"older continuity in time. For this, we use the equation for $u$ and the $L^2([0,T];H^{-\alpha})$ bound on $\omd$ and show that $\omd$ is in $L^\infty([0,T];\dot{H}^{-1})$ uniformly with respect to $\delta$, see Lemma \ref{lem:Hm1_bd_sup}. Then, exploiting this bound, we show the uniform $C^\gamma([0,T];H^{-4})$ bound for $\omd$ for $0<\gamma<1/2$, see Lemma \ref{lem: C gamma of H minus 3}. Thanks to the uniform $L^2([0,T];H^{-\alpha})$ bound and the uniform $C^\gamma([0,T];H^{-4})$ bound, and the compact embedding of $H^{-\alpha}_{loc}$ into $H^{-1}_{loc}$, we can show tightness of the family $(\omd)_\delta$ in the space $L^2([0,T];\dot{H}^{-1}_{loc})$ (actually, we consider the space $\dot{H}^{-1}$ on $\R^2$ with a suitable weight), endowed with the \textit{strong} topology. The $\dot{H}^{-1}$ strong topology for the vorticity field, which is the $L^2$ strong topology for the velocity field, allows to pass to the limit in the nonlinear term, hence we can show that any limit point in $L^2([0,T];\dot{H}^{-1})$ of the family $(\omd)_\delta$ is a weak, $H^{-1}$-valued solution to the Euler equation \eqref{eq:main}, thus proving the existence result.

We expect that the a priori estimates are robust enough to work with other regularizations $\omd$, at least provided that the approximation to the initial condition does not explode too rapidly. One such possible regularization is the addition of small viscosity, i.e. to consider $\omd$ solution to the Navier-Stokes equations with $\delta$ viscosity, perturbed by (possibly regularized in space) Kraichnan noise; in this way, we could get a solution $\omega$ to \eqref{eq:main} also as a vanishing viscosity limit of Navier-Stokes equations with Kraichnan noise.

The proof of the uniqueness result among $L^p$ solutions is also based on an energy estimate, together with the conservation of $L^p$ norm. Given two solutions $\omega^1$ and $\omega^2$, the energy of the difference $\omega=\omega^1-\omega^2$ satisfies
\begin{align*}
    \rd \|\omega\|_{\dot{H}^{-1}}^2 &\approx 2\langle \nabla G\ast \omega, (K\ast \omega^1)\omega\rangle \rd t +2\langle \nabla G\ast \omega, (K\ast \omega)\omega^2\rangle \rd t
    +\rd\,\text{martingale} \\
    &\quad +\iint \tr[(Q(0)-Q(x-y))D^2G(x-y)]\omega(x)\omega(y)\rd x\rd y \rd t.
\end{align*}
The term $\langle\nabla G\ast \omega, (K\ast \omega)\omega^2\rangle$ vanishes, while the term $\langle \nabla G\ast \omega, (K\ast \omega^1)\omega\rangle$ can be bounded in terms of the $L^p$ norm of $\omega^1$, which is uniformly bounded, and the $H^{-\alpha}$ norm of $\omega$. The latter norm can be controlled by the term $\iint \tr[(Q(0)-Q(x-y))D^2G(x-y)]\omega(x)\omega(y)\rd x\rd y$ as in the existence part. We arrive at
\begin{align*}
    \rd \E[\|\omega\|_{\dot{H}^{-1}}^2] \le -c^\prime \E[\|\omega\|_{H^{-\alpha}}^2] \rd t +C\E[\|\omega\|_{\dot{H}^{-1}}^2] \rd t
\end{align*}
and we conclude by Gr\"onwall lemma that $\omega=0$.

It is worth noting that the same results and proofs hold also without the nonlinear term, that is for the linear, non-smooth Kraichnan model of passive scalars \eqref{eq:CE Kraichnan}.

We close this paragraph with some remarks on the regularization effect induced by the Kraichnan noise. The regularization effect from $H^{-1}$ (regularity of $\omega_0$) to $H^{-\alpha}$ (regularity of $\omega_t$ for a.e. $t>0$) is anomalous for transport-type PDEs. As already written, this anomalous regularization is strongly linked to the splitting phenomenon for the non-smooth Kraichnan model. To have an intuitive picture, recall that, as shown in as shown in Le Jan and Raimond \cite{leJean_Raimond}, the unique Lagrangian solution associated with the Kraichnan model \eqref{eq:CE Kraichnan} is given by a flow of (random) kernels $\eta$ instead of a flow of maps, see formula \eqref{eq:random_kernel}.
The anomalous regularization is then induced by the smoothing effect of these kernels (in formula \eqref{eq:random_kernel}, $\mu_0$ is smoothed by $\eta$).

Mind however that the regularization is only induced through a random transport term, whose vector field is Gaussian, white in time, rough but coloured (with H\"older continuous covariance) in space.

In some sense, the non-smooth Kraichnan noise has a regularization effect corresponding, at negative Sobolev scales, to a fractional Laplacian. However, we expect that such a regularization effect does not hold for any positive regularity index, and actually that a smooth initial condition loses regularity, again due to the splitting phenomenon. Indeed, considering even just the linear Kraichnan model \eqref{eq:CE Kraichnan}, if by contradiction the solution $\mu$ were regular, then, by a DiPerna-Lions argument \cite{diperna1989ordinary}, we expect that $\mu$ should be renormalizable, hence it should be associated with a flow of maps, which is not the case.

We also point out that the regularization for negative Sobolev space takes place already for the linear Kraichnan model \eqref{eq:CE Kraichnan}, and we are able to transfer this regularization to the nonlinear Euler equation essentially because the nonlinear Euler term preserves the energy.

\textbf{Organization of the paper.} The paper is organized as follows: In Section \ref{sec: settings} we give the precise settings of the problem and we state the main results; Section \ref{sec:regularization} introduces a regularized model; In Section \ref{sec:a priori bd} we give the main a priori energy estimate and in Section \ref{sec: continuity bounds} we give the a priori estimate on H\"older continuity; Section \ref{sec: weak existence} and Section \ref{sec: uniqueness} are devoted to the proof of weak existence and the proof of strong uniqueness respectively. Some technical and preliminary results are given in the Appendices.

\textbf{Notation}: We use $(\Omega,\mathcal{A},(\mathcal{F}_t)_t,P)$ for a 
filtered probability space (satisfying the standard assumption); $\mathbb{E}$ denotes the expectation. In accordance to the literature in the deterministic case, $\omega$ denotes the vorticity and not a generic element in $\Omega$.

For $x=(x^1,x^2)\in \R^2$, $x^\perp = (x^2,-x^1)$ denotes the vector obtained rotating $x$ clockwise of $90$ degrees. The open, resp. closed ball of center $x$ and radius $R$ will be denoted by $B_R(x)$, resp. $\bar{B}_R(x)$. We use the notation $L^2 := L^2(\mr^2)$, $H^s := H^s(\mr^2)$ and similarly for other function spaces on $\mr^2$. We denote the $L^2$ scalar product by
\begin{align*}
    \langle f,g\rangle_{L^2} = \int_{\R^2} f(x)g(x) \rd x
\end{align*}
and we use the same notation for $f$ and $g$ in dual Sobolev spaces, mostly $f\in H^s$ and $g\in H^{-s}$.
For a tempered distribution $u$ on $\mr^2$, $\hat{u}=\mathcal{F}u$ denotes the Fourier transform of $u$, that is
\begin{align*}
    \hat{u}(n)=\mathcal{F}u(n) = \int_{\R^2} u(x)e^{-2\pi ix\cdot n} \rd x;
\end{align*}
$n$ denotes usually the Fourier variable. For $n\in\R^2$, we use $\langle n\rangle =(1+|n|^2)^{1/2}$. We will often use the Plancherel identity for real functions $f$ and $g$ in $L^2$, and more in general for $f$ and $g$ tempered distribution with suitable integrability properties,
\begin{align*}
    \int_{\R^2} f(x)g(x) \rd x = \int_{\R^2} \hat{f}(n)\overline{\hat{g}(n)} \rd n.
\end{align*}
We will often use the identity, for $h$ and $g$ in $L^2(\R^2)$ and more in general for $h$ and $g$ with suitable integrability properties,
\begin{align*}
    \widehat{h\ast g}(n) = \hat{h}(n)\hat{g}(n).
\end{align*}
and its consequence
\begin{align*}
    \langle f, h\ast g\rangle = \int_{\R^2} \hat{f}(n)\overline{\hat{g}(n)\hat{h}(n)} \rd n,
\end{align*}
provided that $f$, $g$ and $h$ have enough integrability.
For $k$ positive integer, we write $D^k f$ for the tensor-valued $k$-th order derivative of $f$. We use the notation, for $f:[0,T]\to B$ with $B$ Banach space,
\begin{align*}
\|f\|_{C_t^\gamma(B)} = \sup_{t\in [0,T]}\|f(t)\|_B +\sup_{0\le s<t\le T}\frac{\|f(t)-f(s)\|_B}{|t-s|^\gamma}.
\end{align*}
The symbol $f \lesssim g$ is used for $f\le Cg$ for a constant $C\ge 0$; the possible dependence of $C$ on parameters will be specified when needed.

\subsection*{Acknowledgments}
This research was partially funded by Istituto Nazionale di Alta Matematica, group GNAMPA, through the project `Fluidodinamica Stocastica' - CUP\_E55F22000270001, and by the Hausdorff Research Institute for Mathematics in Bonn under the Junior Trimester Program `Randomness, PDEs and Nonlinear fluctuations'; M.M. acknowledges support from the Royal Society through the Newton International Fellowship NF170448 `Stochastic Euler equations and the Kraichnan model'. Parts of this work were undertaken when M.M. was at University of York, UK, and at Universit\`a degli Studi di Milano, Italy and M.C. was at Technische Universität Berlin, Germany.

We thank Franco Flandoli for the discussions at an early stage of this project, which were essential to the development of the current paper, in particular for the suggestion to look at Euler equations with Kraichnan noise (and at the related two-point motion) in a previous research line and for the suggestion to study the splitting phenomenon in the vortex model perturbed by Kraichnan noise, and for the related discussions. We thank Francesco Grotto and Marco Romito for suggesting the regularized Green kernel we use here. We thank Zdzislaw Brzezniak, Theodore Drivas and Lucio Galeati for fruitful discussions on the topic, in particular we thank Lucio Galeati for suggesting Remark \ref{rem: deterministic forcing} and for pointing out some typos in the first version. We thank Willem van Zuijlen for the answer to some technical questions and for its notes on function spaces \cite{van2022theory}.

\subsection*{Statement of competing interests}
The authors have no competing interests to declare that are relevant to the content of this article.

\section{Setting and main results}
\label{sec: settings}

\subsection{Sobolev Spaces and Green kernel}
\label{sec: sobolev spaces}

We work with the homogeneous and inhomogeneous Sobolev spaces, see for example \cite{bahouri2011fourier}. For $n\in\R^2$, we use the notation $\langle n\rangle=(1+|n|^2)^{1/2}$.

\begin{definition}
Let $s$ be a real number.
	\begin{itemize}
		\item 
	The inhomogeneous Sobolev space $H^s(\mr^2)$ is the space of tempered distributions $u$ such that $\hat{u} \in L^1_{loc}(\mr^2)$ and	\begin{equation*}
		\|u\|_{H^s}^2 :=
		\int_{\mr^2}\langle n\rangle^{2s}|\hat{u}(n)|^2 \rd n < \infty.
	\end{equation*}
	where $\hat{u}$ is the Fourier transform of $u$.
	
	\item For $s<1$, the homogeneous Sobolev space $\dot{H}^s(\mr^2)$ is the space of tempered distributions $u$ such that $\hat{u} \in L^1_{loc}(\mr^2)$ and	
	\begin{equation*}
		\|u\|_{\dot{H}^s}^2 
		:= \int_{\mathbb{R}^2} |n|^{2s} |\hat{u}(n)|^2 \rd n <\infty.
	\end{equation*}
	\end{itemize}
\end{definition}

\begin{remark}
	If $s\leq 0$, then $\|u\|_{H^s}\leq \|u\|_{\dot{H}^s}$.
\end{remark}

For $f$ in $H^s$ and $g$ in $H^{-s}$ (or $f$ in $\dot{H}^s$ and $g$ in $\dot{H}^{-s}$), we can extend the $L^2(\R^2)$ scalar product by
\begin{align*}
\langle f,g \rangle_{L^2} = \int_{\R^2} \hat{f}(n)\overline{\hat{g}(n)} \rd n.
\end{align*}

The function $G$ is the Green kernel of $-\Delta$ on $\mr^2$, that is
\begin{equation}
    \label{eq: green kernel}
    G(x) = -\frac{1}{2\pi}\log|x|,\quad x\in\R^2.
\end{equation}
The function $K$ is the Biot-Savart kernel, which is the orthogonal derivative of the Green kernel $G$:
\begin{equation}
   \label{eq: green kernel K}
    K(x)=\nabla^\perp G(x) =  -\frac{1}{2\pi}\frac{x^{\perp}}{|x|^2},\quad x\in\R^2,
\end{equation}
where $\nabla^\perp = (\partial_{x_2},-\partial_{x_1})$ and $x^\perp = (x_2,-x_1)$. The distributional  derivative (and pointwise derivative in $x\neq 0$) of $G$ is
\begin{equation*}
    \nabla G(x) = -\frac{1}{2\pi}\frac{x}{|x|^2},
\end{equation*}
and the pointwise second derivative of $G$ in $x\neq 0$ is
\begin{equation*}
    D^2G(x) = - \frac{1}{2\pi} \frac{1}{|x|^2}(I-\frac{2 xx^{\top}}{|x|^2}).
\end{equation*}
The Fourier transform of $K$ is
\begin{align*}
    \widehat{K}(n)= (2\pi)^{-1}\frac{i n^\perp}{|n|^2},\quad n\in\R^2.
\end{align*}

We also use the Riesz potentials in $\mathbb{R}^2$, that is
\begin{equation*}
	G_{\beta}(x) := \gamma_{2\beta}^{-1} |x|^{-2+2\beta},
\end{equation*}
for $0<\beta < 1$, with $\gamma_{2\beta} = \pi 2^{2\beta} \Gamma(\beta)/\Gamma(1-\beta)$ (and $\Gamma$ is the Gamma function). We recall the following classical fact:
\begin{lemma}[Riesz potentials]
	\label{lem:sobolev riesz potential}
	Let $0<\beta<1$. Then the Fourier transform of $G_\beta$ is
    \begin{equation}
        \widehat{G}_{\beta}(n) = (2\pi)^{-2\beta}|n|^{-2\beta}.\label{eq:Riesz Fourier}
    \end{equation}
	Moreover we have, for every $f$ in $\dot{H}^{-\beta}$,
	\begin{equation}
		\|f\|_{\dot{H}^{-\beta}}^2 = (2\pi)^{2\beta}\lan f, G_{\beta}\ast f\ranL.\label{eq:Hm norm and Riesz}
	\end{equation}
\end{lemma}

\begin{proof}
    The formula \eqref{eq:Riesz Fourier} is for example in \cite[Chapter V, Section 1, Lemma 2]{stein1970singular}. The formula \eqref{eq:Hm norm and Riesz} follows classically by \eqref{eq:Riesz Fourier}: indeed if $f \in \dot{H}^{-\beta} \cap C^{\infty}_c$ then
	\begin{align*}
		(2\pi)^{2\beta}\langle f, G_{\beta}\ast f \rangle &=  \langle \hat{f}, \widehat{G_{\beta} \ast f} \rangle
		=  (2\pi)^{2\beta}\langle \hat{f}, \widehat{G_{\beta}} \hat{f} \rangle \\
		&= \int_{\mathbb{R}^2} |n|^{-2\beta} |\hat{f}(n)|^2 \rd n = \| f \|_{\dot{H}^{-\beta}}^2,
	\end{align*}
    and the general case follows by density.
\end{proof}

\subsection{Definition of the Kraichnan covariance}
\label{sec: covariance}

Now we define the spatial isotropic covariance matrix $Q$.

\begin{assumption}
We take $0<\alpha<1$ (the regularity index) and $Q:\R^2\to \R^{2\times 2}$ with Fourier transform
\begin{equation}
    \widehat{Q} (n) = \langle n \rangle^{-(2+2\alpha)} \left(I - \frac{nn^{\top}}{|n|^2}\right),\quad n\in\R^2.\label{eq:Q def}
\end{equation}
\end{assumption}

Since the Fourier transform of $Q$ is integrable, $Q$ is a continuous bounded function on $\R^2$.

We recall that we can find noise coefficients $\sigma_k$ with covariance matrix $Q$:
\begin{lemma}
There exists a sequence of divergence-free functions $\sigma_k$ in $H^{1+\alpha}$, $k\in \mathbb{N}$, such that
\begin{align}
    Q(x-y) = \sum_k\sigma_k(x)\sigma_k(y)^\top \quad \forall x,y\in\R^2\label{eq:Q_sigma_k}
\end{align}
(the series converging absolutely for every $x$ and $y$) and
\begin{align}
    \sup_{x,y\in\R^2}\sum_k|\sigma_k(x)||\sigma_k(y)|\le \sup_{x\in\R^2}\sum_k |\sigma_k(x)|^2 <\infty.\label{eq:sigma_k_conv}
\end{align}
\end{lemma}

\begin{proof}
For every vector fields $\varphi$, $\psi$ in $C^\infty_c(\R^2)^2$, we have
\begin{align}
\begin{aligned}\label{eq:Q_repr_1}
    \iint \varphi(x)\cdot Q(x-y) \psi(y) \rd x \rd y &= \int \overline{\hat{\varphi}(n)}\cdot \hat{Q}(n)\hat{\psi}(n) \rd n\\
    &=  \int \langle n \rangle^{-2-2\alpha} \overline{\hat{\varphi}(n)} \cdot\hat{\psi}(n) \left(I-\frac{nn^T}{|n|^2}\right)\rd n\\
    &= \lan \Pi\psi, \Pi\varphi \ran_{H^{-1-\alpha}}.
\end{aligned}
\end{align}
Here $\Pi$ is the Leray projection on the divergence-free vector fields, that is, in Fourier modes, $\hat{\Pi}(n)=I-\frac{nn^\top}{|n|^2}$, and $\lan\cdot,\cdot\ran_{H^\gamma}$ is the $H^\gamma$ scalar product,
\begin{align*}
    \lan f,g\ran_{H^\gamma} := \int \lan n \ran^{-2-2\alpha} \hat{f}(n)\cdot \overline{\hat{g}(n)} \rd n.
\end{align*}
Now we take a complete orthonormal basis $(e_k)_k$ of $H^{-1-\alpha}_{\operatorname{div}=0}$, the space of $H^{-1-\alpha}$ divergence-free vector-valued distributions, and we take $\sigma_k$, $k\in\mathbb{N}$, with Fourier transform
\begin{align*}
    \widehat{\sigma_k}(n) = \langle n \rangle^{-2-2\alpha} \widehat{e_k}(n).
\end{align*}
Then $(\sigma_k)_k$ is a complete orthonormal basis of $H^{1+\alpha}_{\operatorname{div}=0}$, the space of $H^{1+\alpha}$ divergence-free vector fields and, for every $C^\infty_c$ vector fields $\varphi$, $\psi$,
\begin{align}
\begin{aligned}\label{eq:Q_repr_2}
    \lan \Pi\psi,\Pi\varphi\ran_{H^{-1-\alpha}} &= \sum_k \lan \Pi\psi,e_k\ran_{H^{-1-\alpha}} \lan \Pi\varphi,e_k\ran_{H^{-1-\alpha}}\\
    &= \sum_k \lan \Pi\psi,\sigma_k\ran_{L^2} \lan \Pi\varphi,\sigma_k\ran_{L^2} = \sum_k \lan \psi,\sigma_k\ran_{L^2} \lan \varphi,\sigma_k\ran_{L^2}\\
    &= \sum_k \iint \varphi(x)\cdot  \sigma_k(x)\sigma_k(y)^\top\psi(y) \rd x\rd y.
\end{aligned}
\end{align}
In particular, for any $x_0$ in $\R^2$, taking a sequence $(\varphi_m=\psi_m)_m$ converging to $v_0\delta_{x_0}$ for some given vector $v_0\in \R^2$, we get
\begin{align*}
    \sum_k(\sigma_k(x_0)\cdot v_0)^2 =  \|\Pi[v_0\delta_{x_0}]\|_{H^{-1-\alpha}}^2 = \int \lan n\ran^{-2-2\alpha}\left(I-\frac{nn^\top}{|n|^2}\right) |v_0|^2 \rd n = C|v_0|^2 <\infty
\end{align*}
for some $C>0$ independent of $x_0$ and $v_0$. In particular, taking $v_0=(1,0)$ or $(0,1)$, we get
\begin{align*}
    \sum_k|\sigma_k(x_0)|^2\le C<\infty
\end{align*}
and so
\begin{align*}
    \sum_k|\sigma_k(x)||\sigma_k(y)|\le \left(\sum_k|\sigma_k(x)|^2\right)^{1/2}\left(\sum_k|\sigma_k(y)|^2\right)^{1/2} \le C<\infty.
\end{align*}
that is \eqref{eq:sigma_k_conv}. Putting together the equalities \eqref{eq:Q_repr_1} and \eqref{eq:Q_repr_2} and the convergence \eqref{eq:sigma_k_conv}, we obtain for every $C^\infty_c$ vector fields $\varphi$, $\psi$,
\begin{align*}
\iint \varphi(x)\cdot  Q(x-y)\psi(y) \rd x\rd y 
& = \sum_k \iint \varphi(x)\cdot  \sigma_k(x)\sigma_k(y)^\top\psi(y) \rd x\rd y \\
&= \iint \varphi(x)\cdot  \sum_k\sigma_k(x)\sigma_k(y)^\top\psi(y) \rd x\rd y.
\end{align*}
By the arbitrariness of $\varphi$ and $\psi$, we conclude that \eqref{eq:Q_sigma_k} holds (with pointwise absolute convergence).
\end{proof}

\begin{remark}
    We will often use that
    \begin{align*}
        \langle \varphi,Q\ast \psi\rangle_{L^2} = \sum_k \langle \varphi,\sigma_k\rangle_{L^2}\langle \psi,\sigma_k\rangle_{L^2}
    \end{align*}
    whenever $\varphi$ and $\psi$ are in $L^1(\R^2)$.
\end{remark}

\subsection{Properties of the covariance}

In this subsection, we explore the structure of the covariance matrix and its behaviour close to $x=0$. This behaviour is fundamental in the energy bounds for $\omega$. The material in this subsection is known, e.g. Gawedzki \cite{gawedzki2008stochastic} and Le Jan and Raimond \cite{leJean_Raimond}, but, given its importance here, we give a self-contained explanation.

We have
\begin{align*}
Q(x) = \int_{\R^2} \langle n \rangle^{-(2+2\alpha)} \left(I-\frac{nn^T}{|n|^2}\right) e^{2\pi ix\cdot n} \rd n.
\end{align*}
Since by symmetry
\begin{align*}
&2\int_{\R^2}\langle n\rangle^{-(2+2\alpha)}\frac{(n^1)^2}{|n|^2} \rd n = \int_{\R^2}\langle n\rangle^{-(2+2\alpha)}\left(\frac{(n^1)^2}{|n|^2}+\frac{(n^2)^2}{|n|^2}\right) \rd n= \int_{\R^2}\langle n\rangle^{-(2+2\alpha)} \rd n,\\
&\int_{\R^2}\int_{\R^2}\langle n\rangle^{-(2+2\alpha)}\frac{n^1n^2}{|n|^2} \rd n =0,
\end{align*}
we have ($I$ being the $2\times 2$ identity matrix)
\begin{equation}
    \label{eq:Q 0}
    Q(0) 
    = \frac{1}{2}\int_{\R^2}\langle n \rangle^{-(2+2\alpha)}  \rd n \ I
    = \frac{\pi}{2\alpha} I.
\end{equation}

We fix $x$ in $\R^2$ and take a unitary matrix $U_x$ in $\R^{2\times 2}$ such that $U_x(Re_1) = x$, where $R=|x|$. With the change of variable $\tilde n=U_x^{-1}n$, we get
\begin{align}
Q(x) &= \int_{\R^2} \langle \tilde n \rangle^{-(2+2\alpha)} \left(I-\frac{U_x\tilde n \tilde n^T U_x^T}{|\tilde n|^2}\right)  e^{2\pi iRe_1\cdot \tilde n} \rd \tilde{n}= U_xQ(Re_1)U_x^T,\label{eq:Q isotropic U}
\end{align}
so it is enough to consider $x=Re_1$. For $j,k=1,2$ we have
\begin{align}
\label{eq: QRe_1}
Q(Re_1)^{jk} = \int_{\R^2} \langle n \rangle^{-(2+2\alpha)} \left(\delta^{jk}-\frac{n^jn^k}{| n|^2}\right) (\cos(2\pi Rn^1)+i\sin(2\pi Rn^1) \rd n.
\end{align}
For $j=k$, we have
\begin{align*}
Q(Re_1)^{jj} &= \int_{\R^2} \langle n \rangle^{-(2+2\alpha)} \left(1-\frac{(n^j)^2}{| n|^2}\right) \cos(2\pi Rn^1)\rd n.
\end{align*}
For $j\neq k$, the integrand in \eqref{eq: QRe_1} is odd and therefore
\begin{align*}
Q(Re_1)^{jk}=0.
\end{align*}
We call $B_L(R) := Q(Re_1)^{11}$ and $B_N(R) := Q(Re_1)^{22}$, we end with
\begin{align*}
Q(Re_1) = B_L(|x|) e_1e_1^{\top} + B_N(|x|) (I-e_1e_1^{\top})
\end{align*}
and hence, by \eqref{eq:Q isotropic U}, we have for any $x\neq 0$,
\begin{align*}
Q(x) = B_L(|x|) \frac{xx^\top}{|x|^2} + B_N(|x|) \left(I-\frac{xx^\top}{|x|^2}\right).
\end{align*}
As we will see, $B_L(0)=B_N(0)$ and so the above formula is extended to $x=0$ by continuity.

We further analyse the term $B_L$ and $B_N$, for $R$ close to $0$. Passing to polar coordinates $n=(\rho\cos\theta,\rho\sin \theta)$, we get
\begin{align*}
B_L(R) &= \int_0^\infty \int_{-\pi}^\pi \langle \rho \rangle^{-2-2\alpha} \rho \cos (2\pi R\rho \cos\theta) (1-(\cos\theta)^2)\rd \theta d\rho,\\
B_N(R) &= \int_0^\infty \int_{-\pi}^\pi \langle \rho \rangle^{-2-2\alpha} \rho \cos (2\pi R\rho \cos\theta) (1-(\sin\theta)^2)\rd \theta d\rho\\
&= \int_0^\infty \int_{-\pi}^\pi \langle \rho \rangle^{-2-2\alpha} \rho \cos (2\pi R\rho \cos\theta) (\cos\theta)^2 \rd \theta d\rho
\end{align*}
For notational simplicity, we call, for a Borel bounded even function $f:\R\to \R$,
\begin{align*}
F_f(R) = \int_0^\infty \int_{-\pi}^\pi \langle \rho \rangle^{-2-2\alpha} \rho \cos (2\pi R\rho \cos\theta) f(\cos\theta) \rd \theta d\rho;
\end{align*}
with $f(u)=1-u^2$ we find $B_L$, with $f(u)=u^2$ we find $B_N$. We start with the change of variables $u= \cos\theta$, using symmetries we get
\begin{align*}
    F_f(R)
    & = 2 \int_0^\infty \int_{0}^\pi \langle \rho \rangle^{-2-2\alpha} \rho \cos (2\pi R\rho \cos\theta) f(\cos\theta) \rd \theta \rd\rho \\
    & =  2 \int_0^\infty \int_{-1}^{1} \langle \rho \rangle^{-2-2\alpha} \rho \cos (2\pi R\rho u) f(u) (1-u^2)^{-1/2} \rd u \rd\rho \\
    & =  4 \int_0^\infty \int_{0}^{1} \langle \rho \rangle^{-2-2\alpha} \rho \cos (2\pi R\rho u) f(u) (1-u^2)^{-1/2} \rd u \rd\rho.
\end{align*}
We apply now the change of variables $\tilde{\rho} = R\rho u$, with differential $\rd \tilde{\rho} = Ru\rd \rho$ (with a slight abuse of notation we keep using $\rho$ instead of $\tilde{\rho}$), recalling that $\langle \rho \rangle = (1+\rho^2)^{1/2}$ we get
\begin{align*}
     F_f(R)
    & = 4R^{2\alpha} \int_0^\infty \int_{0}^{1} ((Ru)^2 + \rho^2)^{-\frac{2+2\alpha}{2}}  \rho \cos(2\pi \rho) f(u) u^{2\alpha} (1-u^2)^{-1/2} \rd u \rd\rho
\end{align*}
For any $R>0$, we can repeat the previous steps for $F_f(0)$, with the change of variable $\tilde \rho = R\rho u$, getting
\begin{align*}
     F_f(0)
    & = 4R^{2\alpha} \int_0^\infty \int_{0}^{1} ((Ru)^2 + \rho^2)^{-\frac{2+2\alpha}{2}}  \rho f(u) u^{2\alpha} (1-u^2)^{-1/2} \rd u \rd\rho.
\end{align*}
We take the difference and expand the integral around $R=0$, obtaining
\begin{align}
     F_f(0) - F_f(R)
    & = 4 R^{2\alpha} \int_0^\infty \int_{0}^{1} ((Ru)^2 + \rho^2)^{-\frac{2+2\alpha}{2}}  \rho (1-\cos(2\pi \rho)) f(u) u^{2\alpha} (1-u^2)^{-1/2} \rd u \rd\rho \nonumber\\
     & = 4 R^{2\alpha} \int_0^\infty \rho^{-1-2\alpha} (1-\cos (2\pi\rho)) d\rho \int_{0}^{1} f(u) u^{2\alpha} (1-u^2)^{-1/2} \rd u \nonumber\\
     &\quad +4 R^{2\alpha} \int_0^\infty \int_{0}^{1} \left[ ((Ru)^2 + \rho^2)^{-\frac{2+2\alpha}{2}} - \rho^{-(2+2\alpha)}\right]\nonumber\\
     &\quad\cdot\rho (1-\cos(2\pi\rho)) f(u) u^{2\alpha} (1-u^2)^{-1/2} \rd u \rd\rho \nonumber\\
     &=: \beta_f R^{2\alpha} + \mathrm{Rem}_f(R).\label{eq:def beta remainder}
\end{align}
By Lemma \ref{lem: properties remainder}, the remainder is of order $R^2$, for small $R$. In particular, we find, for small $R$,
\begin{align*}
B_L(0)-B_L(R) &= \beta_L R^{2\alpha} +O(R^2),\\
B_N(0)-B_N(R) &= \beta_N R^{2\alpha} +O(R^2),
\end{align*}
with $\beta_L=\beta_{1-u^2}$ and $\beta_N=\beta_{u^2}$.

We compute $B_L(0)$, $B_N(0)$ and the constants $\beta_L$ and $\beta_N$. We write
\begin{align*}
    F_f(0) = \int_{0}^{\infty}\int_{-\pi}^{\pi}
    \langle \rho \rangle^{-(2+2\alpha)} \rho f(\cos\theta)\rd \theta d\rho 
    &= \int_{0}^{\infty}
    \langle \rho \rangle^{-(2+2\alpha)} \rho d\rho \cdot  \int_{-\pi}^{\pi} f(\cos\theta)\rd \theta .
\end{align*}
Recall that
\begin{align*}
     &\int_{-\pi}^{\pi} (\cos\theta)^2 \rd \theta 
     = \int_{-\pi}^{\pi} (\sin\theta)^2 \rd \theta 
     = \int_{-\pi}^{\pi} (1-(\cos\theta)^2) \rd \theta = \pi,\\
     & \int_0^\infty
    \langle \rho \rangle^{-(2+2\alpha)} \rho d\rho 
    = \frac{1}{2\alpha}.
\end{align*}
Therefore we have 
\begin{align*}
B_L(0) = B_N(0) = F_{u^2}(0) = \frac{\pi}{2\alpha},
\end{align*}
in particular, calling $I$ the $2\times 2$ identity matrix,
\begin{align}
\label{eq: explicit Q0}
    Q(0) = \frac{\pi}{2\alpha}I.
\end{align}

Concerning $\beta_L$ and $\beta_N$, we call $\bar{\beta} = 4\int_0^\infty \rho^{-1-2\alpha} (1-\cos(2\pi\rho)) d\rho >0$ and $\mathrm{B}$ the Beta function. Note that
\begin{align*}
    \beta_L&=\beta_{1-u^2} =\bar{\beta}\mathrm{B}(1+2\alpha,3/2),\\
    \beta_N&=\beta_{u^2} =\bar{\beta}\mathrm{B}(3+2\alpha,1/2).
\end{align*}
By integration by parts (or well-known properties of the Beta function) we get
\begin{align*}
\beta_L 
& =\beta_{1-u^2} 
= \bar{\beta} \int_{0}^{1} u^{2\alpha} (1-u^2)^{\frac{1}{2}} \rd u \\
& = \frac{\bar{\beta}}{1+2\alpha}u^{1+2\alpha}(1-u^2)^{1/2}\mid_{u=0}^{1} +\frac{\bar{\beta}}{1+2\alpha}\int_0^1 u^{2+2\alpha}(1-u^2)^{-1/2} \rd u \\
& = 0+\frac{\beta_{u^2}}{1+2\alpha} 
= \frac{\beta_N}{1+2\alpha}.
\end{align*}

We sum up the obtained properties in the following:
\begin{proposition}\label{prop:covariance structure}
We have:
\begin{align*}
    &Q(x) = B_L(|x|)\frac{xx^\top}{|x|^2} +B_N(|x|)\left(I-\frac{xx^\top}{|x|^2}\right),
\end{align*}
with
\begin{align*}
    &B_L(R) = \frac{\pi}{2\alpha} -\beta_L R^{2\alpha} -\mathrm{Rem}_{1-u^2}(R),\\
    &B_N(R) = \frac{\pi}{2\alpha} -\beta_N R^{2\alpha} -\mathrm{Rem}_{u^2}(R),\\
    &\beta_N=(1+2\alpha)\beta_L >\beta_L>0,
\end{align*}
and $\mathrm{Rem}_{1-u^2}(R)=O(R^2)$, $\mathrm{Rem}_{u^2}(R)=O(R^2)$ as $R\to 0$. In particular we have, for some $C>0$,
\begin{align}
|Q(0)-Q(x)| \le C(|x|^{2\alpha} \wedge 1) \quad \forall x\in \R^2.\label{eq:Q Holder bound}
\end{align}
\end{proposition}

The key fact that we will use is that $\beta_N>\beta_L>0$.

\begin{lemma}
    For every $\epsilon>0$, $P$-a.s., the random field
    \begin{align*}
    V(t,x)= \sum_{k=1}^\infty \sigma_k(x) W_t
    \end{align*}
    is locally $(\alpha-\epsilon)$-H\"older continuous in $x$, uniformly in $t$, and $(1/2-\epsilon)$-H\"older continuous in $t\in [0,T]$, locally uniformly in $x$.   
\end{lemma}
    \begin{proof}
    The proof is based on Kolmogorov continuity theorem.
    
    For $m\in \mathbb{N}$, $x,y\in\R^2$ and $t\geq 0$ we have
    \begin{align*}
        &\E\left[
            |\sum_{k} (\sigma_k(x)-\sigma_k(y)) W^k_t|^{2m}
        \right]
        = 
            \sum_{k_1,\dots,k_{2m}} (\sigma_{k_1}(x)-\sigma_{k_1}(y))^{\top}(\sigma_{k_2}(x)-\sigma_{k_2}(y))\dots \\
        &\quad \cdot     (\sigma_{k_{2m-1}}(x)-\sigma_{k_{2m-1}}(y))^{\top}(\sigma_{k_{2m}}(x)-\sigma_{k_{2m}}(y)) 
             \E[W^{k_{1}}_t\dots W^{k_{2m}}_t].
    \end{align*}
    The expectation on the right hand side vanishes as soon as $W^k_t$ appears an odd number of times in the product, in all other cases, the expectation can be bounded by a global constant $c_{m,T}>0$, which might change from line to line in the following. Hence, only squared terms remain in the sum and we have
    \begin{align*}
        \E\left[
            |\sum_{k} (\sigma_k(x)-\sigma_k(y)) W^k_t|^{2m}
        \right]
        &\leq 
            c_{m,T}\sum_{k_1,\dots,k_{m}} |\sigma_{k_1}(x)-\sigma_{k_1}(y)|^{2}\dots |\sigma_{k_{m}}(x)-\sigma_{k_{m}}(y)|^2\\
        &\leq c_{m,T} \left(
            \sum_{k} |\sigma_{k}(x)-\sigma_{k}(y)|^{2}
        \right)^m \\
        & = c_{m,T} \left|
            2Q(0)-2Q(x-y)
        \right|^m\\
        &\leq c_{m,T} |x-y|^{2\alpha m},
    \end{align*}
    where in the last line we used \eqref{eq:Q Holder bound}.
    As an application of Kolmogorov continuity criterion, we have that $V(t,\cdot)$ is $\beta$-H\"older continuous for every $\beta \in (0, \alpha-1/m)$. By choosing $m>0$ big enough, we get the desired result.

    The time regularity follows in a similar way. Let $m\in \mathbb{N}$, $x\in \R^2$ and $t,s\in [0,T]$, there exist $c_{m}>0$ (which might change from one line to the next) such that
    \begin{align*}
        &\E \left[
            |V(t,x) - V(s,x)|^{2m}
        \right]
        = \E\left[
            |\sum_{k} \sigma_k(x)( W^k_t-W^k_s)|^{2m}
        \right]\\
        &\quad =\sum_{k_1,\dots,k_{2m}} \sigma_{k_1}(x)^{\top}\sigma_{k_2}(x)\dots    \sigma_{k_{2m-1}}(x)^{\top}\sigma_{k_{2m}}(x)
             \E[(W^{k_{1}}_t-W^{k_{1}}_s)\dots (W^{k_{2m}}_t-W^{k_{2m}}_s)]\\
        &\quad =\sum_{k_1,\dots,k_{m}} |\sigma_{k_1}(x)|^2\dots |\sigma_{k_{m}}(x)|^2
             \E[(W^{k_{1}}_t-W^{k_{1}}_s)^2\dots (W^{k_{m}}_t-W^{k_{m}}_s)^2]\\
        &\quad \leq c_m |t-s|^{m} \left(\sum_{k} |\sigma_{k}(x)|^2\right)^m
        \leq c_m |t-s|^{m}.
    \end{align*}  
    By Kolmorogrov continuity theorem we have that $V(\cdot,x)$ is $\beta$-H\"older continuous for every $\beta \in (0,1/2-1/(2m))$, uniformly in $x\in \R^2$.
\end{proof}

\subsection{The existence result}

We give the rigorous definition of weak solutions in $\dot{H}^{-1}$ to \eqref{eq:main}. We recall that $K$ is the Biot-Savart kernel (defined in \eqref{eq: green kernel K}) and $Q$, the spatial covariance matrix, is defined in \eqref{eq:Q def}. In particular, $\sigma_k$ are divergence-free and $Q(0)=\frac{\pi}{2\alpha}I$, so that $\tr[Q(0)D^2] = \frac{\pi}{2\alpha}\Delta$ and equation \eqref{eq:main} reads formally in It\^o form as
\begin{align*}
&\rd \omega +u\cdot\nabla\omega \rd t +\sum_k\sigma_k\cdot\nabla\omega \rd W^k = \frac{\pi}{4\alpha} \Delta \omega,\\
&u=K\ast\omega.
\end{align*}
We will always work with the It\^o formulation.

Note that, for any scalar field $f$ in $\dot{H}^{-1}$, the vector field $v=K*f$ is well-defined in $L^2$ and satisfies $\diverg v=0$ and $\text{curl} v=f$. Moreover, if $f$ is also in $L^2$, then
\begin{align}
\label{eq: curl div}
v\cdot \nabla f = \text{curl}\, \diverg (vv^\top),
\end{align}
and the right-hand side makes sense for $f$ only in $\dot{H}^{-1}$.

\begin{definition}\label{def:def_main}
A weak solution in $\dot{H}^{-1}$ to \eqref{eq:main} is an object $(\Omega,\mathcal{A},(\mathcal{F}_t)_t,P,(W^k)_k,\omega)$, where $(\Omega,\mathcal{A},(\mathcal{F}_t)_t,P)$ is a filtered probability space satisfying the standard assumption, $(W^k)_k$ is a sequence of independent real $(\mathcal{F}_t)_t$-Brownian motions, $\omega:[0,T]\times\Omega\to \dot{H}^{-1}(\R^2)$ is a $(\mathcal{F}_t)_t$-progressively measurable process satisfying
\begin{align}
\omega \in L^\infty_t(\dot{H}^{-1}) \cap C_t(H^{-4}) \quad P\text{-a.s.}\label{eq: properties def}
\end{align}
and we have, as equality in $H^{-4}$,
\begin{align}
\begin{aligned}\label{eq: formula def 27}
\omega_t &= \omega_0 -\int_0^t \operatorname{curl}\operatorname{div}(u_r u_r^\top) \rd r\\
&\quad -\sum_{k=1}^\infty\int_0^t \operatorname{div}(\sigma_k\omega_r) \rd W^k_r+\frac{\pi}{4\alpha}\int_0^t \Delta \omega_r \rd r \quad \text{for every }t\in [0,T],
\end{aligned}
\end{align}
where $u_r=K*\omega_r$.
\end{definition}

As we will see in the proof of Lemma \ref{lem: C gamma of H minus 3} and in the Step 4 in the proof of Theorem \ref{thm: weak-existence}, if $\omega$ is in $L^\infty_t(\dot{H}^{-1}) \cap C_t(H^{-4})$, then all the terms in \eqref{eq: formula def 27} make sense.

\begin{remark}
\label{rem: ito strato}
    Equation \eqref{eq: formula def 27} is given in It\^o form, whereas equation \eqref{eq:main} is given in Stratonovich formulation. At least formally, the two formulations \eqref{eq: formula def 27} and \eqref{eq:main} coincide. Indeed, at least formally, the last term in \eqref{eq: formula def 27} corresponds to the It\^o-Stratonovich correction, as the following formal computation on the correction term shows (see also \cite{coghi_flandoli,galeati2023weak}): denoting by $[M,N]$ the quadratic covariation of $M$ and $N$,
    \begin{align*}
        \sum_{k=1}^{\infty}\rd [\operatorname{div}(\sigma_k \omega),W^k]_t
        & = \sum_{k,l=1}^{\infty} \operatorname{div}(\sigma_k \sigma_l \cdot \nabla \omega_t) \rd [W^k,W^l]_t
        = \sum_{k=1}^{\infty} \operatorname{div}(\sigma_k \sigma^{\top}_k \nabla \omega_t) \rd t\\
        & = \operatorname{div}(Q(0) \nabla \omega_t) \rd t
        = \frac{\pi}{2\alpha}\Delta \omega_t \rd t,
    \end{align*}
where we have used that $\sigma_k$ are divergence-free and, in the last equality, we have used \eqref{eq: explicit Q0}.

As noted in \cite[Section 2.2]{galeati2023weak}, to make this computation rigorous, we need some smoothness on $\sigma_k$ or on $\omega$, which we do not have here. Hence at the rigorous level, we work directly with the It\^o formulation \eqref{eq: formula def 27}, without using the Stratonovich formulation.

However, we will construct solutions $\omega$ to the It\^o formulation \eqref{eq: formula def 27} as limits of solutions $\omd$ to an approximate Euler equation (see equation \eqref{eq: regularized main} below), which is obtained from \eqref{eq: formula def 27} by regularizing $K$ and $Q$ and for which the equivalence between the It\^o formulation and Stratonovich formulation is rigorous (in the sense for example of \cite[Proposition 2.14]{galeati2023weak}). Hence the It\^o formulation \eqref{eq: formula def 27} is a natural rigorous formulation of equation \eqref{eq:main}.
\end{remark}

Our first main result, on existence of $\dot{H}^{-1}$ weak solutions, is the following:

\begin{theorem}
\label{thm: weak-existence}
Assume that the spatial covariance matrix $Q$ is defined in \eqref{eq:Q def}. Assume that $\omega_0$ is in $\dot{H}^{-1}(\R^2)$. Then equation \eqref{eq:main} admits a weak solution in $\dot{H}^{-1}$ (in the sense of Definition \ref{def:def_main}) and satisfying in addition
\begin{align}
\label{eq:Hm alpha bd sol}
    \omega \in L^\infty([0,T];L^2(\Omega;\dot{H}^{-1})) \cap L^2([0,T]\times \Omega;H^{-\alpha}).
\end{align}
\end{theorem}

\subsection{The uniqueness result}

Our second main result, on strong uniqueness of $L^p$ solutions, is the following:

\begin{theorem}\label{thm:Lp_wellposed}
    Assume that the spatial covariance matrix $Q$ is defined in \eqref{eq:Q def}. Take $3/2<p<\infty$ and $\max\{0,2/p-1\}<\alpha<\min\{1-1/p,1/2\}$. Assume that $\omega_0$ is in $L^p(\R^2)\cap L^1(\R^2) \cap\dot{H}^{-1}(\R^2)$. Then:
    \begin{enumerate}
        \item There exists a weak solution $(\Omega,\mathcal{A},(\mathcal{F}_t)_t,P,(W^k)_k,\omega)$ in $\dot{H}^{-1}$ (in the sense of Definition \ref{def:def_main}) to \eqref{eq:main}, which satisfies in addition \eqref{eq:Hm alpha bd sol} and
        \begin{align}   \label{eq: Lp bound}
            \sup_{t\in [0,T]}[\|\omega_t\|_{L^p}+\|\omega_t\|_{L^1}]\le \|\omega_0\|_{L^p}+\|\omega_0\|_{L^1},\quad P\text{-a.s.}.
        \end{align}
        \item Strong uniqueness holds for \eqref{eq:main} in the class of $\dot{H}^{-1}$ solutions in the space
        \begin{align*}
            \chi:= L^\infty([0,T]\times\Omega;L^p\cap L^1) \cap L^\infty([0,T];L^2(\Omega;\dot{H}^{-1})) \cap L^2([0,T]\times \Omega;H^{-\alpha}).
        \end{align*}
        That is: if $\omega^1$ and $\omega^2$ are two weak solutions in $\dot{H}^{-1}$ to \eqref{eq:main} (in the sense of Definition \ref{def:def_main}), on the same filtered probability space $(\Omega,\mathcal{A},(\mathcal{F}_t)_t,P)$ and with respect to the same sequence $(W^k)_k$ of independent Brownian motions, and $\omega^1,\omega^2$ are in $\chi$, then $\omega^1=\omega^2$ $P$-a.s..
    \end{enumerate}
\end{theorem}

In particular, we get strong uniqueness for $p=2$ (solutions with finite enstrophy) and any $0<\alpha<1/2$.

\begin{remark}\label{rmk:comment p alpha}
    The class $L^2([0,T]\times \Omega;H^{-\alpha}(\R^2))$ contains the class $L^\infty([0,T]\times\Omega;L^p(\R^2)\cap L^1(\R^2))$ if $p\ge 2/(1+\alpha)$ by Sobolev embedding, hence in this case the class $L^2([0,T]\times \Omega;H^{-\alpha}(\R^2))$ is redundant in the uniqueness statement.
\end{remark}

\begin{remark}\label{rmk:non optimality}
While the condition $1<p<\infty$, $0<\alpha<\min\{1-1/p,1/2\}$ seems structural from our approach, the condition $p>3/2$, $\alpha>2/p-1$ is rather unnatural and appears due to a technical point, namely the control of the remainder term $R2$ in formula \eqref{eq:B split} below. We expect that it can be avoided with a finer approximation argument.
\end{remark}

\section{Regular approximations}
\label{sec:regularization}

In this section, we introduce a series of regularizations for the irregular elements appearing in \eqref{eq:main}.

\subsection{Regularized covariance}

Recall that the covariance matrix $Q$ is defined in Subsection \ref{sec: covariance}. For $\delta>0$, we introduce a function $\rho^\delta$ whose Fourier transform $\widehat{\rho^\delta}$ is a real-valued radial $C^\infty$ function with $0\le \widehat{\rho^\delta}\le 1$ everywhere, $\widehat{\rho^\delta}(n)=1$ on $|n|\le 1/\delta$ and $\widehat{\rho^\delta}(n)=0$ on $|n|>2/\delta$. Hence $\rho^\delta$ is a $C^\infty$ even function and $\rho^\delta$ and its derivatives of all orders are in $L^p(\R^2)$ for every $1\le p\le \infty$. We take
\begin{align*}
    \sigma^\delta_k &= \rho^\delta \ast \sigma_k,\\
    Q^\delta &= \rho^{\delta,2}\ast Q
\end{align*}
where $\rho^{\delta,2}=\rho^\delta\ast \rho^\delta$. Note that
\begin{align*}
    \widehat{Q^\delta}(n)= \hat{Q}^\delta(n)\widehat{\rho^\delta}(n)^2 = \langle n \rangle^{-(2+2\alpha)} \left(I-\frac{nn^T}{|n|^2}\right) \widehat{\rho^\delta}(n)^2
\end{align*}
and so $Q^\delta$ satisfies
\begin{align}
    \label{eq: properties Q delta}
    Q^{\delta}\in C^{\infty},
	\qquad
    \|Q^{\delta}\|_{L^\infty}\le \|\widehat{Q^\delta}\|_{L^1} \le\|\widehat{Q}\|_{L^1} <\infty.
\end{align}

Moreover, as for \eqref{eq:Q 0}, one can check that, as $\delta\to 0$,
\begin{equation}
    \label{eq:Q delta 0}
    Q^{\delta}(0) 
    = \frac12 \int_{\R^2}\langle n \rangle^{-(2+2\alpha)} \widehat{\rho^\delta}(n) \rd n \ I
    = : c_{\delta} I \to \frac{\pi}{2\alpha} I=Q(0).
\end{equation}

\begin{lemma}\label{lem:Q_reguralized}
We have, for every $m$ in $\mathbb{N}$,
    \begin{align*}
        \sup_{x\in \R^2}\sum_k|D^m\sigma^{\delta}_k(x)|^2<\infty
    \end{align*}
    and
    \begin{align*}
        Q^{\delta}(x-y) = \sum_k \sigma^{\delta}_k(x)\sigma^{\delta}_k(y)^\top.
    \end{align*}
\end{lemma}

\begin{proof}
    We have, for every $x$ and $y$, using that $\rho^\delta$ is even and in $L^1$,
    \begin{align*}
        \sum_k \sigma^{\delta}_k(x) \sigma^{\delta}_k(y)^{\top}
        &= \sum_k \iint \sigma_k(x^{\prime}) \rho^{\delta}(x-x^{\prime})
        \sigma_k(y^{\prime}) \rho^{\delta}(y-y^{\prime})
        \rd x^{\prime} \rd y^{\prime} \\
        &= \iint \rho^{\delta}(x-x^\prime) \rho^{\delta}(y-y^\prime) Q(x^{\prime}-y^{\prime}) \rd x^{\prime} \rd y^{\prime} \\
        &= \iint \rho^{\delta}(x-y-z-u) \rho^{\delta}(u)\rd u Q(z) \rd z  \\
        &= \int (\rho^{\delta}\ast\rho^{\delta})(x-y-z) Q(z) \rd z  \\
        &= Q^\delta(x-y).
    \end{align*}
    Similarly, taking derivatives, we get
    \begin{align*}
        \sum_k D^m\sigma^{\delta}_k(x) D^m\sigma^{\delta}_k(y)^{\top} &= \iint D^m\rho^{\delta}(x-y-z-u) D^m\rho^{\delta}(u)\rd y^{\prime} Q(z) \rd z\\
        &= D^{2m}Q^\delta(x-y).
    \end{align*}
    In particular, taking $x=y$, we get
    \begin{align*}
        \sum_k|D^m\sigma_k^\delta(x)|^2 =\sum_{i_1,\ldots i_m \in \{1,2\}}\tr[\partial_{x_{i_1}}^2\ldots \partial_{x_{i_m}}^2Q^\delta(0)] <\infty.
    \end{align*}
    The proof is complete.
\end{proof}

\begin{remark}\label{rmk:sigma_delta_1}
    Proceeding as in the previous Lemma, one can show that
    \begin{align*}
     Q^{\delta,h}(x-y) := \rho^\delta\ast Q(x-y) =  \sum_k\sigma^\delta_k(x)\sigma_k(y)^{\top} = \sum_k\sigma_k(x)\sigma^\delta_k(y)^{\top}.
 \end{align*}
\end{remark}

\subsection{Initial conditions}
\label{section: initial condition}

For a given initial condition $\omega_0$ in $\dot{H}^{-1}$, we take a family $(\omega_0^\delta)_{\delta>0}$ such that
\begin{align}
\begin{aligned}
\label{eq: omega zero wish list}
    &\omega_0^\delta\in L^1\cap L^\infty\cap \dot{H}^{-1},\\
    &\int_{\R^2}\omega_0^\delta(x) \rd x =0,\\
    &\int_{\R^2}|x|^m|\omega_0^\delta(x)|\rd x<\infty\text{ for every }m\in \mathbb{N},
\end{aligned}
\end{align}
and
\begin{align}
    &\omega_0^\delta \to \omega_0 \text{ in }\dot{H}^{-1} \text{ as }\delta\to 0.\label{eq:omega zero convergence}
\end{align}

The existence of such a family $(\omega_0^\delta)_\delta$ is a consequence of a classical density argument. Precisely, we take $u_0=K\ast \omega_0$, which is in $L^2(\R^2)$ and satisfies $\operatorname{curl}(u_0)=\omega_0$ in the sense of distributions. Then we take a family $(u^\delta_0)_\delta$ of $C^\infty_c$ vector fields (with possibly non-zero divergence) tending to $u_0$ in $L^2$ and take $\omega^\delta_0=\operatorname{curl}(u^\delta_0)$. The properties \eqref{eq: omega zero wish list} are a consequence of the smoothness and compact support of $u^\delta$ and of the definition of $\omega^\delta_0$ as curl. The convergence \eqref{eq:omega zero convergence} follows from the convergence of $u^\delta_0$ to $u_0$ in $L^2$.

We also ask the following condition:
\begin{align}\label{eq:omega zero Lp}
    \text{if }\omega_0\in L^p\cap L^1\cap \dot{H}^{-1}\text{ for some }1<p<\infty,\text{ then }\omega_0^\delta \to \omega_0 \text{ in }L^p\cap L^1 \text{ as } \delta \to 0.
\end{align}
To show that the condition \eqref{eq:omega zero Lp} can be satisfied, together with \eqref{eq: omega zero wish list} and \eqref{eq:omega zero convergence}, we consider $\omega_0\in L^p\cap L^1 \cap \dot{H}^{-1}$ for some $1<p<\infty$ and we choose an example of $(\omega^\delta_0)_\delta$ built with the previous construction. As before we take $u_0=K\ast \omega_0$. Following \cite[proof of Theorem 10.2]{majda2002vorticity}, $\nabla u_0$ is in $L^p$. Moreover, $K=K1_{|\cdot|\le 1} +K1_{|\cdot|>1}$ can be decomposed as a sum of an $L^1$ function and a function in $\cap_{q>2}L^q$. Hence, since $\omega_0$ is in $L^p\cap L^1$,
\begin{align*}
    u_0 = K1_{|\cdot|\le 1}\ast \omega_0 +K1_{|\cdot|> 1}\ast \omega_0 =: u_{0,p} +u_{0,2+}
\end{align*}
is the sum of an $L^p$ function and a function in $\cap_{q>2}L^q$. Now we take a $C^\infty_c$ function $\bar{\rho}_1$, a $C^\infty_c$ function $\chi_1$ with $0\le \chi_1\le 1$ everywhere, $\chi_1(x)=1$ on $|x|\le 1$, $\chi_1(x)=0$ on $|x|\ge 2$ and, for $\delta>0$, $\bar{\rho}_\delta = \delta^{-2}\bar{\rho}_1(\delta^{-1}\cdot)$, $\chi_\delta=\chi(\delta\cdot)$. We take, for $\delta>0$,
\begin{align*}
    &u^\delta_0 = (u_0\ast \bar{\rho}_\delta)\chi_\delta,\\
    &\omega^\delta_0 = \operatorname{curl}(u^\delta_0) = (\omega_0\ast \bar{\rho}_\delta)\chi_\delta -(u_0\ast \bar{\rho}_\delta)\cdot \nabla^\perp\chi_\delta.
\end{align*}
Clearly $((\omega_0\ast \bar{\rho}_\delta)\chi_\delta)_\delta$ converges to $\omega_0$ in $L^p\cap L^1$. For $(u_0\ast \bar{\rho}_\delta)\cdot \nabla^\perp\chi_\delta$, we have
\begin{align*}
    &\|(u_{0,p}\ast \bar{\rho}_\delta)\cdot \nabla^\perp\chi_\delta\|_{L^p}\le \|u_{0,p}\|_{L^p} \|\nabla\chi_\delta\|_{L^\infty} \lesssim \delta \|u_{0,p}\|_{L^p} \to 0
\end{align*}
and
\begin{align*}
    &\|(u_{0,2+}\ast \bar{\rho}_\delta)\cdot \nabla^\perp\chi_\delta\|_{L^p}\lesssim \|u_{0,2+}\|_{L^{2p}} \|\nabla\chi_\delta\|_{L^{2p}} \lesssim \delta^{1-1/p} \|u_{0,2+}\|_{L^{2p}} \to 0.
\end{align*}
Hence $(\omega^\delta_0)_\delta$ converges to $\omega$ in $L^p$. Moreover, since $u_0$ is in $L^2$ (because $\omega_0$ is in $\dot{H}^{-1}$), we have, using that $\chi_\delta$ has support on $|x|\ge 1/\delta$,
\begin{align*}
    &\|(u_0\ast \bar{\rho}_\delta)\cdot \nabla^\perp\chi_\delta\|_{L^1} \lesssim \|(u_0\ast \bar{\rho}_\delta)1_{|\cdot|\ge 1/\delta}\|_{L^2} \|\nabla\chi_\delta\|_{L^2} \lesssim \|u_01_{|\cdot|\ge 1/(2\delta)}\|_{L^2} \to 0.
\end{align*}
Hence $(\omega^\delta_0)_\delta$ converges to $\omega$ also in $L^1$.

\subsection{Regularizing the Green kernel}

Recall that $G$ defined in \eqref{eq: green kernel} is the Green kernel of $-\Delta$ in $L^2(\R^2)$. For $\delta>0$, we take the regularized Green kernel $G^\delta$ as
\begin{align}
\label{eq: regularization green kernel}
G^\delta(x) = \int_\delta^{1/\delta} p_s(x) \rd s,
\qquad
x\in \R^2,
\end{align}
where $p$ is the heat kernel on $\R^2$, namely
\begin{align*}
p_s(x) = s^{-1}p_1(s^{-1/2}x),
\qquad
p_1(x) =(4\pi)^{-1}e^{-|x|^2/4},
\qquad
x\in\R^2, \ s>0.
\end{align*}
We set the regularized Biot-Savart kernel as
\begin{align}
K^{\delta} = \nabla^{\perp} G^{\delta}.\label{eq: regularization green kernel K}
\end{align}

In the next two lemmas we collect a few (standard) facts on $G^\delta$ and its Fourier transform.

\begin{lemma}
The Fourier transform of $G^\delta$ is
\begin{align}
    &\widehat{G^\delta}(n) = (2\pi)^{-2}|n|^{-2}(e^{-\delta 4\pi^2|n|^2}-e^{-4\pi^2|n|^2/\delta}).\label{eq:G delta Fourier}
\end{align}
In particular, we have
\begin{align*}
    &|\widehat{\nabla G^\delta}(n)| \le (2\pi)^{-1}|n|^{-1},
\end{align*}
and, for any $\beta<1$, for every $f$ in $\dot{H}^{\beta-1}$
\begin{align}
    \|\nabla G^\delta\ast f\|_{\dot{H}^\beta}\le \|\nabla G\ast f\|_{\dot{H}^\beta} = (2\pi)^{-1}\|f\|_{\dot{H}^{\beta-1}}.\label{eq:nabla G bound H}
\end{align}
Moreover, as $\delta\to 0$, for every $\omega \in \dot{H}^{-1}$,
\begin{equation}
    \label{eq: convergence g delta to H-1}
    \langle \omega, G^\delta\ast\omega \rangle \to \int (2\pi)^{-2}|n|^{-2}|\hat{\omega}(n)|^2 \rd n = (2\pi)^{-2}\|\omega\|_{\dot{H}^{-1}}^2.
\end{equation}
\end{lemma}

\begin{proof}
    The heat kernel $p_s$ has Fourier transform
    \begin{align*}
        \widehat{p_s}(n) = e^{-4\pi^2s|n|^2},
    \end{align*}
    and the expression \eqref{eq:G delta Fourier} of the Fourier transform of $G^\delta$ follows integrating $s\in [\delta,1/\delta]$. The bounds on $\nabla G^\delta$ and $\nabla G^\delta\ast f$ follow from \eqref{eq:G delta Fourier}. Finally \eqref{eq: convergence g delta to H-1} follows from \eqref{eq:G delta Fourier} by monotone convergence theorem.
\end{proof}

\begin{lemma}
For every $k>0$ positive integer, there exists a positive constant $C_k$ such that
\begin{equation}
    \label{eq: bounds G delta}
    |D^k G^\delta|\le C_k|x|^{-k},
\end{equation}
Moreover, there exists $C>0$ such that, for every $\delta \le 1$, for every $x\in \R^2\setminus\{0\}$,
\begin{align}
    &|\nabla (G-G^\delta)|\le  C \left[
        \delta^{3/4} 1_{\delta^{1/4} \leq |x| \leq \delta^{-1/4}}
        + \frac{1}{|x|} 1_{|x|\le \delta^{1/4} \vee |x|\ge \delta^{-1/4}}
    \right],\label{eq: convergence first derivative G delta}\\
    &|D^2 (G-G^\delta)|
    \le C\delta + C\frac{1}{|x|^2}1_{|x|< (8\delta\log(1/\delta))^{1/2}}.
    \label{eq: convergence second derivative G delta}
\end{align}
(here $\nabla G$ is the distributional and pointwise deritative of $G$, $D^2G$ is the pointwise derivative).
\end{lemma}

\begin{proof}
We start proving a uniform bound on $\nabla G^\delta$. By the change of variable $s=|x|^2t$, we have, for every $x\neq 0$,
\begin{align*}
    \nabla G^\delta(x) = \int_\delta^{1/\delta} \nabla p_s(x) \rd s  
    & = - \frac{x}{|x|^2} \frac{1}{2\pi} \left(
        e^{-\frac{|x|^2\delta}{4}} - e^{-\frac{|x|^2}{4\delta}}
    \right)
\end{align*}
and so
\begin{align*}
    |\nabla G^\delta(x)| \le \frac{C}{|x|}.
\end{align*}
Differentiating, we get for $D^2G^\delta$
\begin{align}
    \label{eq: second derivative G delta}
    D^2G^\delta(x) = \frac{1}{2\pi} \frac{1}{|x|^2} \left[\left(-I+\frac{2xx^T}{|x|^2}+\frac{\delta xx^T}{2}\right)e^{-\frac{|x|^2\delta}{4}} - \left(-I+\frac{2xx^T}{|x|^2}+\frac{xx^T}{2\delta}\right)e^{-\frac{|x|^2}{4\delta}}\right]
\end{align}
and so
\begin{align*}
    |D^2G^\delta(x)|\le \frac{C}{|x|^2}.
\end{align*}
For a general $k \in \mathbb{N}$, we have, thanks to the change of variables $s= |x|^2 t$,
\begin{equation*}
D^k G^{\delta}(x) 
= \int_{\delta}^{1/\delta} s^{-k/2-1} D^k p_1(s^{-1/2}x) \rd s
= |x|^{-k} \int_{\frac{\delta}{|x|^2}}^{\frac{1}{\delta|x|^2}}
t^{(k+2)/2}
D^k p_1(t^{-1/2}\frac{x}{|x|}) \rd t,
\end{equation*}
which implies
\begin{equation*}
    |D^k G^{\delta}(x) |
    \leq 
|x|^{-k} \int_{0}^{\infty}
t^{(k+2)/2}
|D^k p_1(t^{-1/2}\frac{x}{|x|}) |\rd t.
\end{equation*}
Note that $D^k p_1(z)$ is a polynomial $q$ in the variable $z \in\R^2$ multiplied by the exponential $e^{-|z|^4/4}$. When we evaluate at $z=t^{-1/2} x/|x|$ and take the absolute value, the contribution of the term $x/|x|$ can be bounded uniformly in $x$ and hence
\begin{equation*}
    \int_0^{\infty} t^{-(2+k)/2}|D^kp_1(z)| \rd t
    \leq \int_0^{\infty} t^{-(1+k/2)}q(t^{-1/2}) e^{-4/t} \rd t =: C_k < \infty
\end{equation*}
for some polynomial $q$.

We move to the bound on $\nabla(G^\delta-G)$. Using the elementary estimate $(1-e^{-a})/a\le 1$ for $a>0$, we obtain for some $C>0$,
\begin{align*}
    |\nabla (G-G^{\delta})(x)|
    &= \frac{1}{2\pi |x|} |1- e^{-\frac{|x|^2\delta}{4}} + e^{-\frac{|x|^2}{4\delta}}| \\
    &\le \frac{1}{2\pi}
    \left(
        \frac{1-e^{-\frac{|x|^2\delta}{4}}}{|x|} + \frac{e^{-\frac{|x|^2}{4\delta}}}{|x|}
    \right) 1_{\delta^{1/4} \leq |x| \leq \delta^{-1/4}} +\frac{1}{2\pi|x|}1_{|x|\le \delta^{1/4} \vee |x|\ge \delta^{-1/4}} \\
    & \le \frac{1}{2\pi}
    \left(
        \frac{|x|\delta}{4} + \frac{e^{-\frac{|x|^2}{4\delta}}}{|x|}
    \right) 1_{\delta^{1/4} \leq |x| \leq \delta^{-1/4}} +\frac{1}{2\pi|x|}1_{|x|\le \delta^{1/4} \vee |x|\ge \delta^{-1/4}}  \\
    & \le \frac{1}{2\pi} \left( \frac{\delta^{3/4}}{4} + \delta^{-1/4} e^{-\frac{\delta^{-1/2}}{4}}
        \right) 1_{\delta^{1/4} \leq |x| \leq \delta^{-1/4}} +\frac{1}{2\pi|x|}1_{|x|\le \delta^{1/4} \vee |x|\ge \delta^{-1/4}}  \\
    & \le C \left[
        \delta^{3/4} 1_{\delta^{1/4} \leq |x| \leq \delta^{-1/4}}
        + \frac{1}{|x|} 1_{|x|\le \delta^{1/4} \vee |x|\ge \delta^{-1/4}}
    \right]
\end{align*}
that is \eqref{eq: convergence first derivative G delta}.

Finally, we show the bound on $D^2(G^\delta-G)$ (recall that $D^2G$ is understood as pointwise derivative of $G$ in $x\neq 0$):
\begin{align*}
    |D^2(G-G^\delta)(x)| &\le \frac{C}{|x|^2}(1-e^{-\frac{|x|^2\delta}{4}}) +C\delta e^{-\frac{|x|^2\delta}{4}} +C\left(\frac{1}{|x|^2}+\frac{1}{\delta}\right)e^{-\frac{|x|^2}{4\delta}}\\
    &\le C\delta +C\left(\frac{1}{\delta\log(1/\delta)}+\frac{1}{\delta}\right)e^{-2\log(1/\delta)}1_{|x|\ge (8\delta\log(1/\delta))^{1/2}}\\
    &\quad +C\frac{1}{|x|^2}\left(1+\frac{|x|^2}{4\delta}\right)e^{-\frac{|x|^2}{4\delta}} 1_{|x|< (8\delta\log(1/\delta))^{1/2}}\\
    &\le C\delta + C\frac{1}{|x|^2}1_{|x|< (8\delta\log(1/\delta))^{1/2}},
\end{align*}
that is \eqref{eq: convergence second derivative G delta}.
\end{proof}

\subsection{The regularized model}

Thanks to the regularized objects defined above, we can consider the regularized equation for \eqref{eq:main}, that is: for every $\delta>0$, we consider the following equation
\begin{align}
\begin{aligned}\label{eq: regularized main}
    &\rd \omd + (K^{\delta}\ast\omd)\cdot\nabla\omd\rd t +\sum_k\sigma^\delta_k\cdot\nabla\omd \rd W^k = \frac{c_\delta}{2} \Delta \omd \rd t,\\
    &\omd_t|_{t=0} = \omd_0.
\end{aligned}
\end{align}
We understand the rigorous definition of a solution $\omd$ to \eqref{eq: regularized main} as in Definition \ref{def:def_main}, but replacing \eqref{eq: formula def 27} with
\begin{align*}
\omd_t &= \omd_0 -\int_0^t \operatorname{div}((K^\delta\ast \omd_r)\omd_r) \rd r\\
&\quad -\sum_{k=1}^\infty\int_0^t \operatorname{div}(\sigma^\delta_k\omd_r) \rd W^k_r+\frac{c_\delta}{2}\int_0^t \Delta \omd_r \rd r \quad \text{for every }t\in [0,T].
\end{align*}
This is a Vlasov-type equation with smooth interaction kernel and with smooth common noise.

\begin{lemma}
\label{lem: l infty and l 1}
	For any $\omd_0$ satisfying \eqref{eq: omega zero wish list}, for every filtered probability space $(\Omega,\mathcal{A},(\mathcal{F}_t)_t,P)$ (with the standard assumption) and every sequence $(W^k)_k$ of independent real Brownian motions, there exists a $\dot{H}^{-1}$ solution $\omd$ to equation \eqref{eq: regularized main} such that
	\begin{align*}
		&\sup_{t\in[0,T]}\|\omd_t\|_{L^{\infty}} 
		\leq \| \omd_0\|_{L^{\infty}},\quad P\text{-a.s.},\\
        &\sup_{t\in[0,T]}\|\omd_t\|_{L^{1}} 
		\leq \| \omd_0\|_{L^{1}},\quad P\text{-a.s.},\\
        &\E\left[\sup_{t\in [0,T]}\|\omd_t\|_{\dot{H}^{-1}}^2\right] <\infty.
	\end{align*}
    Moreover, the solution is unique in the space $L^{\infty}([0,T];L^p(\Omega;L^2))$ for every $p>2$.
\end{lemma}
A sketch of the proof of this lemma is given in Appendix \ref{sec: proof of regular existence}.

From now on, unless otherwise stated we assume that the family $(\omd_0)_\delta$ satisfies \eqref{eq: omega zero wish list}, \eqref{eq:omega zero convergence} and the following bound, for some $0<\bar{\epsilon}<(1\wedge (2\alpha))/4$, as $\delta\to 0$:
\begin{align}
\begin{aligned}\label{eq:omega zero growth}
    \delta \|\omd_0\|_{L^{1}}\int_{\mathbb{R}^2} |x|^2 |\omega_0^{\delta}(x)|\rd  x  &= o(1),\\
    \delta^{\bar{\epsilon}}\| \omega^{\delta}_0 \|_{L^{\infty}} \| \omega^{\delta}_0 \|_{L^{1}} &= o(1),\\
    \delta^{1/8} (\|\omd_0\|_{L^1}+\|\omd_0\|_{L^\infty}) &= o(1).
\end{aligned}
\end{align}
Note that, given a family $(\omd_0)_\delta$ satisfying \eqref{eq: omega zero wish list}, \eqref{eq:omega zero convergence}, we can always make \eqref{eq:omega zero growth} possible by relabeling the family $(\omd_0)_\delta$.

The following lemma gives a sufficient condition for a function to be in $\dot{H}^{-1}$, which we will use later.

\begin{lemma}
\label{lem: properties initial condition}
    If $f$ is in $L^1\cap L^2$ and
    \begin{align*}
        &\int_{\R^2} f(x) \rd x =0,\\
        &\int_{\R^2} |x||f(x)| \rd x<\infty,
    \end{align*}
    then
    \begin{align*}
       |\hat{f}(n)| \le 2\pi |n| \int_{\R^2} |x||f(x)|\rd x.
    \end{align*}
    In particular $f$ is in $\dot{H}^{-1}$ and
    \begin{align*}
        \|f\|_{\dot{H}^{-1}}^2 \le C\left(\int_{\R^2} |x||f(x)|\rd x\right)^2 +\|f\|_{L^2}^2.
    \end{align*}
\end{lemma}

\begin{proof}
    The Fourier transform of $f$ is bounded and satisfies
    \begin{align*}
        &\hat{f}(0)= \int_{\R^2} f(x)\rd x = 0,\\
        &\sup_{n\in \R^2}|\nabla \hat{f}(n)| \le 2\pi \int_{\R^2} |x||f(x)| \rd x <\infty.
    \end{align*}
    Therefore $\hat{f}$ is in $W^{1,\infty}(\R^2)$ and satisfies
    \begin{align*}
        |\hat{f}(n)| \le |\hat{f}(0)| +\sup_{n^\prime\in \R^2} |\nabla \hat{f}(n^\prime)||n| \le 2\pi |n| \int_{\R^2} |x||f(x)|\rd x.
    \end{align*}
    Hence we get
    \begin{align*}
        \|f\|_{\dot{H}^{-1}}^2 &=\int_{|n|\le 1} |n|^{-2}|\hat{f}(n)|^2 \rd n +\int_{|n|> 1} |n|^{-2}|\hat{f}(n)|^2 \rd n\\
        &\le C\left(\int_{\R^2} |x||f(x)|\rd x\right)^2+\int_{\R^2} |f(n)|^2 \rd n\\
        &\le C\left(\int_{\R^2} |x||f(x)|\rd x\right)^2 +\|f\|_{L^2}^2.
    \end{align*}
    The proof is complete.
\end{proof}

The next two technical lemmas will be of use in the passage to the limit $\delta\to 0$.

\begin{lemma}\label{lem:Hm1 approx_new}
    Assume \eqref{eq: omega zero wish list}, \eqref{eq:omega zero convergence} and \eqref{eq:omega zero growth} on the initial condition $\omd_0$. Let $\omd$ be the unique solution to equation \eqref{eq: regularized main} given by Lemma \ref{lem: l infty and l 1}. We have
    \begin{align*}
        &\E\left[\sup_{t\in [0,T]} |\|\omd_t\|_{\dot{H}^{-1}}^2 -4\pi^2 \langle \omd_t, G^\delta\ast \omd_t\rangle| \right] \to 0 \text{ as }\delta\to 0.
    \end{align*}
\end{lemma}

\begin{proof}
By Lemma \ref{lem: properties initial condition}, we have (using $1-e^{-a}\le a$ for $a>0$)
\begin{align*}
&|\|\omd_t\|_{\dot{H}^{-1}}^2 - 4\pi^2\langle \omd_t, G^\delta\ast\omd_t \rangle|\\
&= \int |n|^{-2}(1-e^{-4\pi^2\delta|n|^2})|\widehat{\omd}(n)|^2 \rd n +\int |n|^{-2}e^{-4\pi^2|n|^2/\delta}|\widehat{\omd}(n)|^2 \rd n\\
&\le \int 4\pi^2\delta|\hat{\omd}(n)|^2 \rd n +4\pi^2\int e^{-4\pi^2|n|^2/\delta} \left(\int |x||\omd_t| \rd x\right)^2 \rd n\\
&\lesssim \delta \left(\|\omd_t\|_{L^2}^2 +\left(\int |x||\omd_t| \rd x\right)^2\right)
\end{align*}
By Lemmas \ref{lem: l infty and l 1} and \ref{lem: moments}, we obtain
\begin{align*}
    &\E\sup_{t\in [0,T]}\left[|\|\omd_t\|_{\dot{H}^{-1}}^2 - 4\pi^2\langle \omd_t, G^\delta\ast\omd_t \rangle|\right]\\
    &\le \delta \left(\|\omd_0\|_{L^1}\|\omd_0\|_{L^\infty} +\|\omd_0\|_{L^{1}}\int_{\mathbb{R}^2} |x|^2 |\omega_0^{\delta}(x)|\rd  x 
	+ (\| \omega^{\delta}_0 \|_{L^{\infty}}^2 + \| \omega^{\delta}_0 \|_{L^{1}}^2 + 1) \| \omega^{\delta}_0\|_{L^{1}}^2\right)
\end{align*}
We conclude thanks to the assumption \eqref{eq:omega zero growth} on the initial condition.
\end{proof}

\begin{lemma}\label{lem:Hm2 approx}
Assume \eqref{eq: omega zero wish list}, \eqref{eq:omega zero convergence} and \eqref{eq:omega zero growth} on the initial condition $\omd_0$. Let $\omd$ be the unique solution to equation \eqref{eq: regularized main} given by Lemma \ref{lem: l infty and l 1}. There exists $C>0$ such that, $P$-a.s.,
\begin{equation*}
\sup_{t\in[0,T]} \| |\nabla (G- G^\delta)| \ast |\omd_t| \|_{L^\infty}
\le C\delta^{1/4}(\|\omd_0\|_{L^\infty} +\|\omd_0\|_{L^1}).
\end{equation*}

\end{lemma}

\begin{proof}
Using the bound \eqref{eq: convergence first derivative G delta}, we have
\begin{align*}
|\nabla (G - G^{\delta})|\ast |\omd_t|(x) =&  \int |\nabla G -\nabla G^{\delta}(x-y)| |\omd_t(y)| \rd y \\
\lesssim & \int_{|x-y|\leq \delta^{1/4}} |\omd_t(y)| |x-y|^{-1} \rd y 
+ \delta^{3/4}\int_{\delta^{1/4} < |x-y|
< \delta^{-1/4}} |\omd_t(y)|\rd y\\
&+ \int_{|x-y|\geq \delta^{-1/4}} |\omd_t(y)| |x-y|^{-1} \rd y.\\
\lesssim & \delta^{1/4} \|\omd_t\|_{L^\infty}
+ \delta^{3/4} \|\omd_t \|_{L^1}
+ \delta^{1/4} \|\omd_t\|_{L^1}
\end{align*}
Hence, by Lemma \ref{lem: l infty and l 1} we obtain the desired bound.
\end{proof}

\section{A priori estimates on the energy}\label{sec:a priori bd}

In this section, we establish the main a priori estimate on a solution $\omega$ to \eqref{eq:main}, namely we bound the expected value of the energy $\|\omega\|_{\dot{H}^{-1}}^2$. The key point will be the gain of control of the $H^{-\alpha}$ norm of $\omega$ (integrated over time and $\Omega$).

Throughout all this section and the next one, we take a filtered probability space $(\Omega,\mathcal{A},(\mathcal{F}_t)_t,P)$ satisfying the standard assumption and a sequence $(W^k)_k$ of independent real $(\mathcal{F}_t)_t$-Brownian motions. We take $Q^\delta$, $G^\delta$ and $K^\delta$, $\omd_0$ as in Section \ref{sec:regularization}, in particular $\omd_0$ satisfies \eqref{eq: omega zero wish list}, \eqref{eq:omega zero convergence} and \eqref{eq:omega zero growth}. We take $\omd$ to be the solution to the regularized equation \eqref{eq: regularized main}, as from Lemma \ref{lem: l infty and l 1}, on $(\Omega,\mathcal{A},(\mathcal{F}_t)_t,P)$ with Brownian motions $(W^k)_k$.

We start by computing the equation for the approximated energy in the regularized model.

\begin{lemma}\label{lem:Hm1 norm formula}
We have $P$-a.s.: for every $t$,
\begin{align}
\begin{aligned}\label{eq:energy_eqn}
    &\lan\omd_t, G^{\delta}\ast \omd_t\ranL - \int_{0}^{t} \iint_{\mr^2}\tr[(Q^\delta(0)-Q^\delta(x-y))D^2 G^{\delta}(x-y)] \omd_r(x)\omd_r(y) \rd x \rd y \rd r\\
    &= \lan\omd_0, G^{\delta}\ast \omd_0 \ranL -2\int_0^t  \sum_k \lan\sigma_k^\delta\cdot\nabla\omd_r,G^\delta\ast \omd_r\ranL \rd W^k.
\end{aligned}
\end{align}
In particular we have, for every $t$:
\begin{align}
\begin{aligned}
    &\mathbb{E}[\lan\omd_t, G^{\delta}\ast \omd_t\ranL] - \int_{0}^{t} \mathbb{E}\iint_{\mr^2}\tr[(Q^\delta(0)-Q^\delta(x-y))D^2 G^{\delta}(x-y)] \omd_r(x)\omd_r(y) \rd x \rd y \rd r \\
    & \quad = \lan\omd_0, G^{\delta}\ast \omd_0\ranL.\label{eq:energy_eqn_expectation}
    \end{aligned}
\end{align}
\end{lemma}

\begin{proof}
	Consider the function
	\begin{equation*}
		F: H^{-4} \ni \omega \mapsto \lan \omega, G^{\delta}\ast \omega \ranL \in \mr,
	\end{equation*}
	whose Frechet derivatives are
	\begin{align*}
        &DF(\omega)v = \lan \omega,    G^{\delta}\ast v \ranL + \lan v, G^{\delta}\ast \omega \ranL 
        = 2 \lan v, G^{\delta} \ast \omega \ranL \\
        &D^2F(\omega)[v, w] 
        = 2 \lan v, G^{\delta}\ast w \ranL.
	\end{align*}
	We apply It\^o formula (e.g. 
 \cite[Theorem 4.32]{da2014stochastic} on the Hilbert space $H^{-4}$) to $F(\omd)$ and get
    \begin{align}
    \begin{aligned}\label{eq: ito formula energy}
		\rd \lan \omd, G^{\delta}\ast \omd \ranL &= 
		-2 \lan (K^\delta\ast \omd)\cdot \nabla \omd, G^{\delta}\ast \omd \ranL \rd t
		- 2 \sum_k \lan \sigma^\delta_k \cdot \nabla \omd, G^{\delta} \ast \omd \ranL \rd W^k\\
        & \quad + \lan c_\delta\Delta \omd, G^{\delta}\ast \omd \ranL \rd t 
		+ \sum_{k} \lan \sigma^\delta_k \cdot \nabla\omd, G^{\delta}\ast (\sigma^\delta_k \cdot \nabla \omd)\ranL \rd t.
    \end{aligned}
    \end{align}
	
    By the definition of $K^\delta=\nabla^\perp G^\delta$ \eqref{eq: regularization green kernel K}, we have using integration by parts
    \begin{align}
    \begin{aligned}\label{eq:ito formula drift}
        \lan (K^\delta\ast \omd) \cdot \nabla \omd, G^{\delta}\ast \omd \ranL
        &= \lan \diverg[\nabla^{\perp}(G^{\delta}\ast\omd) \cdot \omd], G^{\delta}\ast \omd \ranL\\
        &= - \int_{\mr^2} \omd(x) \nabla^{\perp}(G^{\delta} \ast \omd)(x) \cdot \nabla(G^{\delta}\ast \omd)(x) \rd x = 0,
    \end{aligned}
    \end{align}
    that is the drift term $(K^\delta\ast \omd)\cdot\nabla\omd$ does not contribute to the energy, as expected.
	
	The integrand in the martingale term in \eqref{eq: ito formula energy} verifies the $L^2$ integrability condition: indeed
    \begin{align*}
	   \E \sum_k |\langle \sigma^\delta_k\cdot\nabla \omd, G^\delta \ast \omd \rangle|^2 &= \mathbb{E}\iint_{\R^2\times\R^2}\omd(x)\omd(y) \nabla G^\delta\ast \omd(x) \cdot Q^\delta(x-y) \nabla G^\delta\ast \omd(y) \rd x \rd y\\
	   &\le \|Q^\delta\|_{L^\infty} \|\omd\|_{L^1}^4\|\nabla G^\delta\|_{L^\infty}^2<\infty
    \end{align*}
    Hence the martingale term in \eqref{eq: ito formula energy} vanishes once we take expectation.

    We apply integration by parts to the second-to-last term in \eqref{eq: ito formula energy} and, recalling that $Q^\delta(0)=c_\delta I$, we get
    \begin{align}
    \begin{aligned}\label{eq:ito formula diffusion 1}
		\lan c_\delta\Delta \omd, G^{\delta}\ast \omd \ranL
		= & \lan \omd, c_\delta\Delta G^{\delta} \ast \omd \ranL \\
		= & \iint_{\mr^2\times\mr^2} \tr[Q^\delta(0) D^2G^{\delta}(x-y)]\omd(x)\omd(y) \rd x \rd y.
    \end{aligned}
    \end{align}
	
    We apply integration by parts also on the last term in \eqref{eq: ito formula energy}, using Einstein summation convention over $i,j=1,2$ we get
    \begin{align}
    \begin{aligned}\label{eq:ito formula diffusion 2}
		\sum_k \lan \sigma^{\delta}_k\cdot \nabla \omd, G^{\delta} \ast (\sigma^{\delta}_k \cdot \nabla \omd) \ranL 
		= & \sum_k \lan \partial_{x_i} (\sigma_k^{\delta,i}\omd), \partial_{x_j}G^{\delta} \ast (\sigma^{\delta,j}_k\omd) \ranL\\
		= & - \sum_k \lan \sigma_k^{\delta,i}\omd, \partial^2_{x_i,x_j}G^{\delta} \ast (\sigma^{\delta,j}_k\omd) \ranL\\
		= & -\sum_k \iint_{\mr^2\times\mr^2} \partial_{i,j}^2 G^{\delta}(x-y)\sigma_k^{\delta,i}(x)\sigma_k^{\delta, j}(y)\omd(x)\omd(y) \rd x \rd y \\
		= & - \iint_{\mr^2\times\mr^2} \tr[Q^\delta(x-y)D^2G^{\delta}(x-y)]\omd(x)\omd(y) \rd x \rd y.
    \end{aligned}
    \end{align}
 
    Putting together the equalities \eqref{eq:ito formula drift}, \eqref{eq:ito formula diffusion 1} and \eqref{eq:ito formula diffusion 2} in \eqref{eq: ito formula energy}, we get the first equality \eqref{eq:energy_eqn}. Taking expectation, we obtain the second equality \eqref{eq:energy_eqn_expectation}.
    The proof is complete.
\end{proof}

\begin{corollary}\label{cor:a_priori_1}
We have, for every $t$:
	\begin{align*}
		& \mathbb{E}[\|\omd_t\|_{\dot{H}^{-1}}^2] -4\pi^2\int_{0}^{t} \mathbb{E}\iint_{\mr^2\times\mr^2}\tr[(Q^\delta(0)-Q^\delta(x-y))D^2 G^{\delta}(x-y)] \omd_r(x)\omd_r(y) \rd x \rd y \rd r \\
  & \quad = \|\omega_0\|_{\dot{H}^{-1}}^2 +o(1),
	\end{align*}
	where $o(1)$ tends to $0$ as $\delta\to 0$.
\end{corollary}

\begin{proof}
The bound follows from Lemma \ref{lem:Hm1 norm formula}, by Lemma \ref{lem:Hm1 approx_new} and the fact that $\omd_0\to \omega_0$ in $\dot{H}^{-1}$ as $\delta\to 0$ (and formula \eqref{eq: convergence g delta to H-1}).
\end{proof}

Due to Corollary \ref{cor:a_priori_1}, we need to bound the term $\tr[(Q^{\delta}(0)-Q^\delta)D^2G^\delta]$, which we split as follows, for every $x\neq 0$:
\begin{align}
	\tr&\left[\left(Q^{\delta}(0)-Q^{\delta}(x)\right)D^2G^{\delta}(x)\right]\nonumber\\
    & =  \tr\left[\left(Q(0)-Q(x)\right)D^2G(x)\right] \varphi(x)\nonumber\\
    &\quad + \tr\left[ ((Q^{\delta}(0)-Q^{\delta}(x))-(Q(0)-Q(x)))D^2G^\delta(x)\right]\varphi(x)\nonumber\\
    &\quad + \tr\left[(Q(0)-Q(x))D^2(G^{\delta}-G)(x)\right]\varphi(x)\nonumber\\
    &\quad + \tr[(Q^{\delta}(0)-Q^{\delta}(x))D^2G^{\delta}(x)](1-\varphi(x))\nonumber\\
    & =: A(x) + R1(x) + R2(x) + R3(x),\label{eq:QD2G split}
\end{align}
where $D^2G(x)$ is understood as the pointwise second derivative of $G$ in $x\neq 0$ and $\varphi(x)=\varphi(|x|)$ is a radial $C^\infty$ function satisfying $0\le \varphi\le 1$ everywhere, $\varphi(x)=1$ for $|x|\le 1$ and $\varphi(x)=0$ for $|x|\ge 2$.

The key term in \eqref{eq:QD2G split} is $A$ and the key bound, in Fourier modes, is given by the following:

\begin{lemma}\label{lem:main bound A}
For the term $A=\tr\left[\left(Q(0)-Q\right)D^2G\right] \varphi$, we have
\begin{align*}
\widehat{A}(n) \le - \frac{\alpha\beta_L\gamma_{2\alpha}}{4\pi(2\pi)^{2\alpha}}\langle n\rangle^{-2\alpha} +C\langle n\rangle^{-2},\quad \forall n\in \R^2.
\end{align*}
\end{lemma}

In particular, the bound in terms of $\langle n\rangle^{-2\alpha}$, inserted in Corollary \ref{cor:a_priori_1}, will allow to control the $H^{-\alpha}$ norm of $\omega$.

The main point of the proof of Lemma \ref{lem:main bound A} is the following: thanks to the structure of the covariance matrix $Q$, in particular Proposition \ref{prop:covariance structure}, we get, close to $x=0$ (see \eqref{eq:key_computation}),
\begin{align*}
A(x)\approx - G_\alpha(x)
\end{align*}
and $A$ is smooth away from $0$, hence passing to the Fourier transform we expect (by the theory of Riesz potentials, Lemma \ref{lem:sobolev riesz potential})
\begin{align*}
\widehat{A}(n) \approx -\langle n\rangle^{-2\alpha}.
\end{align*}

\begin{proof}
    The key point is to control $\tr\left[\left(Q(0)-Q\right)D^2G\right]$. Using the structure of the covariance matrix $Q$ from Proposition \ref{prop:covariance structure}, we have, for every $x\neq 0$,
\begin{align}
\begin{aligned}\label{eq:key_computation}
    &\tr\left[
    (Q(0)-Q(x)) D^2G(x)
    \right]\\
    &\quad = - \frac{1}{2\pi}
    \tr \left[
    \left((B_N(0)-B_N(|x|)) I + (B_N(|x|)-B_L(|x|)) \frac{xx^{\top}}{|x|^2}\right)
    \frac{1}{|x|^2}\left(I-\frac{2xx^{\top}}{|x|^2}\right)
    \right] \\
    &\quad  = - \frac{1}{2\pi} \frac{1}{|x|^2}
    \tr \left[
    \left((B_N(0)-B_N(|x|)) I + (B_N(|x|)+B_L(|x|)-2B_N(0)) \frac{xx^{\top}}{|x|^2}\right)
    \right] \\
    &\quad  = - \frac{1}{2\pi} \frac{1}{|x|^2}
    \left[
    2(B_N(0)-B_N(|x|) + (B_N(|x|)+B_L(|x|)-2B_N(0)) 
    \right] \\
    &\quad = - \frac{1}{2\pi} \frac{1}{|x|^2}
    \left(
    B_L(|x|)-B_N(|x|)
    \right) \\
    &\quad = - \frac{1}{2\pi} (\beta_N - \beta_L)|x|^{-2+2\alpha} + \frac{1}{2\pi} |x|^{-2} (\mathrm{Rem}_{1-u^2}(|x|)-\mathrm{Rem}_{u^2}(|x|))\\
    &\quad = - \frac{\gamma_{2\alpha}}{\pi}\alpha\beta_L G_\alpha(x) + \frac{1}{2\pi} |x|^{-2} (\mathrm{Rem}_{1-u^2}(|x|)-\mathrm{Rem}_{u^2}(|x|))
\end{aligned}
\end{align}
where $G_\alpha$ is the Riesz kernel.

Now we control the multiplication of $G_\alpha$ by $\varphi$ (recall that $\varphi$ is defined below formula \eqref{eq:QD2G split}). For $M>0$ we define $\varphi_M(x) := \varphi(x/M)$, for $x\in \R^2$. For $n\neq 0$, by the change of variable $x=y|n|$, calling $\hat{n}:=n/|n|$, we have
\begin{align*}
\widehat{G_\alpha\varphi}(n) &= \gamma_{2\alpha}^{-1} \int_{\R^2} |x|^{-2+2\alpha} \varphi(x) e^{-2\pi ix\cdot n}\rd x \\
&= \gamma_{2\alpha}^{-1} |n|^{-2\alpha} \int_{\R^2} |y|^{-2+2\alpha} \varphi(\frac{y}{|n|}) e^{-2\pi iy\cdot \hat{n}} \rd y
= |n|^{-2\alpha} \widehat{G_\alpha\varphi_{|n|}}(\hat{n})
\end{align*}
Since $G_\alpha \varphi_{M}$ is a radial function (i.e. invariant under rotations), its Fourier transform is also radial and therefore, calling $e_1=(1,0)$,
\begin{align*}
    \widehat{G_\alpha\varphi}(n) = |n|^{-2\alpha} \widehat{G_\alpha\varphi_{|n|}}(\hat{n}) = |n|^{-2\alpha} \widehat{G_\alpha\varphi_{|n|}}(e_1)
\end{align*}
By Lemma \ref{lem:indicator_Fourier}, $\widehat{G_\alpha\varphi_{|n|}}(e_1)\to \widehat{G_\alpha}(e_1) =(2\pi)^{-2\alpha}>0$ as $|n|\to \infty$, so there exists $M_0>0$ such that, for $|n|\ge M_0$,
\begin{align*}
\widehat{G_\alpha\varphi}(n) \ge |n|^{-2\alpha} \frac12 \widehat{G_\alpha}(e_1) = \frac12 (2\pi)^{-2\alpha}|n|^{-2\alpha}.
\end{align*}
Moreover, since $G_\alpha\varphi$ is integrable, $\widehat{G_\alpha\varphi}$ is bounded. Therefore we have, for some constant $C$: for every $n$
\begin{align*}
\widehat{G_\alpha\varphi}(n)\ge \frac12(2\pi)^{-2\alpha}\langle n\rangle^{-2\alpha} -C\langle n\rangle^{-2}.
\end{align*}

Concerning the remainder, taking $0<\epsilon<2-2\alpha$ fixed, by Lemma \ref{lem: properties remainder} we get, for some constant $C>0$,
\begin{align*}
|\mathcal{F}(|\cdot|^{-2} \mathrm{Rem}_{1-u^2}(|\cdot|)\varphi)(n)| + |\mathcal{F}(|\cdot|^{-2} \mathrm{Rem}_{u^2}(|\cdot|)\varphi)(n)| \le C \langle n\rangle^{-2+\epsilon}
\end{align*}
and therefore we get by Young inequality, for $\delta>0$ to be determined later and a constant $C_\delta>0$,
\begin{align*}
|\mathcal{F}(|\cdot|^{-2} \mathrm{Rem}_{1-u^2}(|\cdot|)\varphi)(n)| + |\mathcal{F}(|\cdot|^{-2} \mathrm{Rem}_{u^2}(|\cdot|)\varphi)(n)| \le \delta \langle n\rangle^{-2\alpha} +C_\delta \langle n\rangle^{-2}.
\end{align*}
Choosing now $\delta \le  \alpha\beta_L\gamma_{2\alpha}/(4\pi(2\pi)^{2\alpha})$, we conclude on the term $A$, for some constant $C>0$,
\begin{align*}
\widehat{A}(n) \le -\frac{\alpha\beta_L\gamma_{2\alpha}}{4\pi(2\pi)^{2\alpha}} \langle n\rangle^{-2\alpha} +C\langle n\rangle^{-2}.
\end{align*}
The proof is complete.
\end{proof}

Now we bound the terms $R1=\tr[ ((Q^{\delta}(0)-Q^{\delta})-(Q(0)-Q))D^2G^\delta]\varphi$ and $R2=\tr[(Q(0)-Q)D^2(G^{\delta}-G)]\varphi$ in \eqref{eq:QD2G split}:

\begin{lemma}\label{lem:bound R1 R2}
For any given $0<\epsilon<1\wedge (2\alpha)$, we have (with multiplicative constant depending on $\epsilon$ but not on $\delta$)
\begin{align*}
|R1(x)|&\lesssim \delta^\epsilon |x|^{-2+2\alpha-\epsilon}\varphi(x),\\
|R2(x)|&\lesssim \delta\varphi(x) +|x|^{-2+2\alpha}1_{|x|< (8\delta\log(1/\delta))^{1/2}}.
\end{align*}
\end{lemma}

\begin{proof}
For $R1$, we have
\begin{align*}
    |R1(x)|
    \lesssim |(Q^{\delta}(0)-Q^{\delta}(x))-(Q(0)-Q(x))||D^2G^\delta(x)| \varphi(x).
\end{align*}
As $Q^\delta$ and $Q$ are real-valued, the factor with $Q^\delta-Q$ reads
\begin{align*}
    &(Q^{\delta}(0)-Q^{\delta}(x))-(Q(0)-Q(x)) \\
    &\quad = \int \langle n \rangle^{-2-2\alpha}
    \left(
        I - \frac{nn^{\top}}{|n|^2}
    \right)
    (1-e^{2\pi ix\cdot n})  (1-\widehat{\rho^\delta}(n)) \rd n \\
    & \quad = \int \langle n \rangle^{-2-2\alpha}
    \left(
        I - \frac{nn^{\top}}{|n|^2}
    \right)
    (1-\cos(2\pi x\cdot n))  (1-\widehat{\rho^\delta}(n)) \rd n.
\end{align*}
We have that, for every $a\in\R$ and $2\alpha > \epsilon > 0$,
\begin{align*}
    |1-\cos(a) | 
    \le \frac{1}{2} a^2 \wedge 2 \le 2 a^{2\alpha - \epsilon}.
\end{align*}
Hence we obtain, with a change in polar coordinates
\begin{align*}
    |Q^{\delta}(0)-Q^{\delta}(x))-(Q(0)-Q(x)) |
    &\lesssim |x|^{2\alpha - \epsilon}
    \int \langle n \rangle^{-2-2\alpha}
    |n|^{2\alpha - \epsilon} 1_{|n|\geq 1/\delta} \rd n \\
    &\lesssim |x|^{2\alpha - \epsilon}
    \int |n|^{-2-2\alpha}
    |n|^{2\alpha - \epsilon} 1_{|n|\geq 1/\delta} \rd n \\
    &= |x|^{2\alpha - \epsilon}
    \int_{1/\delta}^{\infty}
    \rho^{-1-\epsilon} \rd \rho \\
    &= \frac{1}{\epsilon} |x|^{2\alpha - \epsilon} \delta^{\epsilon}.
\end{align*}
We use the uniform bound on $D^2G^\delta$ in \eqref{eq: bounds G delta} and get
\begin{align*}
|R1(x)| \lesssim |x|^{2\alpha - \epsilon} \delta^{\epsilon} |x|^{-2}\varphi(x) = \delta^\epsilon |x|^{-2+2\alpha-\epsilon}\varphi(x).
\end{align*}

For the term $R2$, by the H\"older bound \eqref{eq:Q Holder bound} on $Q$ and the bounds \eqref{eq: convergence second derivative G delta} on $D^2(G^\delta-G)$, we get
\begin{align*}
|R2(x)| &\lesssim |Q(0)-Q(x)||D^2G^\delta(x)-D^2G(x)|\\
&\lesssim |x|^{2\alpha} (\delta\varphi(x) + |x|^{-2}1_{|x|< (8\delta\log(1/\delta))^{1/2}})\\
&\lesssim \delta\varphi(x) +|x|^{-2+2\alpha}1_{|x|< (8\delta\log(1/\delta))^{1/2}}.
\end{align*}
The proof is complete.
\end{proof}

Now we bound the term $R3 = \tr[(Q^{\delta}(0)-Q^{\delta})D^2G^{\delta}](1-\varphi)$ in \eqref{eq:QD2G split}, in Fourier modes:

\begin{lemma}\label{lem:bound R3}
For any given $0<\epsilon<1\wedge (2\alpha)$, we have (with multiplicative constant depending on $\epsilon$ but not on $\delta$)
\begin{align*}
|\widehat{R3}(n)| \lesssim |n|^{-1}1_{|n|\le 1} +n^{-2-2\alpha+\epsilon}1_{|n|>1} \lesssim |n|^{-2}.
\end{align*}
\end{lemma}

\begin{proof}
By the bound \eqref{eq: bounds G delta} on $D^2G$, we have, for any $2\le p<\infty$,
\begin{align}
\|\widehat{D^2G^\delta(1-\varphi)}\|_{L^p}\lesssim \|D^2G^\delta(1-\varphi)\|_{L^{p^\prime}} \lesssim \||x|^{-2}1_{|x|\ge 1}\|_{L^{p^\prime}} \lesssim 1\label{eq:hat_G_large_1}
\end{align}
(with multiplicative constant depending possibly on $p$). Moreover, by the bounds \eqref{eq: bounds G delta}, for any integer $k\ge 1$ we have
\begin{align}
|n|^k\|\widehat{D^2G^\delta(1-\varphi)}\|_{L^\infty} \lesssim \|D^{2+k}G^\delta(1-\varphi)\|_{L^1}  \lesssim \||x|^{-2-k} 1_{|x|\ge 1}\|_{L^1} \lesssim 1\label{eq:hat_G_large_2}
\end{align}
(with multiplicative constant depending possibly on $k$).
Taking $k=1$ and $k=4$ in \eqref{eq:hat_G_large_2}, for the term $\tr[Q^\delta(0)D^2G^\delta](1-\varphi)$ we have, for every $n$,
\begin{align}
|\mathcal{F}(\tr[Q^\delta(0)D^2G^\delta](1-\varphi))(n)| \lesssim |Q^\delta(0)||\mathcal{F}(D^2G^\delta(1-\varphi))(n)| \lesssim |n|^{-1} \wedge |n|^{-4}.\label{eq: term C 1}
\end{align}
For the term $\tr[Q^\delta D^2G^\delta](1-\varphi)$, we have
\begin{align*}
|\mathcal{F}(\tr[Q^\delta D^2G^\delta](1-\varphi))(n)| &= |\tr[\widehat{Q^\delta}\ast \widehat{D^2G^\delta (1-\varphi)}](n)|\\
&\le \int |\widehat{Q^\delta}(n-n^\prime)||\widehat{D^2G^\delta (1-\varphi)}(n^\prime)| \rd n^\prime \\
& = \int_{|n^\prime|\le R} \ldots + \int_{|n^\prime|> R} \ldots,
\end{align*}
for a suitable $R > 0$ to be determined later. For the integral on $|n^\prime|\le R$, we use the bound \eqref{eq:hat_G_large_1} with $p=2$, while for the integral on $|n^\prime|>R$ we use the bound \eqref{eq:hat_G_large_2}. We get, for any fixed integer $k\ge 3$, (calling $a^+=a\vee 0$ and $\langle a \rangle = (1+a^2)^{1/2}$ for $a\in\mathbb{R}$)
\begin{align*}
|\mathcal{F}(\tr[Q^\delta D^2G^\delta](1-\varphi))(n)| &\lesssim \int_{|n^\prime|\le R} \langle n-n^\prime \rangle^{-2-2\alpha} |\widehat{D^2G^\delta (1-\varphi)}(n^\prime)| \rd n^\prime \\
& \quad +\int_{|n^\prime|> R} |\widehat{D^2G^\delta (1-\varphi)}(n^\prime)| \rd n^\prime \\
&\lesssim \langle (|n|-R)^+\rangle^{-2-2\alpha} \int_{|n^\prime|\le R} |\widehat{D^2G^\delta (1-\varphi)}(n^\prime)| \rd n^\prime \\
&\quad +\int_{|n^\prime|> R} |n^\prime|^{-k} \rd n^\prime\\
&\lesssim \langle (|n|-R)^+\rangle^{-2-2\alpha} R \|\widehat{D^2G^\delta (1-\varphi)}\|_{L^2} +R^{2-k}\\
&\lesssim \langle (|n|-R)^+\rangle^{-2-2\alpha} R +R^{2-k}
\end{align*}
(the multiplicative constant depending possibly on $k$ but not on $n$ nor on $R$).
Now we take $0<\epsilon<1\wedge (2\alpha)$, $R=\langle n\rangle^\epsilon$ and $k>2+(2+2\alpha)/\epsilon$. We get, for every $n$,
\begin{align}
|\mathcal{F}(\tr[Q^\delta D^2G^\delta](1-\varphi))(n)| &\lesssim \langle (|n|-\langle n\rangle^\epsilon)^+ \rangle^{-2-2\alpha+\epsilon} +\langle n\rangle^{-2-2\alpha} \lesssim \langle n\rangle^{-2-2\alpha+\epsilon}.\label{eq: term C 2}
\end{align}
(with multiplicative constant depending on $\epsilon$). Putting together the bounds \eqref{eq: term C 1} and \eqref{eq: term C 2}, we get
\begin{align*}
|\widehat{R3}(n)| \lesssim |n|^{-1} \wedge |n|^{-4} +\langle n \rangle^{-2-2\alpha+\epsilon} \lesssim |n|^{-1}1_{|n|\le 1} +n^{-2-2\alpha+\epsilon}1_{|n|>1} \lesssim |n|^{-2}.
\end{align*}
The proof is complete.
\end{proof}

We can put together the bound on the key term in Lemma \ref{lem:main bound A} and the bounds on the remainder terms in Lemmas \ref{lem:bound R1 R2} and \ref{lem:bound R3}:
\begin{lemma}\label{lem:bound_stoch_term}
There exist $c>0$, $C\ge 0$ such that $P$-a.s. we have: for every $t\in [0,T]$,
\begin{align*}
&\iint_{\R^2\times\R^2}\tr[(Q^\delta(0)-Q^\delta(x-y))D^2 G^{\delta}(x-y)] \omd_t(x)\omd_t(y) \rd x \rd y \\
&\le -c\|\omd_t\|_{H^{-\alpha}}^2 +C\|\omd_t\|_{\dot{H}^{-1}}^2 + o(1),
\end{align*}
where $o(1)$ tends to $0$ as $\delta\to 0$ uniformly on $[0,T]\times \Omega$.
\end{lemma}

\begin{proof}
Using the Fourier isometry and the Fourier transform of convolutions, we write
\begin{align*}
&\iint_{\R^2\times \R^2}\tr[(I-Q^\delta(x-y))D^2 G^{\delta}(x-y)] \omd(x)\omd(y) \rd x \rd y\\
&= \int_{\R^2} (\widehat{A}(n)+\widehat{R3}(n))|\widehat{\omd}(n)|^2 \rd n +\iint_{\R^2\times\R^2} (|R1(x-y)|+|R2(x-y)|)|\omd(x)||\omd(y)| \rd x \rd y.
\end{align*}
By Lemmas \ref{lem:main bound A} and \ref{lem:bound R3}, we obtain
\begin{align}
\begin{aligned}\label{eq:bound_stoch_term_1}
&\int (\widehat{A}(n)+\widehat{R3}(n))|\widehat{\omd}(n)|^2 \rd n\\
&\le -c\int \langle n \rangle^{-2\alpha} |\widehat{\omd}(n)|^2 \rd n +C\int |n|^{-2} |\widehat{\omd}(n)|^2 \rd n = -c\|\omd\|_{H^{-\alpha}}^2 +C\|\omd\|_{\dot{H}^{-1}}^2.
\end{aligned}
\end{align}
We take $\epsilon=4\bar{\epsilon}$, where $\bar\epsilon\in (0,(1\wedge(2\alpha))/4)$ is the exponent in \eqref{eq:omega zero growth}. By Lemma \ref{lem:bound R1 R2}, we have
\begin{align}
\begin{aligned}\label{eq:bound_stoch_term_2}
&\int (|R1(x-y)|+|R2(x-y)|)|\omd(x)||\omd(y)| \rd x \rd y\\
&\le (\|R1\|_{L^1}+\|R2\|_{L^1})\|\omd\|_{L^\infty}\|\omd\|_{L^1}\\
&\lesssim \int \delta^\epsilon |x|^{-2+2\alpha-\epsilon}\varphi(x) +\delta\varphi(x) +|x|^{-2+2\alpha}1_{|x|<(8\delta\log(1/\delta))^{1/2}} \rd x \  \|\omd_0\|_{L^\infty}\|\omd_0\|_{L^1}\\
&\lesssim \delta^{\epsilon/4} \int (1+|x|^{-2+2\alpha-\epsilon})\varphi(x)\rd x \ \|\omd_0\|_{L^\infty}\|\omd_0\|_{L^1}\\
&\lesssim \delta^{\bar\epsilon} \|\omd_0\|_{L^\infty}\|\omd_0\|_{L^1} \lesssim o(1).
\end{aligned}
\end{align}
Putting together the bounds \eqref{eq:bound_stoch_term_1} and \eqref{eq:bound_stoch_term_2}, we get the desired estimate.
\end{proof}

We are now able to give the main a priori estimate on $\omega$:
\begin{proposition}
\label{prop:a_priori_bd}
There exist $c>0$, $C\ge 0$ such that, for every $\delta>0$ (small),
\begin{align*}
\sup_{t\in [0,T]}\E[\|\omd_t\|_{\dot{H}^{-1}}^2] +c \int_0^T \E[\|\omd_r\|_{H^{-\alpha}}^2] \rd r \le C\|\omega_0\|_{\dot{H}^{-1}}^2 +o(1),
\end{align*}
where $o(1)$ tends to $0$ as $\delta\to 0$.
\end{proposition}

\begin{proof}
By Corollary \ref{cor:a_priori_1} and Lemma \ref{lem:bound_stoch_term} (integrated on $[0,T]\times\Omega$), we get, for every $t$,
\begin{align*}
\E[\|\omd_t\|_{\dot{H}^{-1}}^2] +c \int_0^t \E[\|\omd_r\|_{H^{-\alpha}}^2] \rd r \le \|\omega_0\|_{\dot{H}^{-1}}^2 +C\int_0^t \E[\|\omd_r\|_{\dot{H}^{-1}}^2] \rd r +o(1).
\end{align*}
The conclusion follows applying Gr\"onwall lemma.
\end{proof}

\section{Bounds on time continuity}
\label{sec: continuity bounds}

In this section, we establish the a priori bounds on time continuity for a solution $\omega$ to \eqref{eq:main}, namely we bound the expected value of $\|\omega\|_{C_t^\gamma(\dot{H}^{-4})}^2$ for some $0<\gamma<1/2$.

As in the previous section, $Q^\delta$, $G^\delta$ and $K^\delta$, $\omd_0$ satisfy the conditions in Section \ref{sec:regularization} and $\omd$ is the solution to the regularized equation \eqref{eq: regularized main}.

We start with a bound on $\sup_{t\in [0,T]} \|\omd_t\|_{\dot{H}^{-1}}$:

\begin{lemma}\label{lem:Hm1_bd_sup}
There exists $C>0$ such that, for every $\delta>0$,
    \begin{align*}
        \E[\sup_{t\in [0,T]} \|\omd_t\|_{\dot{H}^{-1}}^2] \le C.
    \end{align*}
\end{lemma}

\begin{proof}
    We use formula \eqref{eq:energy_eqn} for $\langle \omd_t, G^{\delta}\ast \omd_t\rangle$, we use Lemma \ref{lem:bound_stoch_term}, we take the supremum over $[0,T]$ and then take expectation (and use Lemma \ref{lem:Hm1 approx_new}):
    \begin{align}
    \begin{aligned}\label{eq:energy_bd_sup}
    \E[\sup_{t\in[0,T]} \|\omd_t\|_{\dot{H}^{-1}}^2] &\le  \|\omega_0\|_{\dot{H}^{-1}}^2	+2 \E\left[\sup_{t\in[0,T]}\left|\int_0^t\sum_k \lan \sigma^\delta_k \cdot \nabla \omd_r, G^{\delta} \ast \omd_r \ranL \rd W^k_r\right| \right]\\
    &\quad + C\int_0^T \E[\|\omd_r\|_{\dot{H}^{-1}}^2] \rd r +o(1),
    \end{aligned}
    \end{align}
    where $o(1)$ tends to $0$ as $\delta\to 0$. We use Burkholder-Davis-Gundy inequality for the martingale term:
    \begin{align}
    \label{eq: martingale term}
        &\E\left[ 
        \sup_{t\in[0,T]} \left|\sum_{k} \int_{0}^{t}\langle G^\delta \ast \omd_r, \sigma_k^{\delta}  \cdot \nabla \omd_r\rangle \rd W_r^k\right|
        \right]
        \le 
        \E \left[ \left(
            \int_{0}^{T} \sum_k|\langle G^{\delta}\ast \omd, \sigma_k^{\delta}\ast \nabla \omd\rangle |^2 \rd r
            \right)^{\frac{1}{2}}
        \right]\nonumber\\
        &\quad \leq \E \left[ \left(
            \int_{0}^{T} \sum_k|\langle G\ast \omd, \sigma_k^{\delta}\cdot \nabla \omd\rangle |^2 \rd r
            \right)^{\frac{1}{2}}
        \right] 
        + \E \left[ \left(
            \int_{0}^{T} \sum_k|\langle (G^{\delta}-G)\ast \omd, \sigma_k^{\delta}\cdot \nabla \omd\rangle |^2 \rd r
            \right)^{\frac{1}{2}}
        \right] \nonumber\\
        & =:M1+MR
    \end{align}
    Note the following equality, in the distributional sense: calling $\ud=K\ast \omd$ (where $K=\nabla^\perp G$), we have $\omd = \operatorname{curl}\ud$ and
    \begin{align*}
    \sigma_k^\delta\cdot \nabla \omd = \operatorname{curl}((\sigma_k^\delta \cdot \nabla)\ud + (D\sigma_k^\delta)^{\top}\ud),
    \end{align*}
    since $\operatorname{div}(\sigma_k)=0$. We use this fact to expand the integrand in the first term $M1$ on the right hand side of \eqref{eq: martingale term}. By integration by parts we get
    \begin{align*}
        |\langle G\ast \omd, \sigma_k^{\delta}\cdot \nabla \omd\rangle |^2
        = | \langle \nabla^{\perp} G\ast \omd, (\sigma_k^{\delta} \cdot \nabla)\ud + (D\sigma_k^{\delta})^{\top}\ud\rangle|^2
        = |\langle \ud, (D\sigma_k^{\delta})^{\top}\ud\rangle|^2,
    \end{align*}
    where in the last equality we used that $\sigma_k$ is a divergence-free vector field and so
    \begin{align*}
        \langle \ud, (\sigma_k^{\delta} \cdot \nabla) \ud\rangle
        = \frac{1}{2} \int \sigma_k^{\delta} \cdot \nabla (|\ud|^2) \rd x 
        = 0.
    \end{align*}
    We have now that (using Einstein's summation convention),
    \begin{align*}
        \sum_k |\langle \ud, (D\sigma_k^{\delta})^{\top}\ud\rangle|^2
        &= \iint \sum_k u^{\delta,i}(x) u^{\delta,j}(x) \partial_{i} \sigma_k^{\delta,j}(x) u^{\delta,l}(y) u^{\delta,m}(y) \partial_l \sigma_k^{\delta,m}(y) \rd x\rd y \\
        &= \iint u^{\delta,i}(x) u^{\delta,j}(x) u^{\delta,l}(y) u^{\delta,m}(y) \partial_{i}\partial_{l} Q^{\delta}_{jm}(x-y) \rd x\rd y
        \\
        &= \langle u^{\delta,i} u^{\delta,j}, \partial_i \partial_l Q^{\delta}_{jm} \ast u^{\delta,l} u^{\delta,m} \rangle
        \\
        &= \langle \widehat{u^{\delta,i} u^{\delta,j}}, \overline{\widehat{\partial_i \partial_l Q^{\delta}_{jm} } \widehat{u^{\delta,l} u^{\delta,m}}} \rangle
        \\
        &\lesssim \max_{i,j} \int |\widehat{u^{\delta,i}u^{\delta,j}}(k)|^2
        |k|^2 |\widehat{Q^{\delta}}(k)| \rd k
        \\
        &\lesssim \max_{i,j} \int |\widehat{u^{\delta,i} u^{\delta,j}}(k)|^2
        \langle k\rangle^{-2\alpha} \rd k \approx \| \ud (\ud)^{T} \|_{H^{-\alpha}}^2.
    \end{align*}
    
    We apply first Sobolev embedding, precisely \cite[Corollary 1.39]{bahouri2011fourier} with $s=-\alpha$ and $p=2/(1+\alpha)$, getting
    \begin{align*}
        \| \ud (\ud)^{\top} \|_{H^{-\alpha}} &\lesssim \| \ud (\ud)^{\top} \|_{L^{2/(1+\alpha)}}
        \approx \| |\ud|^2 \|_{L^{2/(1+\alpha)}}.
    \end{align*}
    We apply now H\"{o}lder inequality (with exponents $r=2/(1+\alpha)$, $p=2$ and $q=2/\alpha$) and again Sobolev embedding, see \cite[Theorem 1.38]{bahouri2011fourier} with $s=1-\alpha$, and we get
    \begin{equation*}
        \| |\ud|^2 \|_{L^{2/(1+\alpha)}} 
        \leq \| \ud \|_{L^2}\cdot \| \ud \|_{L^{2/\alpha}}
        \lesssim  \| \ud \|_{L^2}\cdot \| \ud \|_{\dot{H}^{1-\alpha}}.
    \end{equation*}

    Putting all together and using Young inequality we have, for $\epsilon>0$ to be determined later and a constant $c_\epsilon>0$,
    \begin{align*}
        \E \left[ \left(
            \int_{0}^{T} \sum_k|\langle G\ast \omd, \sigma_k^{\delta}\cdot \nabla \omd\rangle |^2 \rd r
            \right)^{\frac{1}{2}}
        \right] 
        &\leq \E \left[ \left(
            \int_{0}^{T} \| \ud_r \|_{L^2}^2\cdot \| \ud_r \|_{\dot{H}^{1-\alpha}}^2 \rd r
            \right)^{\frac{1}{2}}
        \right] 
        \\
        &\leq \E \left[
        \sup_{t\in[0,T]}\|\ud_t\|_{L^2}
        \left(\int_{0}^T \|\ud_r \|_{\dot{H}^{1-\alpha}}^2 \rd r\right)^{\frac{1}{2}}
        \right]
        \\
        &\leq \epsilon \E \left[\sup_{t\in[0,T]}\|\ud_t\|_{L^2}^2\right]
        + c_{\epsilon} \E \left[\int_{0}^T \|\ud_r \|_{\dot{H}^{1-\alpha}}^2 \rd r\right].
    \end{align*}
    The last term in the right hand side can be bounded in terms of $\omd$ as
    \begin{align*}
        \|\ud \|_{L^2}^2 &= \int |\widehat{\ud}|^2 \rd n
        = (2\pi)^{-2}\int |\widehat{\omd}|^2 |n|^{-2} \rd n = (2\pi)^{-2}\|\omd \|_{\dot{H}^{-1}}^2\\
        \|\ud \|_{\dot{H}^{1-\alpha}}^2
        &= \int |\widehat{\ud}|^2 |n|^{2(1-\alpha)} \rd n
        = (2\pi)^{-2}\int |\widehat{\omd}|^2 |n|^{-2\alpha} \rd n
        \\
        &\lesssim \int_{|n|\leq 1} \frac{1}{|n|^2} |\widehat{\omd}|^2 \rd n
        + \int_{|n| > 1} \langle n \rangle^{-2\alpha} |\widehat{\omd}|^2\rd n
        = \| \omd \|_{\dot{H}^{-1}}^2 + \| \omd \|_{H^{-\alpha}}^2.
    \end{align*}
    We finally have
\begin{align}
        M1 \lesssim \epsilon \E \left[\sup_{t\in[0,T]}\|\omd_t\|_{\dot{H}^{-1}}^2\right]
        + C_\epsilon \int_0^T \E[\| \omd_r \|_{\dot{H}^{-1}}^2] \rd r + C_{\epsilon}
        \int_{0}^T  \mathbb{E}[\| \omd_r\|_{H^{-\alpha}}^2 ]\rd r.\label{eq:martingale bd 1}
    \end{align}

Now we bound the term $MR$:
\begin{align*}
MR
& = \E \left[ \left( \int_{0}^{T} \sum_k|\langle \nabla(G^{\delta}-G)\ast \omd, \sigma_k^{\delta}\omd\rangle |^2 \rd r \right)^{\frac{1}{2}} \right]\\
&= \E \biggl[ \Big( \int_{0}^{T} \iiiint \partial_i(G-G^\delta)(x-y) \partial_j(G-G^\delta)(w-z) \omd(x)\omd(y)\omd(w)\omd(z) \\
& \quad \cdot \sum_k\sigma^\delta_i(x)\sigma^\delta_j(w) \rd x\rd y\rd w\rd z \rd r \Big)^{\frac{1}{2}} \biggl]\\
&= \E \biggl[ \Big( \int_{0}^{T} \iiiint \partial_i(G-G^\delta)(x-y) \partial_j(G-G^\delta)(w-z) \omd(x)\omd(y)\omd(w)\omd(z) \\
&\quad \cdot Q^\delta(x-w) \rd x\rd y\rd w\rd z \rd r \Big)^{\frac{1}{2}} \biggl].
\end{align*}
Since $Q^\delta$ is uniformly bounded (see \eqref{eq: properties Q delta}), we have
\begin{align*}
MR
&\lesssim\E \left[ \left( \int_{0}^{T} \left(\iint |\partial_i(G-G^\delta)(x-y)| |\omd(x)||\omd(y)| \rd x\rd y\right)^2 \rd r \right)^{\frac{1}{2}} \right]\\
&\lesssim \E \left[ \sup_{t\in[0,T]} \iint |\nabla(G-G^\delta)(x-y)| |\omd(x)||\omd(y)| \rd x\rd y \right]\\
&\le \E \left[\sup_{t\in[0,T]} \|\omd\|_{L^1} \| |\nabla(G-G^\delta)|\ast |\omd| \|_{L^\infty}\right]
\end{align*}
By Lemmas \ref{lem: l infty and l 1} and Lemma \ref{lem:Hm2 approx} and the condition \eqref{eq:omega zero growth} on $\omd_0$, we conclude that, as $\delta\to 0$,
\begin{align}
MR \lesssim \delta^{1/4} \|\omd_0\|_{L^1} (\|\omd_0\|_{L^\infty} +\|\omd_0\|_{L^1}) = o(1).\label{eq:martingale bd 2}
\end{align}

We plug the estimates \eqref{eq:martingale bd 1} and \eqref{eq:martingale bd 2} on the martingale term in the bound \eqref{eq:energy_bd_sup}, getting, for some $\bar{C}>0$ (independent of $\epsilon$) and some $C_\epsilon>0$,
\begin{align*}
\E[\sup_{t\in [0,T]} \|\omd_t\|_{\dot{H}^{-1}}^2 ] 
&\le \|\omega_0\|_{\dot{H}^{-1}}^2 +\epsilon \bar{C} \E[ \sup_{t\in[0,T]}\|\omd_t\|_{\dot{H}^{-1}}^2] 
+ C_\epsilon \int_0^T \E[\|\omd_r\|_{\dot{H}^{-1}}^2 ] \rd r \\ &\quad +C_\epsilon \int_0^T \E[\|\omd_r\|_{H^{-\alpha}}^2 ] \rd r
+ o(1).
\end{align*}
We take $\epsilon>0$ such that $\epsilon \bar{C}\le 1/2$ and we get
\begin{align*}
E[\sup_{t\in [0,T]} \|\omd_t\|_{\dot{H}^{-1}}^2 ] \le 2\|\omega_0\|_{\dot{H}^{-1}}^2  
+ C \int_0^T \E[\|\omd_r\|_{\dot{H}^{-\alpha}}^2 ] \rd r  
+ C \int_0^T \E[\|\omd_r\|_{\dot{H}^{-\alpha}}^2 ] \rd r +o(1).
\end{align*}
By Proposition \ref{prop:a_priori_bd}, the above right-hand side is bounded uniformly in $\delta$. The proof is complete.
\end{proof}

\begin{corollary}\label{cor:exit_time}
For $R>0$, call
\begin{align*}
\tau^\delta_R := \inf\{t\ge 0\mid \|\omd_t\|_{\dot{H}^{-1}} \ge R\} \wedge T
\end{align*}
the first exit time of $\omd$ from the $\dot{H}^{-1}$ ball of radius $R$. Then
\begin{align*}
P\{\tau^\delta_R <T\} \le \frac{C}{R^2}.
\end{align*}
\end{corollary}

\begin{proof}
By Markov inequality, we have
\begin{align*}
P\{\tau^\delta_R <T\} \le P\{\sup_{t\in [0,T]}\|\omd_t\|_{\dot{H}^{-1}}^2 \ge R^2\} \le \frac{C}{R^2}.
\end{align*}
The proof is complete.
\end{proof}

Now we show the uniform in time continuity of $\omd$, with high probability. We call
\begin{align*}
\omdR_t :=\omd_{t\wedge \tau^R_\delta}.
\end{align*}
We introduce a modified (mixed homogeneous-inhomogeneous) $H^{-4}$ norm of a (scalar) Schwartz distribution $f$, which is the $H^{-3}$ norm of the associated velocity field $K*f$:
\begin{align*}
\|f\|_{\tilde{H}^{-4}}^2 := \int_{\R^2} \langle n\rangle^{-6} |n|^{-2} |\hat{f}(n)|^2 \rd n = (2\pi)^2\|K*f\|_{H^{-3}}^2.
\end{align*}
The space $\tilde{H}^{-4}$, defined by all Schwartz distributions with finite $\tilde{H}^{-4}$ norm, can be identified with the space of divergence-free $H^{-3}$ vector fields and hence it is a separable Hilbert space.
Note that
\begin{align*}
\|\nabla f\|_{\tilde{H}^{-4}}^2 = (2\pi)^2\int_{\R^2} \langle n\rangle^{-6} |n|^{-2}|n|^2 |\hat{f}(n)|^2 \rd n = (2\pi)^2\|f\|_{H^{-3}}^2.
\end{align*}

\begin{lemma}
\label{lem: C gamma of H minus 3}
For every $0<\gamma<1/2$, $p\ge 2$, $R>0$, there exists $C=C_{\gamma,p,R}>0$ such that, for every $\delta>0$,
\begin{align*}
\E\left[ \|\omdR\|_{C^\gamma_t(\tilde{H}^{-4})}^p \right] \le C.
\end{align*}
\end{lemma}

\begin{proof}
We estimate the $\tilde{H}^{-4}$ norm of $\omdR_t-\omdR_s$ and we apply Burkholder-Davis-Gundi inequality to the stochastic integral, getting
\begin{align*}
\E\left[\|\omdR_t-\omdR_s\|_{H^{-4}}^p\right]
&\lesssim \E\left[\left\| \int_s^t (K^\delta\ast \omdR) \cdot \nabla \omdR \rd r\right\|_{\tilde{H}^{-4}}^p\right]\\
&\quad +\E\left[\left\| \sum_k \int_s^t \sigma^\delta_k \cdot \nabla \omdR \rd W^k\right\|_{\tilde{H}^{-4}}^p\right] +\E\left[\left\| \int_s^t c_\delta \Delta\omdR \rd r\right\|_{\tilde{H}^{-4}}^p\right]\\
&\lesssim \E\left[\left\| \int_s^t (K^\delta\ast \omdR) \cdot \nabla \omdR \rd r\right\|_{\tilde{H}^{-4}}^p\right]\\
&\quad +\E\left[\left(\int_s^t \sum_k \|\sigma^\delta_k \cdot \nabla \omdR \|_{\tilde{H}^{-4}}^2\rd r\right)^{p/2}\right] +\E\left[\left\| \int_s^t c_\delta \Delta\omdR \rd r\right\|_{\tilde{H}^{-4}}^p\right]\\
&=: S_1+S_2+S_3.
\end{align*}
For the term $S_1$, we actually consider the term $\tilde{S_1}$, obtained replacing $K^\delta=\nabla^\perp G^\delta$ with $K=\nabla^\perp G$ in the expression for $S_1$, and claim that, as $\delta\to 0$,
\begin{align}
S_1\lesssim \tilde{S_1} +(t-s)^p o(1). \label{eq:AmAtilde}
\end{align}
We leave the proof of this (technical) inequality at the end of the proof and we focus on the bound on $\tilde{S_1}$. Calling $u^{\delta,R}=K\ast \omdR$ and using the equality \eqref{eq: curl div}, we have that
\begin{align*}
\tilde{S_1}&= \E\left[\left\| \int_s^t (K\ast \omdR) \cdot \nabla \omdR \rd r\right\|_{\tilde{H}^{-4}}^p\right]\\
&\le (t-s)^p \E\left[\sup_{r\in[0,T]}\left\| (K\ast \omdR) \cdot \nabla \omdR \right\|_{\tilde{H}^{-4}}^p\right]\\
&\le (t-s)^p \E\left[\sup_{r\in[0,T]}\left\| \text{curl}\,\diverg (u^{\delta,R} (\tilde{u}^{\delta,R})^{\top})\right\|_{\tilde{H}^{-4}}^p\right]\\
&\lesssim (t-s)^p \E\left[\sup_{r\in[0,T]}\| u^{\delta,R} (u^{\delta,R})^\top\|_{H^{-2}}^p\right].
\end{align*}
By Sobolev embedding of $L^1$ in $H^{-2}$ (e.g. \cite[Proposition 2.71]{bahouri2011fourier} and \cite[Section 2.3.2, Proposition 2 and Section 2.5.7, formula (1)]{triebel1983function}), we have
\begin{align*}
\|u^{\delta,R} (u^{\delta,R})^{\top}\|_{H^{-2}} \lesssim \|u^{\delta,R} (u^{\delta,R})^{\top}\|_{L^1} \lesssim \|u^{\delta,R}\|_{L^2}^2 = (2\pi)^{-2}\|\omdR\|_{\dot{H}^{-1}}^2.
\end{align*}
Hence we have
\begin{align}
\tilde{S_1}\lesssim (t-s)^p \E\left[\sup_{r\in [0,T]}\| \omdR\|_{\dot{H}^{-1}}^{2p}\right] \le (t-s)^p R^{2p}.\label{eq:equicont_A}
\end{align}
For the term $S_2$, we have
\begin{align*}
S_2&=\E\left[\left(\int_s^t \sum_k \|\diverg (\sigma^\delta_k \omdR) \|_{\tilde{H}^{-4}}^2\rd r\right)^{p/2}\right]\\
&\le \E\left[\left(\int_s^t \sum_k \|\sigma^\delta_k \omdR\|_{H^{-3}}^2\rd r\right)^{p/2}\right] \le \E\left[\left(\int_s^t \sum_k \|\sigma^\delta_k \omdR\|_{H^{-2}}^2\rd r\right)^{p/2}\right]\\
&\le (t-s)^{p/2} \E\left[\sup_{r\in [0,T]} \left(\sum_k \|\sigma^\delta_k \omdR\|_{H^{-2}}^2\right)^{p/2}\right].
\end{align*}
Exploiting the Fourier transform, the integrand reads as
\begin{align}
\sum_k \|\sigma^\delta_k \omdR\|_{H^{-2}}^2 &= \sum_k \int \langle n \rangle^{-4}|\widehat{\sigma^\delta_k\omdR}(n)|^2 \rd n\nonumber\\
&= \sum_k \int \langle n \rangle^{-4} \iint \sigma^\delta_k(x)\omdR(x)e^{-2\pi ix\cdot n}\cdot \sigma^\delta_k(y)\omdR(y)e^{2\pi iy\cdot n} \rd x\rd y \rd n\nonumber\\
&= \iint \omdR(x)\omdR(y) \tr Q^\delta(x-y) \int \langle n \rangle^{-4} e^{-2\pi i(x-y)\cdot n} \rd n \rd x\rd y\nonumber\\
&= \int\omdR(x)\psi\ast \omdR(x) \rd x\nonumber\\
&= \int |\widehat{\omdR}(n)|^2 \widehat{\psi}(n) \rd n,\label{eq:Hm2 omega Q}
\end{align}
where we have called
\begin{align*}
\psi(x) = \tr Q^\delta(x) \int \langle n \rangle^{-4} e^{-2\pi ix\cdot n} \rd n.
\end{align*}
For the Fourier transform of $\psi$ we have, for every $n$,
\begin{align}
\widehat{\psi}(n) &= \int \tr\widehat{Q^\delta}(n-n^\prime)\cdot  \langle -n^\prime \rangle^{-4} \rd n^\prime\nonumber\\
&\le \int \langle n-n^\prime \rangle^{-2-2\alpha} \langle n^\prime \rangle^{-4} \rd n^\prime\nonumber\\
&= \int_{|n^\prime|\le |n|/2} \langle n-n^\prime \rangle^{-2-2\alpha} \langle n^\prime \rangle^{-4} \rd n^\prime +\int_{|n^\prime|> |n|/2} \langle n-n^\prime \rangle^{-2-2\alpha} \langle n^\prime \rangle^{-4} \rd n^\prime\nonumber\\
&\lesssim \int_{|n^\prime|\le |n|/2} \langle |n|/2 \rangle^{-2-2\alpha} \langle n^\prime \rangle^{-4} \rd n^\prime +\int_{|n^\prime|> |n|/2} \langle n^\prime \rangle^{-4} \rd n^\prime \nonumber\\
&\lesssim \langle n \rangle^{-2-2\alpha} \int \langle n^\prime \rangle^{-4} \rd n^\prime +\int_{|n^\prime|> |n|/2} \langle n^\prime \rangle^{-4} \rd n^\prime \nonumber\\
&\lesssim \langle n \rangle^{-2-2\alpha} +\langle n \rangle^{-2} \lesssim \langle n \rangle^{-2}. \label{eq:psi Fourier bd}
\end{align}
Therefore we get
\begin{align}
\label{eq: sum sigma H minus 2}
\sum_k \|\sigma^\delta_k \omdR\|_{H^{-2}}^2 = \int |\widehat{\omdR}(n)|^2 \widehat{\psi}(n) \rd n
\lesssim \int |\widehat{\omdR}(n)|^2 \langle n \rangle^{-2} \rd n = \|\omdR\|_{H^{-1}}^2
\end{align}
and so
\begin{align}
S_2\lesssim (t-s)^{p/2}\E\left[\sup_{r\in [0,T]}\|\omdR\|_{H^{-1}}^p\right] \le (t-s)^{p/2} R^p.\label{eq:equicont_B}
\end{align}
For the term $S_3$, we have (recall that, by \eqref{eq:Q delta 0}, $c_\delta$ is uniformly bounded)
\begin{align}
S_3&=\E\left[\left\| \int_s^t c_\delta \Delta\omdR \rd r\right\|_{\tilde{H}^{-4}}^p\right] \nonumber\\
&\le (t-s)^p \E\left[\sup_{r\in [0,T]}\left\| c_\delta \Delta\omdR\right\|_{\tilde{H}^{-4}}^p\right] \nonumber\\
&\lesssim (t-s)^p \E\left[\sup_{r\in [0,T]}\| \omdR\|_{H^{-2}}^p\right] \le (t-s)^p R^p.\label{eq:equicont_C}
\end{align}
Putting together the bounds \eqref{eq:AmAtilde}, \eqref{eq:equicont_A}, \eqref{eq:equicont_B} and \eqref{eq:equicont_C}, we arrive at
\begin{align*}
\sup_{0\le s<t\le T}\E\left[\|\omdR_t-\omdR_s\|_{\tilde{H}^{-4}}^p\right] \lesssim (t-s)^{p/2}(1+R^{2p}),
\end{align*}
in particular, for every $0<\beta<1/2$,
\begin{align*}
\E\left[\int_0^T\int_0^T\frac{\|\omdR_t-\omdR_s\|_{\tilde{H}^{-4}}^p}{|t-s|^{\beta p+1}}\right] \lesssim 1+R^{2p}.
\end{align*}
By the Kolmogorov criterion, or the Sobolev embedding $W^{\beta,p}([0,T])\hookleftarrow C^\gamma([0,T])$ for $\gamma<\beta-1/p$ (see e.g. [DaPrato-Zabczik, Ergodicity for infinite dimensional systems, Theorem B.1.5]), we conclude that, for every $0<\gamma<1/2$,
\begin{align*}
\E\left[ \|\omdR\|_{C^\gamma_t(\tilde{H}^{-4})}^p \right] \lesssim 1+R^{2p}.
\end{align*}

It remains to show \eqref{eq:AmAtilde}. We recall that $u^{\delta,R}=K\ast \omd$ and we call $\tilde{u}^{\delta,R}=K^\delta\ast \omd$. We bound
\begin{align*}
|S_1^{1/p}-\tilde{S_1}^{1/p}| &= \left|\E \left[\left\| \int_s^t \diverg(u^{\delta,R} \omdR) \rd r \right\|_{\tilde{H}^{-4}}^p \right]^{1/p} -\E \left[\left\| \int_s^t \diverg(\tilde{u}^{\delta,R} \omdR) \rd r \right\|_{\tilde{H}^{-4}}^p \right]^{1/p}\right|\\
& \le \E \left[\left\| \int_s^t \diverg((u^{\delta,R}-\tilde{u}^{\delta,R}) \omdR) \rd r \right\|_{\tilde{H}^{-4}}^p \right]^{1/p}\\
&\le (t-s) \E \left[\sup_{r\in [0,T]} \left\|\diverg((u^{\delta,R}-\tilde{u}^{\delta,R}) \omdR)\right\|_{\tilde{H}^{-4}}^p\right]^{1/p}\\
&\le (t-s) \E \left[\sup_{r\in [0,T]} \left\|(u^\delta-\tilde{u}^\delta) \omdR\right\|_{L^2}^p\right]^{1/p}.
\end{align*}
By Lemma \ref{lem: l infty and l 1}, Lemma \ref{lem:Hm2 approx} and the condition \eqref{eq:omega zero growth} on $\omd_0$, we get
\begin{align}
|S_1^{1/p}-\tilde{S}_1^{1/p}|
&\le (t-s) \E \left[\sup_{r\in [0,T]} \|u^{\delta,R}-\tilde{u}^{\delta,R}\|_{L^\infty}^p \|\omdR\|_{L^2}^p\right]^{1/p}\nonumber\\
&\le (t-s) \|\omdR_0\|_{L^2}\E \left[\sup_{r\le T} \|(\nabla G^\delta-\nabla G)\ast \omdR\|_{L^\infty}^p \right]^{1/p}\nonumber\\
&\lesssim (t-s) \delta^{1/4} \|\omdR_0\|_{L^1}^{1/2}\|\omdR_0\|_{L^\infty}^{1/2} (\|\omdR_0\|_{L^\infty} +\|\omdR_0\|_{L^1})\nonumber \\
&\lesssim (t-s) \delta^{1/4}(\|\omdR_0\|_{L^\infty} +\|\omdR_0\|_{L^1})^2\nonumber \\
&\lesssim (t-s)o(1),\label{eq: A minus A tilde}
\end{align}
which implies \eqref{eq:AmAtilde}. The proof is complete.
\end{proof}

\begin{remark}\label{rmk:extension product}
    By a standard approximation argument for quadratic forms, one can show that the equality \eqref{eq:Hm2 omega Q} and the bound \eqref{eq: sum sigma H minus 2} hold true replacing $\sigma_k^\delta$ with $\sigma_k$ and $\omega^{\delta,R}$ with any $\omega$ in $\dot{H}^{-1}$.
\end{remark}

\begin{corollary}
\label{cor: time continuity prob}
    Let $0< \gamma < 1/2$.
    For every $\epsilon>0$, there exists $C_\epsilon>0$ such that, for every $\delta>0$ (small),
    \begin{align*}
        P\{\|\omega^\delta\|_{C_t^\gamma(\tilde{H}^{-4})}> C_\epsilon\} <\epsilon.
    \end{align*}
\end{corollary}

\begin{proof}
By \ref{cor:exit_time} and Lemma \ref{lem: C gamma of H minus 3}, for every $M, R > 0$, we have, some constant $C>0$ independent of $M,R$ and $C_R>0$ independent of $M$,
\begin{align*}
    \sup_{\delta}P\{\|\omega^{\delta}\|_{C_t^{\gamma}(\tilde{H}^{-4})} > M \}
    & \leq \sup_{\delta} P\{\tau_R^{\delta} < T\}
    + \sup_{\delta}P\{\tau_R^{\delta} \ge T, \|\omega^{\delta}\|_{C_t^{\gamma}(\tilde{H}^{-4})} > M \} \\
    & \leq \frac{C}{R^2} + \frac{1}{M^2} \sup_{\delta}\mathbb{E}[\|\omega_{\cdot\wedge\tau_R^{\delta}}^{\delta}\|_{C_t^{\gamma}(\tilde{H}^{-4})}^2]
    \leq \frac{C}{R^2} + \frac{C_R}{M^2}.
\end{align*}
For $\epsilon >0$, we can choose $R$ such that $\frac{C}{R^2}<\epsilon$ and then $C_\epsilon=M$ such that $\frac{C_R}{M^2}<\epsilon$. The proof is complete.
\end{proof}

\section{Weak existence}
\label{sec: weak existence}

In this section we prove weak existence for \eqref{eq:main}, that is Theorem \ref{thm: weak-existence}. We exploit a classical strategy, showing tightness of solutions $(\omd)_\delta$ to the regularized equation \eqref{eq: regularized main} and then showing that any limit of the family $(\omd)_\delta$ solves \eqref{eq: regularized main}.

The key point of the argument is the uniform bound on the $H^{1-\alpha}$ norm of the velocity field $u^\delta=K^\delta\ast\omd$. Indeed, the embedding $H^{1-\alpha}$ into $L^2$ is compact on each bounded domain of $\R^2$ and hence the uniform $H^{1-\alpha}$ bound allows us to get convergence in $L^2$ with the strong topology and hence to pass to the limit in the nonlinear term $\diverg(u^\delta u^{\delta,\top})$ in Euler equation.

In order to prove tightness of the family of the laws of $u^{\delta}=K\ast \omd$, as mentioned above, the idea is to apply the stochastic Aubin-Lions lemma to the triplet of spaces $H^{1-\alpha}(\R^2) \subset L^2(\R^2) \subset H^{-3}(\R^2)$ and the compact embedding of $L^2_t(H^{1-\alpha}) \cap C^{\alpha}_t(H^{-3})$ into $L^2_t(L^2)$. However, the immersion of $H^{1-\alpha}$ into $L^2$ is not compact when the underlying space, in this case $\R^2$, is not compact. To circumvent this technical difficulty, we introduce a suitable weight $w$ in the space $L^2$ and show the compact embedding into $L^2_t(L^2(\R^2;w))$.

Let $w:\R^2 \to \R$ be a $C^\infty$ function such that $0<w\leq 1$ and 
    \begin{equation*}
        \sup_{B_R^c} w \to 0,
        \qquad
        \mbox{as }
        R\to \infty.
    \end{equation*}
We define the weighted Lebesgue space $L^2(\R^2;w)$ as the space of $L^2_{loc}(\R^2)$ functions $f$ satisfying
\begin{align*}
    \|f\|_{L^2(\R^2;w)}^2:= \int_{\R^2}|f(x)|^2 w(x) \rd x <\infty.
\end{align*}

\begin{lemma}\label{lem:Aubin_Lions_mod}
The embedding
    \begin{equation*}
    L^2_t(H^{1-\alpha})\cap C^{\gamma}_t(H^{-3})
    \subset \subset L^2_t(L^2(\R^2;w))
\end{equation*}
is compact.
\end{lemma}

\begin{proof}
As in \cite[Section 2.5.1]{edmunds1996function}, for $s\in \R$, $R>0$, we define the space $H^s(B_R)$ as the space of restrictions of $H^s(\R^2)$ distributions to $B_R$, $H^s(B_R)$ is a separable Hilbert space with the norm
\begin{align*}
    \|h\|_{H^s(B_R)} := \inf\{\|\tilde{h}\|_{H^s(\R^2)} \mid \tilde{h}\in H^s(\R^2),\, \tilde{h}\mid_{B_R} = h\}.
\end{align*}
By \cite[Section 2.5.1]{edmunds1996function}, $H^s \subset H^{s^\prime}$ if $s > s^{\prime}$ and the embedding is compact.

By the stochastic Aubin-Lions lemma (see \cite[Theorem 2.1]{flandoli1995martingale}), for every $R>0$, the following embedding is compact:
\begin{equation*}
   S_R := L^2_t(H^{1-\alpha}(B_R))\cap C^{\gamma}_t(H^{-3}(B_R))
    \subset \subset L^2_t(L^2(B_R)) \simeq L^2([0,T]\times B_R).
\end{equation*}
Assume that $(f_n)_n$ is a bounded family of functions in $L^2_t(H^{1-\alpha})\cap C^{\gamma}_t(H^{-3})$. 
Then $(f_n)_n$ is also bounded in $S_R$, therefore it admits a subsequence $(f_{n_k})_k$ converging to a function $f^R\in L^2([0,T]\times B_R)$. By a diagonal argument we can make the subsequence $n_k$ independent of $R\in \mathbb{N}^+$. We define now
\begin{equation*}
    f(t,x) := \sum_{R\in \mathbb{N}^+} f^R(t,x) 1_{x\in B_R\setminus B_{R-1}},
\end{equation*}
(where $B_0 = \emptyset$), so that the restriction of $f$ to $B_R$ equals $f^R$ Lebesgue-a.e.. Then $f$ is in $L^2([0,T]\times \R^2)$ and
\begin{align*}
    \int_{0}^{T} \int_{\R^2}|f(t,x)|^2 \rd x \rd t &= \lim_{R\to \infty} \int_{0}^{T} \int_{B_R}|f(t,x)|^2 \rd x \rd t\\
    &=\lim_{R\to\infty}\lim_k \int_{0}^{T} \int_{B_R}|f_{n_k}(t,x)|^2 \rd x \rd t\\
    &\leq \liminf_k \int_{0}^{T} \int_{\R^2}|f_{n_k}(t,x)|^2 \rd x \rd t
    \leq \liminf_k \|f_{n_k}\|_{L^2_t(H^{1-\alpha})}^2
\end{align*}
For the convergence in $L^2_t(L^2(\R^2;w))$, we have
\begin{align*}
\int_{0}^{T} \int_{\R^2}|(f_{n_k}-f)(t,x)|^2 w(x)\rd x \rd t 
& \leq \int_{0}^{T} \int_{B_R}|(f_{n_k}-f)(t,x)|^2 w(x) \rd x \rd t \\
&\quad +
\int_{0}^{T} \int_{B_R^c}|(f_{n_k}-f)(t,x)|^2 w(x) \rd x \rd t \\
& \leq \int_{0}^{T} \int_{B_R}|(f_{n_k}-f)(t,x)|^2 \rd x \rd t \\
&\quad +
2\sup_{B_R^c}w \cdot  (\sup_{k}\|f_{n_k}\|_{L^2_t(L^2(\R^2))}^2+\|f\|_{L^2_t(L^2(\R^2))}^2)\\
&=:A_{R,k}+A_R
\end{align*}
For any $\eps>0$, we can choose $R_\eps\in\mathbb{N}$ large enough such that $A_{R_\eps}<\epsilon$, then we can choose $k_\epsilon$ such that, for every $k\ge k_\epsilon$, $A_{R_\eps,k}<\epsilon$. Hence $(f_{n_k})_k$ converges to $f$ in $L^2_t(L^2(\R^2;w))$. The proof is complete.
\end{proof}

We are now ready to prove our weak existence result Theorem \ref{thm: weak-existence}. As in the previous sections, $Q^\delta$, $G^\delta$ and $K^\delta$, $\omd_0$ satisfy the conditions in Section \ref{sec:regularization} and $\omd$ is the solution to the regularized equation \eqref{eq: regularized main}.

\textbf{Step 1:} \textit{Tightness}. Let $u^{\delta} = K\ast \omega^{\delta}$. We want to prove the thightness of the family $(u^{\delta})_{\delta \geq 0}$ in the space $L^2_t(L^2(\R^2;w))$. We recall the following equalities and bounds on $K\ast f$:
\begin{align*}
&\|K\ast f\|_{L^2} = (2\pi)^{-1}\|f\|_{\dot{H}^{-1}},\\
&\|K\ast f\|_{H^{1-\alpha}} \lesssim \|f\|_{\dot{H}^{-1}}+\|f\|_{H^{-\alpha}},\\
&\|K\ast f\|_{H^{-3}} = (2\pi)^{-1}\|f\|_{\tilde{H}^{-4}}.
\end{align*}
By Proposition \ref{prop:a_priori_bd}, Corollary \ref{cor:exit_time} and Corollary \ref{cor: time continuity prob}, we get, for $0<\gamma<1/2$, for some $C>0$,
\begin{align}
&\limsup_{\delta\to 0}\left(\sup_{t\in [0,T]}\E[\|u^\delta\|_{L^2}^2] + \int_0^T \E[\|u^\delta\|_{H^{1-\alpha}}^2] \rd t \right) \le C\|u_0\|_{L^2}^2,\label{eq:unif_bd_1}\\
&\lim_{M\to\infty}\sup_\delta P\{\sup_{t\in [0,T]}\|u^\delta_t\|_{L^2}^2 \ge M\} =0,\label{eq:unif_bd_3}\\
&\lim_{M\to \infty}\sup_\delta P\{\|u^\delta\|_{C_t^\gamma(H^{-3})}> M\} =0.\label{eq:unif_bd_2}
\end{align}

By Lemma \ref{lem:Aubin_Lions_mod}, for $M>0$ the set
\begin{equation*}
    K_M = 
    \{
    f \in L^2_t(L^2(\R;w)) \mid \|f\|_{L^2_t(H^{1-\alpha})} + \|f\|_{C_t^{\gamma}(H^{-3})} \leq M
    \}
\end{equation*}
is compact in $L^2_t(L^2(\R;w))$.
We get
\begin{align}
     P \{u^{\delta} \notin K_M \}
    & \leq  P \{ \|u^{\delta}\|_{L^2_t(H^{1-\alpha})} > \frac{M}{2} \}
    +  P \{ \|u^{\delta}\|_{C^{\gamma}_t(H^{-3})} > \frac{M}{2} \}\nonumber \\
    & \leq \frac{4}{M^2} \mathbb{E}[\|u^{\delta}\|_{L^2_t(H^{1-\alpha})}^2]
    +  P \{ \|u^{\delta}\|_{C^{\gamma}_t(H^{-3})} > \frac{M}{2} \}\label{eq: compactness}
\end{align}
By the bounds \eqref{eq:unif_bd_1} and \eqref{eq:unif_bd_2}, for every $\epsilon>0$ we can choose $M$ such that the right-hand side of \eqref{eq: compactness} is smaller than $\epsilon$.
Hence the laws of $u^\delta$ are tight in $L^2_t(L^2(\R^2;w))$. As a consequence, the laws of $(u^\delta,(W^k)_k)$ are tight on $L^2_t(L^2(\R^2;w))\times C_t^{\mathbb{N}_+}$ (with $\mathcal{B}(L^2_t(L^2(\R^2;w)))\otimes \mathcal{B}(C_t)^{\otimes \mathbb{N}_+}$ as $\sigma$-algebra).

\textbf{Step 2:} \textit{Convergence $P$-a.s. of a subsequence of copies}. By Skorohod representation theorem \cite[Theorem 3.30]{kallenberg1997foundations}, there exist a complete probability space $(\tilde{\Omega},\tilde{\mathcal{A}},\tilde{P})$, a sequence of $L^2_t(L^2(\R^2;w))\times C_t^{\mathbb{N}_+}$-valued random variables $(\tilde{u}^{\delta_n},(\tilde{W}^{k,\delta_n})_k)$ on $(\tilde{\Omega},\tilde{\mathcal{A}},\tilde{P})$, with $\delta_n\to 0 $ as $n\to\infty$, and a $L^2_t(L^2(\R^2;w))\times C_t^{\mathbb{N}_+}$-valued random variable $(\tilde{u},(\tilde{W}^k)_k)$ such that each $(\tilde{u}^{\delta_n},(\tilde{W}^{k,\delta_n})_k)$ has the same law of $(u^{\delta_n},(W^k)_k)$ and the sequence $(\tilde{u}^{\delta_n},(\tilde{W}^{k,\delta_n})_k)_n$ converges $\tilde{P}$-a.s. to $(\tilde{u},(\tilde{W}^k)_k)$ as $n\to\infty$. We call
\begin{align*}
\tilde{\omega}^{\delta_n}=\operatorname{curl} \tilde{u}^{\delta_n},\quad \tilde{\omega}=\operatorname{curl}\tilde{u}.
\end{align*}

Since each $\tilde{u}^\delta_n$ has the same law of $u^{\delta_n}$, the family $(\tilde{u}^\delta_n)_n$ satisfies the bounds \eqref{eq:unif_bd_1}, \eqref{eq:unif_bd_3} and \eqref{eq:unif_bd_2}. More precisely, for each $n$ there is a version of $\tilde{u}^{\delta_n}$ (which we will still call $\tilde{u}^{\delta_n}$) which has H\"older-continuous paths with values in $H^{-3}$ and satisfies the bounds \eqref{eq:unif_bd_1}, \eqref{eq:unif_bd_3} and \eqref{eq:unif_bd_2}. Measurability of $C_t^\gamma(H^{-3})$, $L^2_t(H^{1-\alpha})$
and $L^\infty_t(L^2)$
in $L^2_t(L^2(\R^2,w))$ follows from semicontinuity (see Lemma \ref{lem: semicontinuity} in the appendix)\footnote{We recall that, if $X_n$ and $X$ are random variables on a Polish space $\chi$, $X_n\to X$ $P$-a.s. and $g$ is a non-negative lower semi-continuous function on $\chi$, then $P\{g(X)>M\}\le P(\liminf_n\{g(X_n)>M\}) \le \sup_n P\{g(X_n)>M\}$. Moreover, writing $g$ as supremum of an increasing sequence of continuous functions, we have $\E[g(X)]\le \liminf\E[g(X_n)]$.}. Also the limit $\tilde{u}$ is in $C_t^\gamma(H^{-3})\cap L^\infty_t(L^2)$ $P$-a.s. (more precisely, $\tilde{u}$ has a version which is in $C_t^\gamma(H^{-3})\cap L^\infty_t(L^2)$), showing \eqref{eq: properties def} for $\tilde{\omega}$. Finally, since the $L^2_t(H^{1-\alpha})$ norm and the $L^2_t(L^2)$ norm are lower semi-continuous in $L^2_t(L^2(\R^2;w))$, then the limit $\tilde{u}$ also satisfies
\begin{align}
    \sup_{t\in [0,T]}\E[\|\tilde{u}_t\|_{L^2}^2] + \int_0^T \E[\|\tilde{u}_t\|_{H^{1-\alpha}}^2] \rd t \le C\|u_0\|_{L^2}^2.\label{eq:Hm alpha bd limit}
\end{align}
Indeed, we can pass to the limit in the uniform bound on $\sup_{t\in [0,T]}\E[\|\tilde{u}^{\delta_n}\|_{L^2}^2]$ as follows: First, by lower-semicontinuity of the $L^2_t(L^2)$ norm and the bound \eqref{eq:unif_bd_1}, we have $(2h)^{-1}\int_{t-h}^{t+h} \E[\|\tilde{u}\|_{L^2}^2] \le \liminf_n (2h)^{-1}\int_{t-h}^{t+h} \E[\|\tilde{u}^{\delta_n}\|_{L^2}^2] \le C\|u_0\|_{L^2}^2 <\infty$ for every $t\in (0,T)$ and $h>0$ small (for some $C$ independent of $t$ and $h$). Then, we let $h\to 0$ and get $\mathrm{esssup}_{t\in [0,T]}\E[\|\tilde{u}\|_{L^2}^2]<\infty$; the essential supremum becomes a supremum by path continuity of $\tilde{u}$ in $H^{-3}$.

From \eqref{eq:Hm alpha bd limit} we have property \eqref{eq:Hm alpha bd sol} for $\tilde{\omega}$.

Now we show that $(\tilde{\omega}^{\delta_n},(\tilde{W}^{k,\delta_n})_k)$ satisfies \eqref{eq: regularized main} on a suitable filtered probability space. For $t\in [0,T]$, we call $\tilde{\mathcal{F}}^{0,\delta}_t$ the $\sigma$-algebra on $\tilde{\Omega}$ generated by $\sigma\{\tilde{u}^{\delta_n}_s,(\tilde{W}^{\delta_n,k}_s)_k\mid s\le t\}$ (with $\tilde{u}^{\delta_n}_s$ as $H^{-3}$-valued random variables) and the $\tilde{P}$-null sets; we call then $\tilde{\mathcal{F}}^{\delta_n}_t= \cap_{t'>t} \tilde{\mathcal{F}}^{0,\delta_n}_{t'}$. Similarly we call $\tilde{\mathcal{F}}^{0}_t$ the $\sigma$-algebra generated by $\sigma\{\tilde{u}_s,(\tilde{W}^k_s)_k\mid s\le t\}$ (with $\tilde{u}_s$ as $H^{-3}$-valued random variables) and the $\tilde{P}$-null sets; we call then $\tilde{\mathcal{F}}_t= \cap_{t'>t} \tilde{\mathcal{F}}^{0}_{t'}$.

\begin{lemma}
For each $n$, the filtration $(\tilde{\mathcal{F}}^{\delta_n}_t)_t$ is complete and right-continuous; $(\tilde{W}^{\delta_n,k})_k$ is a cylindrical Brownian motion with respect to $(\tilde{\mathcal{F}}^{\delta_n}_t)_t$ and $\tilde{u}^{\delta_n}$ is $(\tilde{\mathcal{F}}^{\delta_n}_t)_t$-progressively measurable. The analogous statement holds for $(\tilde{\mathcal{F}}_t)_t$, $(\tilde{W}^k)_k$ and $\tilde{u}$.
\end{lemma}

This lemma is a classical technical tool, we omit the proof since it is completely analogous to e.g. the proof of \cite[Lemma 4.12]{brzezniak2019existence}.

For each $n$, since $u^{\delta_n}$ and $\tilde{u}^{\delta_n}$ have the same law in $C_t(H^{-3})$ (Notice that, by Lemma \ref{lem: semicontinuity} the $H^1$ norm is lower semicontinuous in $H^{-3}$), $\tilde{u}^{\delta_n}$ satisfies
\begin{align*}
\sup_{t\in [0,T]}\|\tilde{\omega}^{\delta_n}_t\|_{L^2\cap \dot{H}^{-1}} \approx   \sup_{t\in [0,T]}\|\tilde{u}^{\delta_n}_t\|_{H^1} <\infty.
\end{align*}

\begin{lemma}
For each $n$, the object $(\tilde{\Omega},\tilde{\mathcal{A}},(\tilde{\mathcal{F}}^{\delta_n}_t)_t,\tilde{P},(\tilde{W}^{\delta_n,k})_k,\tilde{\omega}^{\delta_n})$ is a weak solution in $\dot{H}^{-1}$ to \eqref{eq: regularized main} with $\delta=\delta_n$.
\end{lemma}

Also this lemma can be proved by classical tools, see e.g. the proof of \cite[Lemma 4.13]{brzezniak2019existence}, so we omit the proof here.

\textbf{Step 3:} \textit{Pass to the limit in the SPDE in the distributional sense}. For simplicity of notation, from now on we will rename $(\tilde{\omega}^{\delta_n},\tilde{u}^{\delta_n}, (\tilde{W}^{\delta_n,k})_k)$ as $(\omega^{\delta_n},u^{\delta_n},(W^{\delta_n,k})_k$, similarly for $(\tilde{\omega},\tilde{u}, (\tilde{W}^{k})_k)$. Let $\varphi \in C^{\infty}_c(\R^2)$, then $\omega^{\delta}= \operatorname{curl}u^{\delta}$ satisfies the following equation
\begin{align}
\begin{aligned}\label{eq:SPDE_approx_tilde}
        \langle \omega^{\delta_n}_t, \varphi \rangle
        & = \langle \omega^{\delta_n}_0, \varphi\rangle
        + \int_0^t \langle (K^{\delta_n} \ast \omega^{\delta_n}_r) \omega^{\delta_n}_r, \nabla \varphi \rangle \rd r
        + \sum_{k} \int_{0}^{t} \langle \sigma_k^{\delta_n}\omega^{\delta_n}_r, \nabla \varphi \rangle \rd W^{\delta_n,k}_r \\
      &\quad + \frac{c_{\delta_n}}{2} \int_{0}^{t} \langle \omega^{\delta_n}_r, \Delta \varphi \rangle \rd r
       := A_1 + A_2 + A_3 + A_4.
\end{aligned}
\end{align}
We will now send $n \to \infty$ in each term, possibly along a subsequence, to recover an equation for $\langle\omega,\varphi\rangle$.

We observe that, for any $R>0$,
\begin{equation}
    \label{eq: conv u delta on balls}
    \|u^{\delta_n}-u\|_{L^2_t(L^2(B_R))} \le
    \|u^{\delta_n}-u\|_{L^2_t(L^2(\R^2;w))} \to 0,
    \quad P\text{-a.s.}.
\end{equation}
For the left-hand side of the equation, since $\operatorname{supp}\varphi \subset B_R$ for some $R>0$, we have $P$-a.s.
\begin{align}\label{eq:conv LHS}
    \int_{0}^{t} |\langle \omega^{\delta_n}_s, \varphi \rangle - \langle \omega_{s}, \varphi \rangle |^2 \rd s
    & = \int_{0}^{t} |\langle u^{\delta_n}_s - u_{s}, \nabla^{\top} \varphi \rangle |^2 \rd s\\
    &\le \|\varphi\|_{H^1} \|u^{\delta_n}-u\|_{L^2_t(L^2(B_R))} \to 0,
\end{align}
so the left-hand side of \eqref{eq:SPDE_approx_tilde} converges to $\langle \omega_t,\varphi\rangle$ in $L^2([0,T])$, $P$-a.s..

By the same argument and the fact that $c_{\delta_n} \to \pi/(2\alpha)$ (see \eqref{eq:Q delta 0}), we have
\begin{align}\label{eq:conv A4}
    A_4 \to \frac{\pi}{4\pi}\int_0^t\langle \omega_r,\varphi\rangle \rd r \text{ in }C([0.T]),\quad P\text{-a.s.}.
\end{align}

By the condition \eqref{eq:omega zero convergence}, $\omega^{\delta_n}_0 \to \omega_0$ in $\dot{H}^{-1}$, which proves the convergence of $A_1$.

We prove now the convergence, up to a subsequence,
\begin{equation}
    \label{eq: nonlinear term}
    A_2\to \int_{0}^{t} \langle u u^{\top}, \nabla \nabla^{\perp} \varphi \rangle
    \rd r \text{ in }C([0,T]),\quad P\text{-a.s.}.
\end{equation}
Recall equality \eqref{eq: curl div} for $\diverg(K\ast \omega^\delta)\omega^\delta$. We have
\begin{align*}
   & \left |
        \int_{0}^{t}
        \left(
        \langle (K^{\delta_n}\ast \omega^{\delta_n}_r) \omega^{\delta_n}_r, \nabla \varphi \rangle 
        - \langle u_ru_r^\top, \nabla\nabla^\perp \varphi \rangle 
        \right)
        \rd r
    \right | 
    \\
    & \leq 
    \left |
        \int_{0}^{t} \langle ((K^{\delta_n} - K) \ast \omega^{\delta_n}_r)\omega^{\delta_n}_r, \nabla \varphi \rangle \rd r
    \right |
    + \left |
        \int_{0}^{t}
        \left(
        \langle (K\ast \omega^{\delta_n}_r) \omega^{\delta_n}_r, \nabla \varphi \rangle 
        - \langle u_ru_r^\top, \nabla\nabla^\perp \varphi \rangle 
        \right)
        \rd r
    \right | \\
    &= \left |
        \int_{0}^{t} \langle ((K^{\delta_n} - K) \ast \omega^{\delta_n}_r)\omega^{\delta_n}_r, \nabla \varphi \rangle \rd r
    \right |
    +  \left |
        \int_{0}^{t} 
        \left(
        \langle 
        u^{\delta_n}_r u^{\delta_n, \top}_r -
        u_r u^{\top}_r, \nabla \nabla^{\perp} \varphi \rangle
        \right)        
        \rd r
    \right | \\
    & =: A_{21} + A_{22}.
\end{align*}    
For the term $A_{21}$, we have, for any $1\le p<\infty$, (note that $\|\cdot\|_{H^{-4}} \leq \|\cdot\|_{\tilde{H}^{-4}}$)
\begin{align*}
    \E[\sup_{t\in [0,T]}A_{21}^p]\le T\|\varphi\|_{H^4}^p  \E\left[\sup_{r\in [0,T]}\left\| \diverg [((K-K^{\delta_n})\ast \omega^{\delta_n}_r)\omega^{\delta_n}_r] \rd r\right\|_{\tilde{H}^{-4}}^p\right] \to 0,
\end{align*}
where the convergence follows from the the estimate \eqref{eq: A minus A tilde} (with $s=0$). Hence, passing to a subsequence, $A_{21}$ tends to $0$ in $C([0,T])$, $P$-a.s.. For the term $A_{22}$, by Cauchy-Schwartz inequality and \eqref{eq: conv u delta on balls}, taking $R>0$ such that $\operatorname{supp}\varphi \subset B_R$., we get
\begin{align*}
    A_{22} \leq \|\varphi \|_{C^2}
    \left(    \|u^{\delta_n}\|_{L^2_t(L^2(B_R))}^2 + \|u\|_{L^2_t(L^2(B_R))}^2 
    \right)^{1/2}
    \|u^{\delta_n} - u\|_{L^2_t(L^2(B_R))}
    \to 0\quad P\text{-a.s.}.
\end{align*}

Concerning the convergence of the stochastic integral $A_3$, we use a classical result, see e.g. \cite[Lemma 2.1]{debussche2011local}, here in the version \cite[Lemma 4.3]{bagnara2023no}: if
\begin{align}
\sum_k \int_0^T \left|\langle \sigma_k^{\delta_n}\omega^{\delta_n}_r,\nabla\varphi\rangle -\langle \sigma_k\omega_r,\nabla\varphi\rangle\right|^2 \rd r \to 0 \quad \text{in probability},\label{eq:conv_stoch_integrand}
\end{align}
then
\begin{align}\label{eq:conv diffusion coeff}
\sup_{t\in [0,T]} \left|\sum_k \int_{0}^{t} \langle \sigma_k^{\delta_n}\omega^{\delta}_r, \nabla \varphi \rangle \rd W^{\delta_n,k}_r - \sum_k \int_{0}^{t} \langle \sigma_k\omega_r, \nabla \varphi \rangle \rd W^k_r\right| \to 0\quad \text{in probability}.
\end{align}
To show \eqref{eq:conv_stoch_integrand}, we split
\begin{align*}
\sum_k \int_0^T \left|\langle \sigma_k^{\delta_n}\omega^{\delta_n}_r,\nabla\varphi\rangle -\langle \sigma_k\omega_r,\nabla\varphi\rangle\right|^2 \rd r &\lesssim \sum_k \int_0^T \left|\langle \sigma_k(\omega^{\delta_n}_r-\omega_r),\nabla\varphi\rangle \right|^2 \rd r\\
&\quad + \sum_k \int_0^T \left|\langle (\sigma_k^{\delta_n}-\sigma_k)\omega^{\delta_n}_r,\nabla\varphi\rangle \right|^2 \rd r\\
&=: A_{31}+A_{32}.
\end{align*}
To bound the term $A_{31}$, we consider a $C^\infty_c$ function $\psi$ with $0\le \psi\le 1$ on $\R^2$ and $\psi=1$ on the support of $\varphi$. We have
\begin{align*}
A_{31} &= \sum_k \int_0^T \left|\langle \sigma_k(\omega^{\delta_n}_r-\omega_r)\psi,\nabla\varphi\rangle \right|^2 \rd r\\
&\le \|\varphi\|_{H^3}^2 \sum_k \int_0^T \|\sigma_k(\omega^{\delta_n}_r-\omega_r)\psi\|_{H^{-2}}^2 \rd r.
\end{align*}
Now we proceed as in the bound \eqref{eq: sum sigma H minus 2} for $\sum_k\|\sigma_k^{\delta,R} \omd\|_{H^{-2}}$ in the proof of Lemma \ref{lem: C gamma of H minus 3}, replacing $\sigma_k^\delta$ with $\sigma_k$ and $\omd$ with $(\omega^{\delta_n}-\omd)\psi$ (see Remark \ref{rmk:extension product}) and we get
\begin{align*}
\sum_k \|\sigma_k(\omega^{\delta_n}-\omega)\psi\|_{H^{-2}}^2 \lesssim \|(\omega^{\delta_n}-\omega)\psi\|_{H^{-1}}^2.
\end{align*}
For the right-hand side, we have $P$-a.s., (with multiplicative constant depending on the support of $\varphi$)
\begin{align*}
\|(\omega^{\delta_n}-\omega)\psi\|_{L^2_t(H^{-1})} &= \|\operatorname{curl}((u^{\delta_n}-u)\psi) +(u^{\delta_n}-u)\cdot\nabla^\perp \psi \|_{L^2_t(H^{-1})}\\
&\le \|(u^{\delta_n}-u)\psi\|_{L^2_t(L^2))} +\|(u^{\delta_n}-u)\cdot\nabla^\perp \psi\|_{L^2_t(L^2)}\\
&\lesssim \|u^{\delta_n}-u\|_{L^2_t(L^2(\R^2;w))} \to 0
\end{align*}
Therefore we obtain
\begin{align*}
A_{31}\lesssim \|\varphi\|_{H^3}\|(\omega^{\delta_n}-\omega)\psi\|_{L^2_t(H^{-1})}\to 0\quad P\text{-a.s.}.
\end{align*}
For $A_{32}$, we have
\begin{align*}
    A_{32}\le \|\varphi\|_{H^3}^2 \sum_k\int_0^T \|(\sigma_k^{\delta_n}-\sigma_k)\omega^{\delta_n}\|_{H^{-2}}^2 \rd r.
\end{align*}
Recalling Lemma \ref{lem:Q_reguralized} and Remark \ref{rmk:sigma_delta_1}
and proceeding as in the equality \eqref{eq:Hm2 omega Q}, replacing $\sigma_k^\delta$ with $\sigma_k^{\delta_n}-\sigma_k$ (see again Remark \ref{rmk:extension product}), we get
\begin{align}
    &\sum_k\|(\sigma_k^{\delta_n}-\sigma_k)\omega^{\delta_n}\|_{H^{-2}}^2\nonumber\\
    &= \sum_k \int \langle n\rangle^{-4}|\mathcal{F}((\sigma_k^{\delta_n}-\sigma_k)\omega^{\delta_n})(n)|^2 \rd n\nonumber\\
    &= \sum_k \int \langle n \rangle^{-4} \iint (\sigma^{\delta_n}_k(x)-\sigma_k(x))\omd(x)e^{-2\pi ix\cdot n}\cdot (\sigma^{\delta_n}_k(y)-\sigma_k(y))\omd(y)e^{2\pi iy\cdot n} \rd x\rd y \rd n\nonumber\\
    &= \iint \omega^{\delta_n}(x)\omega^{\delta_n}(y) \tr [Q+Q^{\delta_n}-2Q^{\delta_n,h}](x-y) \int \langle n \rangle^{-4} e^{-2\pi i(x-y)\cdot n} \rd n \rd x\rd y\nonumber\\
    &= \int |\widehat{\omega^{\delta_n}}(n)|^2 \widehat{\phi}(n) \rd n,\label{eq:sigma delta minus sigma}
\end{align}
where
\begin{align*}
    \phi(x) = \tr[Q+Q^{\delta_n}-2Q^{\delta_n,h}](x) \int \langle n \rangle^{-4} e^{-ix\cdot n} \rd n.
\end{align*}
For the Fourier transform of $\phi$ we have, for every $n$,
\begin{align}
\widehat{\phi}(n) &= \int \tr[Q+Q^{\delta_n}-2Q^{\delta_n,h}](n-n^\prime)\cdot \langle -n^\prime \rangle^{-4} \rd n^\prime\nonumber \\
&\lesssim \int_{|n-n^\prime|>1/{\delta_n}} \langle n-n^\prime \rangle^{-2-2\alpha} \langle n^\prime \rangle^{-4} \rd n^\prime.\label{eq:A32_intermediate}
\end{align}
If $|n|\le 1/(2{\delta_n})$, then the term in \eqref{eq:A32_intermediate} can be bounded (up to a multiplicative constant) by
\begin{align*}
&\delta_n^{2+2\alpha} \int_{|n-n^\prime|>1/{\delta_n}} \langle n^\prime \rangle^{-4} \rd n^\prime 
\le \delta_n^{2+2\alpha} \int_{|n^\prime|>|n|} \langle n^\prime \rangle^{-4} \rd n^\prime \lesssim \delta_n^{2+2\alpha} \langle n\rangle^{-2},
\end{align*}
where we have used that, if $|n-n^\prime|>1/{\delta_n}$, then $|n^\prime| \ge |n^\prime-n|-|n| \ge 1/{\delta_n} \ge |n|$. If instead $|n|>1/(2{\delta_n})$, then, as in \eqref{eq:psi Fourier bd}, the term in \eqref{eq:A32_intermediate} can be bounded (up to a multiplicative constant) by $\langle n\rangle^{-2}$. Hence we arrive at
\begin{align*}
    \widehat{\phi}(n) \lesssim {\delta_n}^{2+2\alpha} \langle n\rangle^{-2} 1_{|n|\le 1/(2{\delta_n})} +\langle n\rangle^{-2} 1_{|n|> 1/(2{\delta_n})}.
\end{align*}
Inserting this inequality in \eqref{eq:sigma delta minus sigma}, we get
\begin{align*}
    \sum_k\|(\sigma_k^{\delta_n}-\sigma_k)\omega^{\delta_n}\|_{H^{-2}}^2 &\lesssim \delta_n^{2+2\alpha}\int |\widehat{\omega^{\delta_n}}(n)|^2 \langle n\rangle^{-2} \rd n +\int |\widehat{\omega^{\delta_n}}(n)|^2 \langle n\rangle^{-2} 1_{|n|> 1/(2\delta_n)} \rd n\\
    &\lesssim \delta_n^{2+2\alpha}\|\omega^{\delta_n}\|_{H^{-1}}^2 +\delta_n^{2-2\alpha} \|\omega^{\delta_n}\|_{H^{-\alpha}}^2.
\end{align*}
Therefore we obtain
\begin{align*}
    \E[A_{32}]\lesssim \|\varphi\|_{H^3} (\delta_n^{2+2\alpha}\E[\|\omega^{\delta_n}\|_{L^2_t(H^{-1})}^2] +\delta_n^{2-2\alpha} \E[\|\omega^{\delta_n}\|_{L^2_t(H^{-\alpha})}^2])\to 0
\end{align*}
and so, up to taking a subsequence, $A_{32}$ converges to $0$ in probability. Hence \eqref{eq:conv_stoch_integrand} holds up to a subsequence and so \eqref{eq:conv diffusion coeff} holds up to a subsequence.

Putting together \eqref{eq:conv LHS}, \eqref{eq:conv A4}, \eqref{eq: nonlinear term} and \eqref{eq:conv diffusion coeff} (and the convergence of $A_1$), we can pass to the limit in \eqref{eq:SPDE_approx_tilde}, up to a subsequence, in $L^2([0,T])$ in probability and get the equation for $\langle \omega,\varphi\rangle$, for every $\varphi$ in $C^\infty_c$, $P$-a.s.: for a.e. $t\in [0,T]$,
\begin{align}
        \langle \omega_t, \varphi \rangle
        & = \langle \omega_0, \varphi\rangle
        + \int_0^t \langle uu^\top, \nabla \nabla^\perp\varphi \rangle \rd s
        + \sum_{k\in \mathbb{N}} \int_{0}^{t} \langle \sigma_k\omega_s, \nabla \varphi \rangle \rd W^{k}_s 
      + \frac{c}{2} \int_{0}^{t} \langle \omega_s, \Delta \varphi \rangle \rd s \label{eq:main_test}
    \end{align}
Since the left-hand side and all the integrals in \eqref{eq:main_test} are continuous in time, this equation holds for every $t$ (on a $P$-full set independent of $t$).

\textbf{Step 4:} \textit{Conclusion}. To conclude that $\omega$ is a weak solution to \eqref{eq:main} (in the sense of Definition \ref{def:def_main}), it is enough to remove the test function in the formulation \eqref{eq:main_test}. We start noting that, from the bounds in the proof of Lemma \ref{lem: C gamma of H minus 3}, since $\omega$ is in $L^\infty_t(H^{-1})\cap C_t(H^{-4})$, $\omega$ and all the integrals in \eqref{eq: formula def 27} are well defined and continuous with values in $H^{-4}$. For each $\varphi$ in $C^\infty_c$, equation \eqref{eq: formula def 27} holds (for every $t$) tested against $\varphi$, on a $P$-full set which might depend on $\varphi$. We can make the $P$-null set $\Omega_0\subseteq \Omega$ independent of $\varphi$, for $\varphi \in C^\infty_c$ in a countable dense set of $H^4$. Hence \eqref{eq: formula def 27} holds, for every $t$, on the $P$-full set $\Omega_0$. The proof of Theorem \ref{thm: weak-existence} is complete.

\section{Uniqueness}
\label{sec: uniqueness}

In this section, we prove pathwise uniqueness for \eqref{eq:main} in the space $L^p$, that is Theorem \ref{thm:Lp_wellposed}. The idea of the proof is to estimate the $\dot{H}^{-1}$ norm of the difference of two solutions $\omega^1$, $\omega^2$ and get advantage both of the control of the $H^{-\alpha}$ norm of $\omega^1-\omega^2$, due to the noise, and the uniform $L^p$ bound on $\omega^1$ and $\omega^2$.

We need the following result, which follows from \cite{bahouri2011fourier}:

\begin{lemma}\label{lem:product_Sobolev}
    Let $0<\alpha<1/2$ be given. We have the following inequality:
    \begin{equation*}
    \| f g \|_{\dot{H}^\alpha} \lesssim \| f \|_{\dot{H}^{1-\alpha}} \|g\|_{\dot{H}^{2\alpha}}.
\end{equation*}
\end{lemma}

\begin{proof}
    From \cite[Example at page 63]{bahouri2011fourier} we know that the homogeneous Besov space $\dot{B}_{2,2}^{\alpha}$ coincides with the homogeneous Sobolev space $\dot{H}^{\alpha}$ when $0<\alpha<1$. Hence we can apply \cite[Proposition 2.20 and Corollary 2.55]{bahouri2011fourier} to obtain
    \begin{align*}
        \|f g \|_{\dot{H}^{\alpha}}
        \lesssim
        \|f g \|_{\dot{B}^\alpha_{2,1}}
        \lesssim \|f\|_{\dot{H}^{1-\alpha}}\|g\|_{\dot{H}^{2\alpha}}.
    \end{align*}
   Note that in order for the last bound to hold, we need $1-\alpha<1$ and $2\alpha < 1$, which gives us the condition $0<\alpha < 1/2$. The proof is complete.
\end{proof}

\begin{proof}[Proof of Theorem \ref{thm:Lp_wellposed}]

\textbf{Existence.} From Theorem \ref{thm: weak-existence} and its proof, we know that there exists a solution in $\dot{H}^{-1}(\R^2)$, which is the limit in law in $L^2_t(L^2(\R^2;w))$ of a family of solutions $\omd$ to the regularized equation \ref{eq: regularized main}. The bound \eqref{eq: Lp bound} follows by passing to the limit in the bounds in Lemma \ref{lem: l infty and l 1}, using the fact \eqref{eq:omega zero Lp} on the initial condiion and the lower semi-continuity of the $L^\infty_t(L^p\cap L^1)$ norm in the space $L^2_t(L^2(\R^2;w))$.

\textbf{Uniqueness.}
Let $\omega^1$ and $\omega^2$ be two $\dot{H}^{-1}$ solutions to \eqref{eq:main}, on the same filtered probability space $(\Omega,\mathcal{A},(\mathcal{F}_t)_t,P)$ (satisfying the standard assumption) and with respect to the same cylindrical Brownian motion $(W^k)_k$, satisfying $\omega_1, \omega_2 \in L^\infty([0,T]\times\Omega;L^p(\R^2)\cap L^1(\R^2))$. Note that the nonlinear term satisfies relation \eqref{eq: curl div}: this is clear for $p\ge 2$, while for $3/2<p<2$ one can use for example that under our assumptions, for $\eps>0$, $\|\omega^i(K\ast \omega^i)\|_{\dot{H}^{1-\alpha-\eps}}\lesssim \|\omega^i\|_{\dot{H}^{-\alpha-\eps}}\|\omega^i\|_{\dot{H}^{2(\alpha+\eps)-1}} \lesssim \|\omega^i\|_{\dot{H}^{-\alpha-\eps}}\|\omega^i\|_{L^1\cap L^p}$, by similar arguments as in the estimates below on $I_1$ and $I_2$. 

The difference $\omega:=\omega^1-\omega^2$ satisfies the equality in $H^{-4}$
\begin{align*}
\rd\omega + [(K\ast\omega^1)\cdot\nabla \omega +(K\ast\omega)
\cdot \nabla \omega^2]\rd t +\sum_{k=1}^\infty \sigma_k\cdot\nabla \omega \rd W^k = \frac{\pi}{4\alpha} \Delta\omega \rd t. 
\end{align*}
We take $G^\delta$ as in \eqref{eq: regularization green kernel} and we apply It\^o formula to $\langle \omega, G^\delta*\omega \rangle$: as in the proof of Lemma \ref{lem:Hm1 norm formula}, we obtain
\begin{align}
    \rd \langle \omega,G^\delta\ast\omega \rangle
    &= 2\langle \nabla G^\delta\ast \omega, (K\ast \omega^1)\omega\rangle \rd t +2\langle \nabla G^\delta\ast \omega, (K\ast \omega)\omega^2\rangle \rd t
    +2\sum_k M_k \rd W^k \nonumber \\
    &\quad +\iint \tr[(Q(0)-Q(x-y))D^2G^\delta(x-y)]\omega(x)\omega(y)\rd x\rd y \rd t \nonumber\\
    &=: (2I_1+2I_2)\rd t +2\sum_k M_k \rd W^k +J\rd t, \label{eq: ito uniqueness}
\end{align}
where $M_k  = \langle \nabla G^\delta\ast\omega, \sigma_k\omega\rangle$. Note that, for each $\delta>0$, since $\omega$ is in $L^\infty([0,T]\times\Omega;L^1(\R^2))$ and $\nabla G^\delta$ and $Q$ are bounded,
\begin{align*}
\sup_{t\in [0,T]}\sum_k \E[|M_k(t)|^2] & = \iiiint \tr[\nabla G^\delta(x-x') \cdot Q(x-y)\nabla G^\delta(y-y')]\\
&\quad \cdot\E[\omega(x')\omega(y')\omega(x)\omega(y)] \rd x\rd x'\rd y\rd y'<\infty
\end{align*}
and hence the It\^o integral is a martingale with zero expectation.

Concerning term $I_1$, the idea is to control $\omega$ by its $H^{-\alpha}$ norm and $K\ast \omega^1$ by its $L^p$ norm. We fix $\epsilon>0$ such that $\alpha+\epsilon<\min\{1-1/p,1/2\}$. We exploit Lemma \ref{lem:product_Sobolev} and get, using \eqref{eq:nabla G bound H},
\begin{align*}
    |I_1| &= |\langle \nabla G^\delta\ast \omega, (K\ast \omega^1)\omega\rangle|\\
    &\le \|\omega\|_{\dot{H}^{-\alpha-\epsilon}} \|(K\ast \omega^1)\nabla G^\delta\ast \omega\|_{\dot{H}^{\alpha+\epsilon}}\\
    &\lesssim \|\omega\|_{\dot{H}^{-\alpha-\epsilon}} \|K\ast \omega^1\|_{\dot{H}^{2(\alpha+\epsilon)}} \|\nabla G^\delta\ast \omega\|_{\dot{H}^{1-\alpha-\epsilon}}\\
    &\lesssim \|\omega\|_{\dot{H}^{-\alpha-\epsilon}} \|\omega^1\|_{\dot{H}^{2(\alpha+\epsilon)-1}} \|\omega\|_{\dot{H}^{-\alpha-\epsilon}} = \|\omega^1\|_{\dot{H}^{2(\alpha+\epsilon)-1}} \|\omega\|_{\dot{H}^{-\alpha-\epsilon}}^2.
\end{align*}
Taking $1/\tilde{p}=1-\alpha-\epsilon$ (so $1<\tilde{p}<p$), thanks to the Sobolev embedding $L^{\tilde{p}}\subseteq \dot{H}^{2(\alpha+\epsilon)-1}$ (e.g. \cite[Corollary 1.39]{bahouri2011fourier}), we have
\begin{align*}
|I_1|\lesssim \|\omega^1\|_{L^{\tilde{p}}} \|\omega\|_{\dot{H}^{-\alpha-\epsilon}}^2.
\end{align*}
By assumption, the $L^p$ norm and the $L^1$ norm of $\omega^1$ are bounded uniformly on $[0,T]\times\Omega$, hence there exists a constant $C_{\omega^1}>0$, such that (calling $\theta:=1-p^{\prime}/\tilde{p}^{\prime}$), $dt\otimes P$-a.s., 
\begin{align}
\label{eq: Lp initial condition}
\|\omega^1_t\|_{L^{\tilde{p}}} \le  \|\omega^1_t\|_{L^1}^\theta\|\omega^1_t\|_{L^p}^{1-\theta} \le \|\omega^1_t\|_{L^1}+\|\omega^1_t\|_{L^p} \le C_{\omega^1},
\end{align}
Thanks to the interpolation result \cite[Theorem 2.80]{bahouri2011fourier}, we have for $\bar{\epsilon} > 0$ to be determined later
\begin{align*}
    \|\omega\|_{\dot{H}^{-\alpha-\epsilon}}^2 &\le \|\omega\|_{\dot{H}^{-1}}^2 +\|\omega\|_{H^{-\alpha-\epsilon}}^2\\
    &\lesssim \|\omega\|_{\dot{H}^{-1}}^2 +\|\omega\|_{H^{-\alpha}}^{2-2\epsilon/(1-\alpha)}\|\omega\|_{H^{-1}}^{2\epsilon/(1-\alpha)}\\
    &\le \|\omega\|_{\dot{H}^{-1}}^2 +\bar{\epsilon}\|\omega\|_{H^{-\alpha}}^2 +C_{\bar{\epsilon}}\|\omega\|_{H^{-1}}^2,
\end{align*}
where $C_{\bar{\epsilon}} \approx_\epsilon \bar{\epsilon}^{1-\frac{1-\alpha}{\epsilon}}$. Taking expectation and integrating in time, we get, for some $\bar{c}>0$ independent of $\bar{\epsilon}$,
\begin{equation}
\label{eq: bound A1}
    \int_0^t\E[|I_1|] \rd r\le C_{\omega_0^1} (1+C_{\bar{\epsilon}}) C\int_0^t\E[\|\omega_r\|_{\dot{H}^{-1}}^2] \rd r + C_{\omega_0^1} \tilde{c} \bar{\epsilon} \int_0^t\E[\|\omega_r\|_{H^{-\alpha}}^2] \rd r.
\end{equation}

Concerning term $I_2$, the main point is that
\begin{align*}
    \langle \nabla G\ast \omega, (K\ast \omega)\omega^2 \rangle = 0,
\end{align*}
so it is enough to control the remainder term with $G-G^\delta$. We have, again by Lemma \ref{lem:product_Sobolev},
\begin{align*}
    |I_2|&= |\langle \nabla (G^\delta-G)\ast \omega, (K\ast \omega)\omega^2\rangle|\\
    &\le \|\omega^2\|_{\dot{H}^{-\alpha-\epsilon}} \|(K\ast \omega) \nabla (G^\delta-G)\ast \omega\|_{\dot{H}^{\alpha+\epsilon}}\\
    &\lesssim \|\omega^2\|_{\dot{H}^{-\alpha-\epsilon}} \|K\ast \omega\|_{\dot{H}^{2(\alpha+\epsilon)}} \|\nabla (G^\delta-G)\ast \omega\|_{\dot{H}^{1-\alpha-\epsilon}}\\
    &\le \|\omega^2\|_{\dot{H}^{-\alpha-\epsilon}} \|\omega\|_{\dot{H}^{2(\alpha+\epsilon)-1}} \|\nabla (G^\delta-G)\ast \omega\|_{\dot{H}^{1-\alpha-\epsilon}}.
\end{align*}
Using again \cite[Corollary 1.39]{bahouri2011fourier} and the $L^{\tilde p}$ bounds \eqref{eq: Lp initial condition}, with $1/\tilde{p} = 1-\alpha-\epsilon$, we get
\begin{align*}
    |I_2|&\le \|\omega^2\|_{\dot{H}^{-\alpha-\epsilon}}(\|\omega^1\|_{L^{\tilde{p}}}+\|\omega^2\|_{L^{\tilde{p}}}) \|\nabla (G^\delta-G)\ast \omega\|_{\dot{H}^{1-\alpha-\epsilon}}\\
    & \le \|\omega^2\|_{\dot{H}^{-\alpha-\epsilon}} (C_{\omega^1} + C_{\omega^2}) \|\nabla (G^\delta-G)\ast \omega\|_{\dot{H}^{1-\alpha-\epsilon}}.
\end{align*}
Taking expectation and integral in time, we get
\begin{align*}
    \int_0^T \E[|I_2|] \rd r & \lesssim (C_{w_0^1} + C_{\omega_0^2}) \left(\int_0^T \E[\|\omega^2_r\|_{\dot{H}^{-\alpha-\epsilon}}^2] \rd r\right)^{1/2}\\
    &\quad  \cdot \left(\int_0^T \E[\|\nabla (G^\delta-G)\ast \omega_r\|_{\dot{H}^{1-\alpha-\epsilon}}^2 ]\rd r\right)^{1/2}.
\end{align*}
Thanks again to the interpolation bounds \cite[Theorem 2.80]{bahouri2011fourier}, we arrive at
\begin{align*}
    \int_0^T \E[|I_2|] \rd r & \lesssim (C_{w_0^1} + C_{\omega_0^2}) \left(\int_0^T \E[\|\omega^2_r\|_{\dot{H}^{-1}}^2 +\|\omega^2_r\|_{H^{-\alpha}}^2] \rd r\right)^{1/2} \\
    &\quad \cdot \left(\int_0^T \E[\|\nabla (G^\delta-G)\ast \omega_r\|_{\dot{H}^{1-\alpha-\epsilon}}^2 ]\rd r\right)^{1/2}.
\end{align*}
To control the last term, we write, using \eqref{eq:G delta Fourier},
\begin{align*}
& \int_0^T \E[\|\nabla (G^\delta-G)\ast \omega_r^1\|_{\dot{H}^{1-\alpha-\epsilon}}^2 ]\rd r\\
& \quad = \int |n|^{-2(\alpha+\epsilon)}(1-e^{-\delta|n|^2}+e^{-|n|^2/\delta})^2 \int_0^T \E[|\widehat{\omega_r^1}(n)|^2] \rd r \rd n.
\end{align*}
Now the integral
\begin{align*}
    \int \int_0^T |n|^{-2(\alpha+\epsilon)} \E[|\widehat{\omega_r^1}(n)|^2] \rd r \rd n = \int_0^T \E[\|\omega_r^1\|_{\dot{H}^{-\alpha-\epsilon}}^2] \rd r
\end{align*}
is finite, therefore, by dominated convergence theorem,
\begin{align*}
    \int_0^T \E[\|\nabla (G^\delta-G)\ast \omega_r^1\|_{\dot{H}^{1-\alpha-\epsilon}}^2 ]\rd r \to 0 \text{ as }\delta \to 0
\end{align*}
and analogously for $\omega^2$. Hence
\begin{equation}
\label{eq: bound A2}
    \int_0^T \E[|I_2|] \rd r = o(1) \text{ as } \delta \to 0.
\end{equation}

As in Section \ref{sec:a priori bd}, the term $J$ in \eqref{eq: ito uniqueness} provides us with a control of the $H^{-\alpha}$ norm of the difference $\omega$. We split $J$ as in \eqref{eq:QD2G split}:
\begin{align}
    \tr\left[\left(Q(0)-Q(x)\right)D^2G^{\delta}(x)\right]
	&=  \tr\left[\left(Q(0)-Q(x)\right)D^2G(x)\right] \varphi(x)
    \nonumber\\
    & \quad - \tr\left[(Q(0)-Q(x))D^2(G-G^{\delta})(x)\right]\varphi(x) \nonumber\\
    & \quad + \tr[(Q(0)-Q(x))D^2G^{\delta}(x)](1-\varphi(x))
    \nonumber\\
    & =:A+R2+\tilde{R3}\label{eq:B split}
\end{align}
where $D^2G$ is understood as the pointwise second derivative of $G$ and $\varphi(x)=\varphi(|x|)$ is a radial $C^\infty$ function satisfying $0\le \varphi\le 1$ everywhere, $\varphi(x)=1$ for $|x|\le 1$ and $\varphi(x)=0$ for $|x|\ge 2$. 
Now we proceed as in the existence part. In particular, we exploit Lemmas \ref{lem:main bound A}, \ref{lem:bound R1 R2} and \ref{lem:bound R3}, where in the latter we can replace $Q^\delta$ by $Q$ (with the same proof) and we obtain, for some $c>0$ independent of $\delta$,
\begin{align*}
J&\le \int (\widehat{A}+\widehat{\tilde{R3}}(n))|\widehat{\omega}(n)|^2 \rd n +\iint |R2(x-y)| |\omega(x)||\omega(y)| \rd x \rd y\\
&\le \int (-c\langle n\rangle^{-2\alpha}+ C|n|^{-2})|\widehat{\omega}(n)|^2 \rd n \\
& \quad +C\iint (\delta\varphi(x-y)+|x-y|^{-2+2\alpha}1_{|x-y|< (8\delta\log(1/\delta))^{1/2}}) |\omega(x)||\omega(y)| \rd x \rd y.
\end{align*}
Using H\"older and Young inequality for the last term, we get, taking $1/r = 2-2/(p\wedge 2)$,
\begin{align*}
J&\le -c\|\omega\|_{H^{-\alpha}}^2 +C\|\omega\|_{\dot{H}^{-1}}^2 +C\delta\|\omega\|_{L^1}^2 \\
& \quad +C\left(\int |x|^{(-2+2\alpha)r} 1_{|x|< (8\delta\log(1/\delta))^{1/2}}) dx\right)^{1/r}\|\omega\|_{L^{p\wedge 2}}^2\\
&\le -c\|\omega\|_{H^{-\alpha}}^2 +C\|\omega\|_{\dot{H}^{-1}}^2 +C\delta\|\omega\|_{L^1}^2 +o(1)\|\omega\|_{L^{p\wedge 2}}^2,
\end{align*}
where we have used that $\alpha>2/p-1$ and so $(-2+2\alpha)r>-2$.

By \eqref{eq: Lp initial condition} the $L^1$ norm and the $L^{p\wedge 2}$ norm of $\omega$ are bounded uniformly on $[0,T]\times\Omega$ by $C_{\omega^1}+C_{\omega^2}$.
Taking expectation and integrating in time, we get
\begin{equation}
\label{eq: bound B}
\int_0^t \E[J] \rd r \le -c\int_0^t\E[\|\omega_r\|_{H^{-\alpha}}^2] \rd r +C\int_0^t\E[\|\omega_r\|_{\dot{H}^{-1}}^2] \rd r +o(1),
\end{equation}
where $o(1)$ is uniform for $t$ in $[0,T]$.

Now we integrate in time and take expectation in \eqref{eq: ito uniqueness} and we let $\delta\to 0$ using the bounds \eqref{eq: bound A1}, \eqref{eq: bound A2}, \eqref{eq: bound B} (and the limit \eqref{eq: convergence g delta to H-1}). We can choose $\bar{\epsilon} > 0$ in \eqref{eq: bound A1} small enough so that $C_{\omega_0^1} \tilde{c}\bar{\epsilon} \le c/2$. We obtain
\begin{align*}
    \mathbb{E}[\|\omega_t\|_{\dot{H}^{-1}}^2]
    \leq 
    \|\omega_0\|_{\dot{H}^{-1}}^2
    -\frac{c}{2}\int_0^t\E[\|\omega_r\|_{H^{-\alpha}}^2] \rd r +\tilde{C}\int_0^t\E[\|\omega_r\|_{\dot{H}^{-1}}^2] \rd r,
\end{align*}
where $\tilde{C} \approx_{\epsilon} 1+C_{\omega_0^1}^{\frac{1-\alpha}{\epsilon}}$. By Gronwall inequality, we get
\begin{align*}
    \sup_{t\in [0,T]}\mathbb{E}[\|\omega_t\|_{\dot{H}^{-1}}^2] +\frac{c}{2}\int_0^T\E[\|\omega_r\|_{H^{-\alpha}}^2] \rd r
    \leq 
    \|\omega_0\|_{\dot{H}^{-1}}^2 e^{\tilde{C}T}.
\end{align*}
In particular, if $\omega^1_0=\omega^2_0$ (that is $\omega_0=0$), then $\omega_t=0$ $P$-a.s., for every $t$. By time continuity of $\omega^1$ and $\omega^2$ with values in $H^{-4}$, then $\omega^1=\omega^2$ $P$-a.s.. The proof of uniqueness is complete.
\end{proof}

\appendix

\section{Estimating the remainder}
In this section we want to find an estimate for the remainder term defined in formula \eqref{eq:def beta remainder}, namely, for $R\ge 0$,
\begin{align*}
    \mathrm{Rem}_f(R)
    & := 4 R^{2\alpha} \int_0^\infty \int_{0}^{1} \left[ ((Ru)^2 + \rho^2)^{-\frac{2+2\alpha}{2}} - \rho^{-(2+2\alpha)}\right]  \\
    &\quad \cdot\rho (1-\cos(2\pi\rho)) f(u) u^{2\alpha} (1-u^2)^{-1/2} \rd u \rd\rho
\end{align*}
where $f: [0,1] \to \R$ is a given Borel bounded function. We take $\varphi$ as defined below formula \eqref{eq:QD2G split}, namely $\varphi:\R^2\to \R$ is a $C^\infty$ function, depending only on $|x|$, with $0\le \varphi\le 1$ everywhere, $\varphi(x)=1$ for $|x|\le 1$ and $\varphi(x)=0$ for $|x|\ge 2$.

\begin{lemma}
    \label{lem: properties remainder}
    We have the following bound:
    \begin{equation*}
    \left|\mathrm{Rem}_f(R)\right| 
    \lesssim R^2,\quad \forall R\ge 0.
    \end{equation*}
    Moreover, for every $\epsilon>0$ the Fourier transform of $\mathrm{Rem}$ satisfies (with multiplicative constant possibly depending on $\epsilon$)
    \begin{align*}
    |\mathcal{F}(|\cdot|^{-2}\mathrm{Rem}_f(|\cdot|)\varphi)(n)|\lesssim \langle n \rangle^{-2+\epsilon},\quad \forall n\in\R^2.
    \end{align*}
\end{lemma}

\begin{proof}
We call
\begin{align*}
    &I(R) := I_\rho(R) := ((Ru)^2 + \rho^2)^{-\frac{2+2\alpha}{2}} - \rho^{-(2+2\alpha)},\\
    &H(R) := \int_0^\infty \int_{0}^{1} I_\rho(R)  \rho (1-\cos(2\pi\rho)) f(u) u^{2\alpha} (1-u^2)^{-1/2} \rd u \rd\rho = \frac{\mathrm{Rem}_f(R)}{4R^{2\alpha}}.
\end{align*}
In order to find both the estimate on $\mathrm{Rem}_f$ and the one on its Fourier transform, we will find suitable bounds on the derivatives of $H$. The first derivative of $I$ in $R$ is
\begin{align*}
    \frac{\rd I}{\rd R
    } (R) = -(2+2\alpha)\ Ru^2 \ ((Ru)^2+\rho^2)^{-(2+\alpha)}.
\end{align*}
We use this expression to estimate the first derivative of $H$ in $R>0$:
\begin{align*}
    \left| \frac{\rd H}{\rd R
    }(R) \right|
    &= (2+2\alpha) R
    \int_{0}^{\infty}\int_{0}^{1}
     ((Ru)^2+\rho^2)^{-(2+\alpha)}
     \rho (1-\cos(2\pi\rho)) f(u) u^{2\alpha + 2} (1-u^2)^{-\frac{1}{2}} \rd u \rd \rho
     \\
    & \lesssim R \int_{0}^{1} \int_{Ru}^{\infty}
    \rho^{-(4+2\alpha)} \rho (1-\cos(2\pi \rho)) \rd \rho \  f(u) u^{2\alpha + 2} (1-u^2)^{-\frac{1}{2}}  \rd u 
    \\
    & \quad + R^{-4-2\alpha} 
    \int_{0}^{1} \int_{0}^{Ru}
    \rho (1-\cos(2\pi \rho)) \rd \rho \  f(u) u^{-2} (1-u^2)^{-\frac{1}{2}}   \rd u.
\end{align*}
Using that $1-\cos(2\pi\rho) \leq (2\pi\rho)^2/2$, we have
\begin{align*}
    \left| \frac{\rd H}{\rd R
    } (R) \right|
    & \lesssim R \int_{0}^{1} \int_{Ru}^{\infty}
    \rho^{-(1+2\alpha)} \rd \rho \ f(u) u^{2\alpha + 2} (1-u^2)^{-\frac{1}{2}}  \rd u 
    \\
    & \quad + R^{-4-\alpha} 
    \int_{0}^{1} \int_{0}^{Ru}
    \rho^3 \rd \rho \ f(u) u^{-2} (1-u^2)^{-\frac{1}{2}}   \rd u
    \\
    & \approx R^{1-2\alpha}.
\end{align*}
From this bound and the explicit expression of $\frac{\rd H}{\rd R}$, we see that $\frac{\rd H}{\rd R}$ is continuous on $(0,\infty)$ and locally integrable on $[0,\infty)$ as a function of $R$. Moreover we have $H(0) = 0$, so we get
\begin{align*}
    |H(R)| 
    \leq \int_{0}^R | \frac{\rd H}{\rd R
    } (r) |\rd r
    \lesssim R^{2-2\alpha}.
\end{align*}
We conclude the bound on $\mathrm{Rem}_f$
\begin{align*}
    |\mathrm{Rem}_f(R)| = 4R^{2\alpha}|H(R)|\lesssim R^2.
\end{align*}

We proceed by estimating the pointwise second derivative of $H$. Again at first we look at the second derivative of $I$, that is
\begin{align*}
    \frac{\rd^2 I}{\rd R^2}(R)
    = - (2+2\alpha) u^2 \left[((Ru)^2 + \rho^2)^{-(2+\alpha)}
    -(4+2\alpha) R^2 u^2 ((Ru)^2 + \rho^2)^{-(3+\alpha)}
    \right],
\end{align*}
for which we have the following estimates
\begin{align*}
    |\frac{\rd^2 I}{\rd R^2} (R)|
    & \lesssim
    u^2 \rho^{-(4+2\alpha)} + R^2u^4 \rho^{-(6+2\alpha)},\\
    |\frac{\rd^2 I}{\rd R^2} (R)|
    & \lesssim
    u^2 ((Ru)^{-(4+2\alpha)} + (Ru)^2(Ru)^{-(6+2\alpha)})
    = 2R^{-(4+2\alpha)} u^{-(2+2\alpha)}.
\end{align*}
We use the first bound for $\rho>Ru$ and the second one for $\rho\le Ru$, together again with the bound $1-\cos(2\pi\rho) \leq (2\pi\rho)^2/2$, and estimate the pointwise second derivative of $H$ in $R>0$:
\begin{align*}
   \left| \frac{\rd^2 H}{\rd R^2
    } (R) \right|
    & \lesssim \int_0^\infty \int_{0}^{1} 
    \left| \frac{\rd^2 I}{\rd R^2} (R) \right|
    \rho (1-\cos(2\pi\rho)) f(u) u^{2\alpha} (1-u^2)^{-\frac{1}{2}} \rd u \rd\rho\\
    & \lesssim  \int_{0}^{1} \int_{Ru}^\infty
     \rho^{-(4+2\alpha)} 
    \rho (1-\cos(2\pi \rho))\rd\rho \ f(u) u^{2+2\alpha} (1-u^2)^{-\frac{1}{2}} \rd u \\
    & \quad + R^2 \int_{0}^{1} \int_{Ru}^\infty
     \rho^{-(6+2\alpha)}
    \rho (1-\cos(2\pi\rho)) \rd\rho \ f(u) u^{4+2\alpha} (1-u^2)^{-\frac{1}{2}}  \rd u \\
    & \quad + R^{-4-2\alpha}  \int_{0}^{1} \int_0^{Ru}
    \rho (1-\cos(2\pi\rho)) \rd\rho \ f(u) u^{-2} (1-u^2)^{-\frac{1}{2}} \rd u \\
    & \approx R^{-2\alpha}.
\end{align*}

Now we come back to the estimate in Fourier modes of
\begin{align*}
|x|^{-2}\mathrm{Rem}_f(|x|)\varphi(x) = 4|x|^{-2+2\alpha}H(|x|)\varphi(x).
\end{align*}
By the previous bounds we have, for $0<|x|< 2$,
\begin{align*}
&|x|^{-2+2\alpha}|H(|x|)| \lesssim |x|^{-2+2\alpha}|x|^{2-2\alpha}=1,\\
&|\nabla [|\cdot|^{-2+2\alpha}H(|\cdot|)](x)| = \left|\left((-2+2\alpha)|x|^{-3+2\alpha}H(|x|) +|x|^{-2+2\alpha}\frac{\rd H}{\rd R}(|x|)\right)\frac{x}{|x|}\right| \lesssim |x|^{-1},\\
&|D^2 [|\cdot|^{-2+2\alpha}H(|\cdot|)](x)|\\
&= \left|\left((-2+2\alpha)(-4+2\alpha)|x|^{-4+2\alpha}H(|x|)+(-5+4\alpha)|x|^{-3+2\alpha}\frac{\rd H}{\rd R}(|x|) +|x|^{-2+2\alpha}\frac{\rd^2 H}{\rd R^2}(|x|)\right)\frac{xx^\top}{|x|^2}\right.\\
&\quad \left.+\left((-2+2\alpha)|x|^{-4+2\alpha}H(|x|) +|x|^{-3+2\alpha}\frac{\rd H}{\rd R}(|x|)\right)I\right|\\
&\lesssim |x|^{-2},
\end{align*}
where $\nabla$ and $D^2$ here denote the pointwise derivates in $x\neq 0$. Since $\varphi$ is $C^\infty$ and constant in a neighborhood of $x=0$, the same bounds hold for $|\cdot|^{-2+2\alpha}H(|\cdot|)\varphi$. 

By Lemma \ref{lem: second derivative remainder} below, we conclude that, for every $\epsilon>0$,
\begin{align*}
|\mathcal{F}(|\cdot|^{-2}\mathrm{Rem}_f(|\cdot|)\varphi)(n)|\lesssim \langle n \rangle^{-2+\epsilon}.
\end{align*}
The proof is complete.
\end{proof}

We have used the following:
\begin{lemma}
\label{lem: second derivative remainder}
Let $g:\R^2\to \R$ be a Borel function, which is $C^2$ on $\R^2\setminus\{0\}$ and has support in $\bar{B}_2(0)$, and assume that
\begin{align*}
|g(x)|\lesssim 1,\quad |\nabla g(x)|\lesssim |x|^{-1},\quad |D^2g(x)|\lesssim |x|^{-2}, \quad \forall x\neq 0.
\end{align*}
Then for every $\epsilon>0$ we have (with multiplicative constant depending on $\epsilon$)
\begin{align*}
|\widehat{g}(n)|\lesssim \langle n \rangle^{-2+\epsilon},\quad \forall n\in\R^2.
\end{align*}
\end{lemma}

\begin{proof}
    We start noting that, due to the smoothness of $g$ outside $x=0$ and the estimates on $g$ and $\nabla g$, the pointwise derivative $\nabla g$ is actually the distributional derivative of $g$
    First we bound the $W^{1+\gamma,p}$ norm of $g$ for any $0<\gamma<1$ and $1<p<2/(1+\gamma)$. Here the $W^{1+\gamma,p}$ norm is defined as
    \begin{align*}
        \|g\|_{W^{1+\gamma,p}}^p = \|g\|_{L^p}^p+\|\nabla g\|_{L^p}^p +\iint_{\R^2\times\R^2} \frac{|\nabla g(x)-\nabla g(y)|^p}{|x-y|^{2+\gamma p}} \rd x\rd y
    \end{align*}
    Clearly the $L^p$ norms of $g$ and $\nabla g$ are finite. For the last term, on one hand we exploit the bound $\nabla g$ and get, for every $x\neq0,y\neq 0$,
    \begin{align}\label{eq:bound_Dg_1}
        \frac{|\nabla g(x)-\nabla g(y)|^p}{|x-y|^{2+\gamma p}}
        \lesssim \frac{1}{(|x|^p\wedge |y|^p)|x-y|^{2+\gamma p}}
    \end{align}
    On the other hand, we take a continuous path $\eta:[0,1]\to \R^2$ with $\eta(0)=y$, $\eta(1)=x$, $\min_{s\in [0,1]}|\eta(s)|\gtrsim |x|\wedge |y|$ and $\max_{s\in [0,1]}|\eta^\prime(s)|\lesssim |x-y|$; we exploit the bound on $D^2 g$ and get:
    \begin{align}
        \frac{|\nabla g(x)-\nabla g(y)|^p}{|x-y|^{2+\gamma p}}
        &\le \left(\int_0^1 |D^2g(\eta(s))||\eta^\prime(s)|\rd s \right)^p \cdot \frac{1}{|x-y|^{2+\gamma p}} \nonumber\\
        &\lesssim \frac{\max_s|\eta^\prime(s)|^p}{\min_s |\eta(s)|^{2p}} \cdot \frac{1}{|x-y|^{2+\gamma p}} \nonumber\\
        &\lesssim \frac{1}{(|x|^{2p}\wedge |y|^{2p})|x-y|^{2-(1-\gamma)p}}.\label{eq:bound_Dg_2}
    \end{align}
    Interpolating between \eqref{eq:bound_Dg_1} and \eqref{eq:bound_Dg_2}, we get, for any $0<\beta<1$,
    \begin{align*}
        \frac{|\nabla g(x)-\nabla g(y)|^p}{|x-y|^{2+\gamma p}} &\lesssim \left(\frac{1}{(|x|^p\wedge |y|^p)|x-y|^{2+\gamma p}}\right)^\beta \left(\frac{1}{(|x|^{2p}\wedge |y|^{2p})|x-y|^{2-(1-\gamma)p}}\right)^{1-\beta}\\
        &\le \left(\frac{1}{|x|^{2p-\beta p}}+\frac{1}{|y|^{2p-\beta p}}\right)\cdot \frac{1}{|x-y|^{2+(\beta+\gamma-1)p}}.
    \end{align*}
    Now, since $1<p<2/(1+\gamma)$, we can find $0<\beta<1$ such that $2p-\beta p <2$ and $2+(\beta+\gamma-1)p < 2$, making the right-hand side above integrable in $x$ and $y$ over $B_2(0)\times B_2(0)$. Hence $g$ has finite $W^{1+\gamma,p}$ norm for $0<\gamma<1$, $1<p<2/(1+\gamma)$.
    
    As a consequence of \cite[Subsection 2.5.7 formulas (5) and (9), Subsection 2.5.6 formula (2), Subsection 2.3.2 formula (9)]{triebel1983function} (see also \cite[Section 24]{van2022theory}), we have (with multiplicative constant possibly depending on $0<\gamma<1$ and $1<p<2/(1+\gamma)$)
    \begin{align*}
        \|g\|_{H^{1+\gamma,p}}:= \|\mathcal{F}^{-1}(\lan \cdot \ran^{1+\gamma}\hat{g})\|_{L^p} \lesssim \|g\|_{W^{1+\gamma,p}} <\infty
    \end{align*}
    and so, since $\mathcal{F}$ is bounded from $L^p$ to $L^{p'}$,
    \begin{align*}
        \lan \cdot \ran^{1+\gamma}\hat{g} \in L^{p'}.
    \end{align*}

    To conclude, we would like show that $\lan \cdot \ran^{1+\gamma}\hat{g}$ is actually bounded (and then we take $2-\epsilon=1+\gamma$). To do so, we will use the compact support of $g$, through Sobolev embedding for $\hat{g}$. Note that, as $g$ has compact support, $\hat{g}$ is $C^\infty$ and
    \begin{align}
        |\nabla (\lan\cdot\ran^{1+\gamma}\hat{g})(n)| &\lesssim \lan n \ran^{1+\gamma} |\nabla \hat{g}(n)| +\lan n\ran^\gamma|\hat{g}(n)| \nonumber\\
        &\lesssim \lan n\ran^2|\nabla \hat{g}(n)| +\lan n\ran|\hat{g}(n)| \label{eq:bound_g_hat_moment}
    \end{align}
    Concerning the first addend of \eqref{eq:bound_g_hat_moment}, we note that
    \begin{align*}
        \lan n\ran^2|\nabla \hat{g}(n)| = |\mathcal{F}[(I-(2\pi)^{-2}\Delta)[2\pi xg(x)]](n)|.
    \end{align*}
    By the assumptions on $g$ and its derivatives, for $1<p<2/(1+\gamma)$, $xg(x)$ is in $L^p$ and, for every $x\neq 0$, the pointwise derivative $\Delta(xg(x))$ satisfies
    \begin{align*}
        |\Delta(xg(x))|\lesssim |\nabla g(x)|+ |x||D^2g(x)|\lesssim |x|^{-1}.
    \end{align*}
    Hence the pointwise derivative $\Delta(xg(x))$ is actually a distributional derivative
    and is in $L^p$. Therefore, $\lan n\ran^2|\nabla \hat{g}(n)|$ is in $L^{p'}$. Similarly, concerning the second addend of \eqref{eq:bound_g_hat_moment}, we have
    \begin{align*}
        \lan n\ran|\hat{g}(n)| \le |\hat{g}(n)| +|n\hat{g}(n)| = |\hat{g}(n)|+(2\pi)^{-1}|\mathcal{F}(\nabla g)(n)|
    \end{align*}
    and, since $g$ and $\nabla g$ are in $L^p$, $\lan n\ran|\hat{g}(n)|$ is in $L^{p'}$. We conclude that, for $0<\gamma<1$, for any $1<p<2/(1+\gamma)$, $\lan \cdot\ran^{1+\gamma}\hat{g}$ is in $W^{1,p'}$ and so, by the Sobolev embedding on $\R^2$, $\lan \cdot\ran^{1+\gamma}\hat{g}$ is bounded. The proof is complete.
\end{proof}

\section{Technical lemmas}

\begin{lemma}\label{lem:indicator_Fourier}

Let $\varphi(x)=\varphi(|x|)$ be a radial $C^\infty$ function satisfying $0\le \varphi\le 1$ everywhere,  $\varphi(0)=1$ and $\varphi(x)=0$ for $|x|\ge 2$. For $M>0$ we define $\varphi_M := \varphi(x/M)$. Then, for $0<\alpha<1$, for every $n\neq 0$,
    \begin{align*}
    \widehat{G_{\alpha}\varphi_M}(n) \to \widehat{G_{\alpha}}(n) \quad \text{as }M\to \infty.
    \end{align*}
\end{lemma}

\begin{proof}
       \textbf{Step 1}: we show that the family $(\widehat{\varphi_M})_M$ satisfies, for every $a>0$,
       \begin{align}
            \sup_{M\ge 1}\sup_{|n|\ge a} |\widehat{\varphi_M}(n)|<\infty\label{eq:boundedness of varphi M}
       \end{align}
       and, for every $f\in C_0(\R^2)$, the space of continuous functions vanishing at infinity,
       \begin{align}
       \label{eq: measure convergence of phi M}
            \langle\widehat{\varphi}_M,f\rangle \to f(0),
            \qquad
            \mbox{as } M\to \infty
       \end{align}
       that is the family of measures $(\widehat{\varphi_M})_{M>0}$ converges weakly-* to the measure $\delta_0$ as $M\to\infty$.

       First we write the Fourier transform of $\varphi_M$ in terms of the Fourier transform of $\varphi$
        \begin{align*}
            \widehat{\varphi}_M(n) 
            = \int \varphi(\frac{y}{M}) e^{2\pi iy\cdot n} \rd y
            = M^2 \int \varphi(x) e^{2\pi iM x\cdot n} \rd x
            = M^2 \widehat{\varphi}(Mn).
        \end{align*}
        Since $\widehat{\varphi}$ is a Schwartz function, we have
        \begin{align*}
            \sup_{M\ge 1}\sup_{|n|\ge a} M^2 |\widehat{\varphi}(Mn)| \le  \frac{1}{a^2}  \sup_{n^\prime} |n^\prime|^2 |\widehat{\varphi}(n^\prime)|<\infty, 
        \end{align*}
        showing \eqref{eq:boundedness of varphi M}. Now we show the convergence \eqref{eq: measure convergence of phi M} for $f$ in $C^\infty_c$: by the Plancherel isometry,
       \begin{align*}
           \langle \widehat{\varphi_M}, f \rangle = \langle \varphi_M, \widecheck{f} \rangle \overset{M\to\infty}{\to} \int \widecheck{f}(x) \rd x = f(0).
       \end{align*}
        Let us compute the $L^1$ norm of $\widehat{\varphi}_M$:
        \begin{align*}
            \int | \widehat{\varphi_M}(n)| \rd n
            = M^2 \int |\widehat{\varphi}(nM)| \rd n 
            = M^2 \int |\widehat{\varphi}(n)| \frac{\rd n}{M^2}
            = \| \widehat{\varphi} \|_{L^1}
        \end{align*}
        Since $(\widehat{\varphi_M})_M$ converges to $\delta_0$ as a tempered distribution and it is uniformly bounded in $L^1$, it also converges weakly-* as a bounded family of measures, proving \eqref{eq: measure convergence of phi M}.

        \textbf{Step 2:} conclusion. Fix $n\neq 0$. We note that $ \widehat{G_{\alpha}\varphi_M}(n) = \langle \widehat{G_{\alpha}}(n-\cdot), \widehat{\varphi_M}\rangle$, so we need to prove
        \begin{align*}
            \langle\widehat{G_{\alpha}}(n-\cdot),\widehat{\varphi_M}\rangle 
            \to \langle\widehat{G_{\alpha}}(n - \cdot), \delta_0\rangle
            = \widehat{G_{\alpha}}(n).
        \end{align*}
        Note that $\widehat{G_\alpha}(n-\cdot)$ is in $L^1_{loc}(\R^2)$, it is continuous on $B_\delta(n)^c$ for every $\delta>0$ and it vanishes at infinity (by the explicit formula for $\widehat{G_\alpha}$ in \ref{lem:sobolev riesz potential}). Therefore, for every $\epsilon>0$ we can find $0<\delta<|n|/2$ and a $C_0(\R^2)$ function $f$, coinciding with $\widehat{G_\alpha}(n-\cdot)$ on $B_\delta(n)^c$, in particular on $B_{|n|/2}(0)$, and with $\|\widehat{G_\alpha}(n-\cdot)-f\|_{L^1}<\epsilon$. In particular, by the uniform bound \eqref{eq:boundedness of varphi M}, we have for some $C>0$ (possibly depending on $n$),
        \begin{align*}
            \sup_{M\ge 1}|\langle \widehat{G_\alpha}(n-\cdot)-f , \widehat{\varphi_M} \rangle| &\le \sup_{M\ge 1}\int_{B_\delta(n)}|\widehat{G_\alpha}(n-n^\prime)-f(n^\prime)||\widehat{\varphi_M}(n^\prime)|\rd n^\prime\\
            &\le \|\widehat{G_\alpha}(n-\cdot)-f\|_{L^1} \sup_{M\ge 1}\sup_{|n^\prime|\ge|n|/2}|\widehat{\varphi_M}(n^\prime)|\le C\epsilon.
        \end{align*}
        By the weak-* convergence of $\widehat{\varphi_M}$ to $\delta_0$, we can then take $M_0$ such that, for every $M> M_0$, $|\langle f , \widehat{\varphi_M} \rangle -\widehat{G_\alpha}(n)| =|\langle f , \widehat{\varphi_M} \rangle -f(0)| < \epsilon$. Therefore we have, for $M>M_0$,
        \begin{align*}
            |\langle \widehat{G_\alpha}(n-\cdot), \widehat{\varphi_M} \rangle -\widehat{G_\alpha}(n)| \le (C+1)\epsilon.
        \end{align*}
        Hence $\widehat{G_{\alpha}\varphi_M}(n) \to \widehat{G_\alpha}(n)$ as $M\to \infty$. The proof is complete.
\end{proof}

\subsection{Lowersemicontinuity of norms}

In order to prove lower-semicontinuity of norms, we prove that the level sets are closed.

\begin{lemma}
\label{lem: semicontinuity}
We have the following
\begin{enumerate}[(i)]
    \item 
    Let $V\subset B$ be Banach spaces with continuous embedding, such that $V$ is reflexive. Then the norm of $V$ is lower-semicontinuous in $B$.

    \item The $C^{\gamma}([0,T];H^{-3})$-norm is lower-semicontinuous in $L^2([0,T];L^2(\R^2;w))$.

    \item The $L^{\infty}([0,T];L^2)$-norm is lower-semicontinuous in $L^2([0,T];L^2(\R^2;w))$.
\end{enumerate}
\end{lemma}

\begin{proof}
\begin{enumerate}[(i)]
    \item 
    We prove that the set $L:=\{x\in B \mid \|x\|_{V} \le 1\}$ is closed. Take a sequence $(x^n)_n$ in $L$ that converges strongly to $x$ in $B$, hence it also converges weakly in $B$. The sequence is also bounded in $V$, which is reflexive, so it admits a subsequence $(x^{n_k})_k$ that converges weakly in $V$ to some $\bar{x}$ with $\|\bar{x}\|_{V} \leq 1$. But since $V$ continuously embeds in $B$, we have that $x^{n_k} \rightharpoonup\bar{x}$ in $B$. By uniqueness of the weak limit in $B$, we have that $x =\bar{x} \in L$.
    \item Let $u^n \to u$ in $L^2_t(L^2_{x,w})$ such that $\|u^n\|_{C^{\gamma}_tH^{-3}}\leq 1$. We can thus find a (non relabelled) susequence such that $u^n_t \to u_t$ in $L^2_{x,w}$ $\rd t$-a.s. On the other hand, $(u_t^n)_n$ is bounded in $H^{-3}$, which is reflexive, so it converges weakly in $H^{-3}$ (up to a subsequence) to $v_t \in H^{-3}$. Let $B_R$ be the ball of radius $R>0$ in $\R^2$. We have that the restriction $u_t|_{B_R}$ is the $L^2$-weak limit of $u^n_t|_{B_R}$ and $v_t|_{B_R}$ is the $H^{-3}$-weak limit of the same sequence. Since $L^2(B_R)$ continuously embeds in $H^{-3}(B_R)$ we have that $u_t|_{B_R} = v_t|_{B_R}$ for every $R>0$. Since $u_t$ and $v_t$ coincide an any ball, they are the same. We showed that $u_t$ is the weak limit in $H^{-3}$ of $u^n_t$. We have for every $f\in (H^{-3})^{\ast}$, 
    \begin{align*}
    |f(u_t-u_s)| 
    &\leq |f(u^n_t)-f(u_t)| 
    + |f(u^n_s) - f(u_s)|
    + |f(u^n_t)-f(u^n_s)|\\
    &\leq |f(u^n_t)-f(u_t)| 
    + |f(u^n_s) - f(u_s)|
    + \|f\|_{(H^{-3})^{\ast}} |t-s|^{\gamma}\\
    &\to  \|f\|_{(H^{-3})^{\ast}} |t-s|^{\gamma},
    \qquad \mbox{as } n\to \infty.
    \end{align*}
    Since $\|u_t-u_s\|_{H^{-3}} = \sup_{\|f\|_{(H^{-3})^{\ast}}\leq 1}|f(u_t-u_s)|\leq |t-s|^{\gamma}$, we have that $\|u\|_{C^{\gamma}H^{-3}}\leq 1$.

    \item 
    We prove that $L=\{u\in L^2_t(L^2_{x,w}) \mid \operatorname{esssup}_{t\in[0,t]}\|u_t\|_{L^2} \leq 1\}$ is closed. Take $u^n \in L$ such that $u^n$ converges to $u$ in $L^2_t(L^2_{x,\omega})$. Then, up to a subsequence, $\|u^n_t(x)-u_t(x)\|_{L^2_{x,w}} \to 0$, $\rd t$-a.s. For a fixed $t$ in a set of full measure, the sequence $(u^n_t)_n$ converges in $L^2_{x,w}$ and is bounded in $L^2$ by $1$. Since $L^2_{x,w}$ is reflexive and continuously embedded in $L^2$ we can apply (i) and we have that limit $u_t$ is also bounded by $1$ in $L^2$. Hence, $\operatorname{esssup}_{t\in[0,T]}\|u_t\|_{L^2}\leq 1$ and $u\in L$. We established that $u_t$ is the weak limit it $H^{-3}$ of $u_t^n$.
    \end{enumerate}
\end{proof}

\section{Regularized equation}
\label{sec: proof of regular existence}

In this section we give a proof to Lemma \ref{lem: l infty and l 1}. We will closely follow the proof of \cite[Theorem 13]{coghi2020regularized} (which is a refinement of \cite{coghi_flandoli}), where the result is proven for the $2$-dimensional torus $\mathbb{T}^2$. We will limit the discussion to the points that needs to be tweaked in order to adapt the proof to $\R^2$ as most of the computation don't change between the two cases.

We call $\mathcal{M}(\R^2)$ the space of signed measures on $\R^2$ endowed with the total variation norm. Notice that $\mathcal{M}(\R^2)$ is the dual of $C_0(\R^2)$, the space of continuous functions vanishing at infinity.
We will denote by $\mathcal{M}_M(\R^2)$ the elements of $\mathcal{M}(\R^2)$ with total variation smaller than $M$.

Moreover, we call $\operatorname{Lip}_c(\R^2)$ the space of Lipschitz functions on $\R^2$ with compact support, for $\varphi \in \operatorname{Lip}_c(\R^2)$, we call $\operatorname{Lip}(\varphi)$ its Lipschitz constant. We will denote the unit ball in $\operatorname{Lip}_c(\R^d)$ as
\begin{equation*}
    \operatorname{Lip}_{1} =
    \{
        \varphi \in \operatorname{Lip}_c(\R^2) \mid \|\varphi\|_{\infty} + \operatorname{Lip}(\varphi) \leq 1
    \}.
\end{equation*}
We endow the space $\mathcal{M}_M(\R^2)$ with the distance
\begin{equation*}
    W_1(\mu,\nu) = \sup_{\varphi \in \operatorname{Lip}_1} |\mu(\varphi) - \nu(\varphi)|,
    \qquad
    \mu,\nu \in \mathcal{M}(\R^2),
\end{equation*}
Since the space $\operatorname{Lip}_c(\R^2)$ is dense in $C_0(\R^2)$, the space $(\mathcal{M}_M(\R^2),W_1)$ is complete. 

Finally, for $p\geq 2$ and $M>0$ we define the space $V_M^{p,T} := L^p(\Omega; C([0,T], \mathcal{M}_M(\R^2)))$, which we endow with the distance
\begin{equation*}
    d_c(\mu,\nu) := \mathbb{E}[\sup_{t\in[0,T]} e^{-ct}W_1(\mu_t,\nu_t)^p]^{\frac{1}{p}},
    \qquad
    \mu,\nu \in V_M^{p,T}.
\end{equation*}
for a given $c>0$ to be determined later.

This space is complete, see \cite[Remark 2.2]{coghi2020regularized}.

As a first step we introduce the following stochastic differential equation (SDE), for a fixed $\mu \in V_M^{p,T}$,
\begin{equation}
    \label{eq: regular SDE}
    \rd X_t = (K^{\delta}\ast \mu)(X_t) \rd t + \sum_{k} \sigma_k^{\delta}(X_t) \rd W^k_t,
    \qquad
    X_0 = x \in \R^2.
\end{equation}
Equation \eqref{eq: regular SDE} admits a pathwise unique strong solution as its coefficients are uniformly Lipschitz continuous. Moreover, the solution adimts a version, which we denote by $\Phi^{\mu}(t,x)$, which is Lipschitz continuous in the initial datum $x\in \R^2$, uniformly in $t\in[0,T]$, and H\"older continuous in time $t\in[0,t]$, uniformly in $x\in \R^2$.

We fix $\omega^{\delta}_0$ satisfying \eqref{eq: omega zero wish list}. We have that $\omega^{\delta}_0 \in \mathcal{M}_M(\R^2)$ for $M= \|\omega^\delta_0\|_{L^1}$.

We define the operator
\begin{equation*}
    \Psi: V_M^{p,T} \ni \mu \mapsto \Phi^{\mu}_{\#}\omega_0^{\delta} \in V_M^{p,T}.
\end{equation*}

From now on we take $p>2$. For fixed $T>0$ and $M>0$, it is proved in \cite[Lemma 3.2]{coghi2020regularized} that there exists $c>0$ such that the operator $\Psi$ is a contraction on $(V_M^{p,T}, d_c)$. Notice that the proof of \cite[Lemma 3.2]{coghi2020regularized} works also on $\R^2$ (whereas in the original lemma it is done in $\mathbb{T}^2$). \cite[Lemma 3.2]{coghi2020regularized} relies on two crucial ingredients. The first one is that $K^{\delta}$ is uniformly Lipschitz and bounded by the definition of $G^\delta$. The second ingredient is that there exists a constant $C=C(\delta) > 0$ such that
\begin{equation*}
    \sum_k |\sigma^{\delta}_k (x) - \sigma^{\delta}_k(y) |^2
    \leq C |x-y|^2,
\end{equation*}
which follows from Lemma \ref{lem:Q_reguralized}.

Since $(V_M^{p,T}, d_c)$ is complete, the operator $\Psi$ admits a unique fixed point $\omega^{\delta} \in V_M^{p,T}$. 

A straight-forward application of It\^o's formula shows that $\omega^{\delta}$ is a strong solution to equation \eqref{eq: regularized main} in the sense of distributions.

We can compute now the $L^1$ and $L^{\infty}$ norms of $\omega^{\delta}$. Notice that, by construction
\begin{equation}
\label{eq: Jacobian expressino}
    \omega^{\delta} = \Phi^{\omega_\delta}_{\#}\omega^{\delta}_0,
    \qquad
    \iff
    \qquad
    \omega^{\delta}_t(x) \rd x
    = \omega^{\delta}_0((\Phi^{\omega^{\delta}})^{-1}(x)) |\operatorname{det}D\Phi^{\omega^\delta}(x)|\rd x,
\end{equation}
where the last equality is to be intended in the space $\mathcal{M}(\R^2)$ for every $t\in[0,T]$, $P$-a.s.

By our assumptions on $\sigma^{\delta}_k$, the It\^o-Stratonovich correction in equation \eqref{eq: regular SDE} vanishes, see \cite[Remark 2.4]{Brzezniak_Flandoli_Maurelli} for more details. Since equation \eqref{eq: regular SDE} coincides with its Stratonovich version and the coefficients are divergence free, the corresponding flow is measure preserving. Equivalently, the Jacobian in \eqref{eq: Jacobian expressino} is constantly equal to $1$. As a consequence we have that $P$-a.s.
\begin{equation*}
  \sup_{t\in[0,T]}\|\omd_t\|_{L^{\infty}} \leq \| \omd_0\|_{L^{\infty}},
    \qquad
    \sup_{t\in[0,T]}\|\omd_t\|_{L^{1}} 
    \leq \| \omd_0\|_{L^{1}}.
\end{equation*}

From the previous bounds we have that $
\omd \in L^{\infty}(\Omega; L^2([0,T]\times \R^2))$. So that every term in equation \eqref{eq: regular SDE} is well-defined in $H^{-4}$. By passing to the limit in the weak formulation of the equation (for more details see Step 4 in the proof of existence) we can conclude that $\omd$ is an analytically strong solution to \eqref{eq: regular SDE} in $H^{-4}$. Moreover, there exists a version $\omd \in C([0,T];H^{-4})$ since each integral in equation \eqref{eq: regular SDE} admits a continuous version.

We must now prove that $\omd \in L^2([0,T],\dot{H}^{-1})$, $P$-a.s.. In order to do so, we prove the following preliminary result.

\begin{lemma}
For every fixed $m\ge 2$ we have (with multiplicative constant depending on $m$)
\label{lem: moments}
	\begin{align*}
	&\mathbb{E}\left[\sup_{t\in[0,T]} \int_{\mathbb{R}^2}|x|^m|\omega^{\delta}_t(x)| \rd x\right] 
	\lesssim \int_{\mathbb{R}^2} |x|^m |\omega_0^{\delta}(x)|\rd  x 
	+ (\| \omega^{\delta}_0 \|_{L^{\infty}}^m + \| \omega^{\delta}_0 \|_{L^{1}}^m + 1) \| \omega^{\delta}_0\|_{L^{1}},\\
    \end{align*}
    and 
    \begin{align*}
        \mathbb{E}\left[\sup_{t\in[0,T]} \left(\int_{\mathbb{R}^2}|x|^{m/2}|\omega^{\delta}_t(x)| \rd x\right)^2\right] 
	&\lesssim \|\omd_0\|_{L^{1}}\int_{\mathbb{R}^2} |x|^m |\omega_0^{\delta}(x)|\rd  x \\
	&\quad + (\| \omega^{\delta}_0 \|_{L^{\infty}}^m + \| \omega^{\delta}_0 \|_{L^{1}}^m + 1) \| \omega^{\delta}_0\|_{L^{1}}^2
	\end{align*}
\end{lemma}

\begin{proof}
	Notice $|\omega^{\delta}_t| = |\omega^{\delta}_0|((\Phi^{\omd}_t)^{-1})$. Since $X_t$ is measure preserving it follows that $|\omega^{\delta}_t| = (\Phi^{\omd}_t)_{\#}|\omega_0^{\delta}|$. By standard estimates via Burkholder-Davis-Gundy inequality, we have (with multiplicative constant depending on $m$ and $T$ but not on $\delta$)
    \begin{align*}
        \E[\sup_{t\in[0,T]}|\Phi^{\omd}_t(x)|^m] \lesssim |x|^m +\|K^\delta\ast \omd\|_{L^\infty(\Omega\times [0,T]\times \R^2)}^m +\tr[Q^\delta(0)]^{m/2}.
    \end{align*}
    Since $K^\delta(x)1_{|x|\le 1}$ is in $L^1$ and $K^\delta(x)1_{|x|> 1}$ is in $L^\infty$, we have
    \begin{align*}
        \|K^\delta\ast \omd\|_{L^\infty(\Omega\times [0,T]\times \R^2)} &\lesssim \|\omd\|_{L^\infty(\Omega\times [0,T]\times \R^2)}+\|\omd\|_{L^\infty(\Omega\times [0,T];L^1(\R^2))}\\
        &\lesssim \|\omd_0\|_{L^\infty}+\|\omd_0\|_{L^1}.
    \end{align*}
    Moreover, $Q^\delta(0)\lesssim 1$ by \eqref{eq:Q delta 0}. Hence
    \begin{align*}
         \E[\sup_{t\in[0,T]}|\Phi^{\omd}_t(x)|^m] \lesssim |x|^m +\|\omd_0\|_{L^\infty}^m+\|\omd_0\|_{L^1}^m +1.
    \end{align*}
    Integrating this expression in $\omd_0$, we get the first bound. The second bound follows from the first one by H\"older inequality.
\end{proof}

Notice that, for every $t\in[0,T]$, $P$-a.s.
\begin{equation*}
    \int_{\R^2} \omd_t(x)\rd x 
    = \int_{\R^2} \omd_0(0) \rd x
    = 0.
\end{equation*}
We can apply Lemma \ref{lem: properties initial condition} to have that $\omd_t \in \dot{H}^{-1}$, for every $t\in[0,T]$, $P$-a.s. and
\begin{align*}
    \mathbb{E}\left[\sup_{t\in [0,T]}\|\omd_t\|_{\dot{H}^{-1}}^2\right]
    \lesssim
    \mathbb{E}\left[\sup_{t\in [0,T]}
    \left(
        \int_{\R^2}|x||\omd_t(x)|\rd x
    \right)^2
    \right]
    + \mathbb{E}\left[\sup_{t\in [0,T]}\|\omd_t\|_{L^2}^2\right]
    < \infty,
\end{align*}
where the last bound follows from Lemma \ref{lem: moments}. This shows that $\sup_{t\in [0,T]}\|\omd_t\|_{\dot{H}^{-1}}$ is in $L^2(\Omega)$.

\textbf{Uniqueness} Assume now that $\omd \in L^{\infty}([0,T];L^p(\Omega;L^2))$ is an analytically weak solution to equation \eqref{eq: regularized main}. The solution we built above belong to this class. We can freeze the nonlinearity in the equation and we see that $\omd$ also solves the linear equation
\begin{align}
\begin{aligned}\label{eq: linear regularized main}
    &\rd \omd + b\cdot\nabla\omd\rd t +\sum_k\sigma^\delta_k\cdot\nabla\omd \rd W^k = \frac{c_\delta}{2} \Delta \omd \rd t,\\
    &\omd_t|_{t=0} = \omd_0,
\end{aligned}
\end{align}
where $b_t(x) := (K^{\delta}\ast\omd_t)(x)$. By classical SPDE theory (e.g. \cite{liu2015stochastic}) equation \eqref{eq: linear regularized main} admits a unique weak solution. Since both $\omd$ and $\Phi^{\omd}_{\#}\omd_0$ are in $L^\infty([0,T];L^p(\Omega;L^2))$ and are weak solutions to \eqref{eq: linear regularized main}, they must coincide. Thus, every solution to \eqref{eq: regularized main} in the space $L^{\infty}([0,T];L^p(\Omega;L^2))$ must be a fixed point to the map $\psi$, which is unique.

This concludes the proof of Lemma \ref{lem: l infty and l 1}.

\bibliographystyle{abbrv}
\bibliography{bibliography}

\end{document}